\documentclass[11pt, twoside, leqno]{article}

\usepackage{amssymb}
\usepackage{amsmath}
\usepackage{mathrsfs}
\usepackage{amsthm}
\usepackage{amsfonts}
\usepackage{color}
\usepackage{latexsym}
\usepackage{txfonts}
\usepackage{indentfirst}

\allowdisplaybreaks

\def\rr{{\mathbb R}}
\def\rn{{\mathbb{R}^n}}

\def\nn{{\mathbb N}}
\def\zz{{\mathbb Z}}
\def\cc{{\mathbb C}}
\def\CG{{\mathcal G}}
\def\CL{{\mathcal L}}
\def\CS{{\mathcal S}}

\def\CM{{\mathcal M}}
\def\CA{{\mathcal M}}
\def\CX{{\mathcal X}}
\def\CY{{\mathcal Y}}
\def\CB{{\mathcal B}}
\def\CJ{{\mathcal J}}

\newcommand{\cw}{\mathrm{cw}}

\newcommand{\RZ}{\mathrm Z}

\newcommand{\CN}{\mathcal{N}}

\newcommand{\CQ}{\mathcal{Q}}
\newcommand{\OH}{\mathring{H}}

\def\fz{\infty }
\def\az{\alpha}
\def\bz{\beta}
\def\dz{\delta}
\def\ez{\epsilon}
\def\kz{\kappa}
\def\thz{\theta}
\def\vz{\varphi}

\def\lf{\left}
\def\r{\right}
\def\ls{\lesssim}
\def\noz{\nonumber}
\def\wz{\widetilde}
\def\gfz{\genfrac{}{}{0pt}{}}

\newcommand{\om}{\omega/(\omega+\eta)}

\def\loc{{\mathrm{loc}}}
\DeclareMathOperator{\supp}{supp}

\def\XXint#1#2#3{{\setbox0=\hbox{$#1{#2#3}{\int}$ }
\vcenter{\hbox{$#2#3$ }}\kern-.6\wd0}}

\def\lz{{\lambda}}

\def\CG{{\mathcal G}}
\def\CA{{\mathcal A}}
\def\CS{{\mathcal S}}
\def\CY{\mathcal Y}
\def\RY{{\mathrm Y}}
\def\gz{\gamma}

\def\UC{{\mathrm{UC}}}

\DeclareMathOperator{\diam}{diam}

\newcommand{\fin}{\mathrm{fin}}

\newtheorem{theorem}{Theorem}[section]
\newtheorem{lemma}[theorem]{Lemma}

\newtheorem{proposition}[theorem]{Proposition}

\theoremstyle{definition}
\newtheorem{remark}[theorem]{Remark}
\newtheorem{definition}[theorem]{Definition}

\renewcommand{\appendix}{\par
   \setcounter{section}{0}%
   \setcounter{subsection}{0}%
   \setcounter{subsubsection}{0}%
   \gdef\thesection{\@Alph\c@section}%
   \gdef\thesubsection{\@Alph\c@section.\@arabic\c@subsection}%
   \gdef\theHsection{\@Alph\c@section.}%
   \gdef\theHsubsection{\@Alph\c@section.\@arabic\c@subsection}%
   \csname appendixmore\endcsname
 }

\newcommand{\go}[1]{\CG_0^\eta(#1)}

\newcommand{\GO}[1]{\mathring{\CG}(#1)}
\newcommand{\GOO}[1]{\mathring{\CG}^\eta_0(#1)}

\newcommand{\at}{\mathrm{at}}

\numberwithin{equation}{section}

\begin{document}
\title{\bf\Large A Complete Real-Variable Theory of Hardy Spaces on Spaces of Homogeneous Type
\footnotetext{\hspace{-0.35cm} 2010 {\it Mathematics Subject Classification}. Primary 42B30;
Secondary 42B25, 42B20, 30L99.\endgraf
{\it Key words and phrases.} space of homogeneous type, Hardy space, maximal function, atom,
Littlewood-Paley function, wavelet.\endgraf
This project is supported by the National Natural Science Foundation of China (Grant Nos.\ 11771446, 11571039, 11726621 and 11761131002). Ji Li is supported by ARC DP 160100153.}}
\author{Ziyi He, Yongsheng Han, Ji Li, Liguang Liu, Dachun Yang\footnote{Corresponding author / April 1, 2018 / newest
version.}\ \ and Wen Yuan}
\date{}
\maketitle
	
\vspace{-0.8cm}

\begin{center}
\begin{minipage}{13cm}
{\small {\bf Abstract}\quad
Let $(X,d,\mu)$ be a space of homogeneous type, with the upper dimension $\omega$, in the sense of
R. R. Coifman and G. Weiss. Assume that $\eta$ is the smoothness index of the wavelets on $X$ constructed by P.
Auscher and T. Hyt\"onen. In this article, when $p\in(\omega/(\omega+\eta),1]$, for
the atomic Hardy spaces $H_{\mathrm{cw}}^p(X)$ introduced by Coifman and Weiss, the authors establish their various
real-variable characterizations, respectively, in terms of the grand maximal function, the radial maximal
function, the non-tangential maximal functions, the various Littlewood-Paley functions and wavelet functions.
This completely answers the question of R. R. Coifman and G. Weiss by showing that
no any additional (geometrical) condition is necessary to guarantee the radial maximal function characterization
of $H_{\mathrm{cw}}^1(X)$ and even of $H_{\mathrm{cw}}^p(X)$ with $p$ as above.
As applications, the authors obtain the finite atomic characterizations of $H^p_{\mathrm{cw}}(X)$, which further 
induce some criteria for the boundedness of sublinear operators on $H^p_{\mathrm{cw}}(X)$.
Compared with the known results, the novelty of this article is that $\mu$ is not
assumed to satisfy the reverse doubling condition and $d$ is only a
quasi-metric, moreover, the range $p\in(\omega/(\omega+\eta),1]$ is natural and optimal.}
\end{minipage}
\end{center}



\vspace{0.2cm}

\section{Introduction}\label{intro}

The real-variable theory of Hardy spaces plays a fundamental role in harmonic analysis.
The classical Hardy space on the $n$-dimensional Euclidean space $\rn$ was initially
developed by Stein and Weiss \cite{sw60} and later by Fefferman and Stein \cite{FS72}.
Hardy spaces $H^p(\rn)$ have been proved to be a suitable substitute of Lebesgue spaces $L^p(\rn)$ with $p\in(0,1]$
in the study of the boundedness of operators.
Indeed, any element in the Hardy space can be decomposed into a sum of some basic elements
(which are called \emph{atoms}); see
Coifman \cite{Coi74} for $n=1$ and Latter \cite{Lat78} for general $n\in\nn$.
Characterizations of Hardy spaces via Littlewood-Paley functions were due to Uchiyama
\cite{Uchi85}. For more study on classical Hardy spaces on $\rn$, we refer the reader to the
well-known monographs \cite{Stein93,Lu95,gr85,Gra1,Gra2}. Modern developments regarding the real-variable
theory of Hardy spaces are so deep and vast that we can only list a few literatures here,
for example, the theory of Hardy spaces associated with operators (see
\cite{BT, BDL, hm09,dy05}), Hardy spaces with variable exponents (see \cite{NS12}),
the real-variable theory of Musielak-Orlicz Hardy spaces (see \cite{Ky14,ylk17}), and also Hardy spaces for
ball quasi-Banach spaces (see \cite{SHYY17}).

In this article, we focus on the real-variable theory of Hardy spaces on spaces of homogeneous type.
It is known that the space of homogeneous type introduced by Coifman and Weiss
\cite{CW71, CW77} provides a natural setting for
the study of both functions spaces and the boundedness of operators.
A \emph{quasi-metric space} $(X,d)$ is a non-empty set $X$ equipped with a \emph{quasi-metric} $d$,
that is, a non-negative function defined on $X\times X$, satisfying that, for any $x,\ y,\ z\in X$,
\begin{enumerate}
\item $d(x,y)=0$ if and only if $x=y$;
\item $d(x,y)=d(y,x)$;
\item there exists a  constant $A_0\in[1,\infty)$ such that $d(x,z)\le A_0[d(x,y)+d(y,z)]$.
\end{enumerate}
The \emph{ball $B$} on $X$ centered at $x_0\in X$ with radius $r\in(0,\fz)$ is defined by setting
$$
B:=B(x_0,r):=\{x\in X:\ d(x,x_0)<r\}.
$$ For any ball $B$ and $\tau\in(0,\infty)$, denote by
$\tau B$ the ball with the same center as that of $B$
but of radius $\tau$ times that of $B$.
Given a quasi-metric space $(X,d)$ and a non-negative measure $\mu$, we call $(X,d,\mu)$ a \emph{space of homogeneous type} if $\mu$ satisfies the {\it doubling condition}:
there exists a positive constant $C_{(\mu)}\in[1,\fz)$ such that,
for any ball $B\subset X$,
$$\mu(2B)\le
C_{(\mu)}\mu(B).$$
The above doubling condition is equivalent to that, for any ball $B$ and $\lz\in[1,\infty)$,
\begin{equation}\label{eq:doub}
\mu(\lz B)\le C_{(\mu)}\lz^\omega\mu(B),
\end{equation}
where $\omega:=\log_2{C_{(\mu)}}$ is called the \emph{upper dimension} of $X$.
If $A_0=1$, we call $(X, d, \mu)$ a \emph{doubling metric measure space}.

According to \cite[pp.\,587--588]{CW77}, we \emph{always make} the following assumptions throughout this article.
For any point $x\in X$, assume that the balls $\{B(x,r)\}_{r\in(0,\infty)}$ form a \emph{basis} of open
neighborhoods of $x$; assume that $\mu$ is Borel regular, which means that open sets are measurable and every
set $A\subset X$ is contained in a Borel set $E$ satisfying that $\mu(A)=\mu(E)$; we also assume that
$\mu(B(x, r))\in(0,\fz)$ for any $x\in X$ and $r\in(0,\infty)$. For the presentation concision,
we always assume that  $(X,d,\mu)$ is non-atomic [namely, $\mu(\{x\})=0$ for any
$x\in X$] and $\diam (X):=\sup\{d(x,y):\ x,\,y\in X\}=\fz$. It is known that $\diam (X)=\fz$ implies that
$\mu(X)=\fz$ (see, for example, \cite[Lemma 8.1]{AH13}).

Let us recall the notion of the atomic Hardy space on spaces of homogeneous type introduced by
Coifman and Weiss \cite{CW77}. For any $\az\in (0,\fz)$, the \emph{Lipschitz space} $\CL_\az(X)$
is defined to be the collection of all measurable functions $f$ such that
$$
\|f\|_{\CL_\az(X)}:=\sup_{x\neq y}\frac{|f(x)-f(y)|}{[\mu(B(x,d(x,y)))]^\az}<\fz.
$$
Denote by $(\CL_\az(X))'$ the \emph{dual space} of $\CL_\az(X)$ equipped with the weak-$\ast$ topology.

\begin{definition}\label{def:atom}
Let $p\in(0,1]$ and $q\in(p,\fz]\cap [1,\fz]$. A function $a$ is called a \emph{$(p,q)$-atom} if
\begin{enumerate}
\item $\supp a:=\{x\in X:\ a(x)\neq 0\}\subset B(x_0,r)$ for some $x_0\in X$ and $r\in(0,\fz)$;
\item $[\int_X |a(x)|^q\,d\mu(x)]^{\frac 1q}\le[\mu(B(x_0,r))]^{\frac 1q-\frac 1p}$;
\item $\int_X a(x)\,d\mu(x)=0$.
\end{enumerate}
The \emph{atomic Hardy space $H^{p,q}_\cw(X)$} is defined as the subspace of $(\CL_{1/p-1}(X))'$ when
$p\in(0,1)$ or of $L^1(X)$ when $p=1$, which consists of all the elements $f$ admitting an atomic decomposition
\begin{equation}\label{eq:atcw}
f=\sum_{j=0}^\fz\lz_ja_j,
\end{equation}
where $\{a_j\}^\fz_{j=0}$ are $(p,q)$-atoms, $\{\lz_j\}_{j=0}^\fz\subset\cc$ satisfies
$\sum_{j=0}^\fz|\lz_j|^p<\fz$ and the series in \eqref{eq:atcw} converges in $(\CL_{1/p-1}(X))'$ when
$p\in(0,1)$ or in $L^1(X)$ when $p=1$. Define
$$
\|f\|_{H^{p,q}_\cw(X)}:=\inf\lf\{\lf(\sum_{j=0}^\fz|\lz_j|^p\r)^{\frac 1p} \r\},
$$
where the infimum is taken over all the representations of $f$ as in \eqref{eq:atcw}.
\end{definition}

It was proved in \cite{CW77} that the atomic Hardy space $H^{p,q}_\cw(X)$ is independent of the choice of
$q$ and hence we sometimes write $H^p_\cw(X)$ for short.
It was also proved in \cite{CW77}  that the dual space of $H^{p}_\cw(X)$ is the Lipschitz space $\CL_{1/p-1}(X)$
when $p\in (0,1)$, and the space $\mathrm{BMO}(X)$  of bounded mean oscillation when $p=1$.

It is well known that the most basic result in the real-variable theory of Hardy spaces is their
characterizations in terms of maximal functions. Coifman and Weiss \cite[pp.\ 641--642]{CW77} observed that a
proof of the duality result between $H^1(\rn)$ and $\mathrm{BMO}(\rn)$ from Carleson \cite{Car76} can be
extended to the general setting of spaces of homogeneous type provided a
certain additional geometrical assumption is added, from which one can then obtain a
radial maximal function characterization of $H_\cw^1(X)$.  Coifman and Weiss \cite[p.\ 642]{CW77}  then
asked that \emph{to what extent their geometrical condition is
necessary for the validity of the radial maximal characterization of $H_\cw^1(X)$}.
Since then, lots of  efforts are made to build various real-variable characterizations of the atomic Hardy
spaces on spaces of homogeneous type with few geometrical assumptions. In this article, we completely answer
the aforementioned question of Coifman and Weiss by showing that no any additional
(geometrical) condition is necessary to guarantee the radial maximal function characterization
of $H_\cw^1(X)$ and even of $H_\cw^p(X)$ with $p\le 1$ but near to $1$.

Recall that a triple $(X,d,\mu)$ is said to be \emph{Ahlfors-$n$ regular} if $\mu(B(x,r))\sim r^n$ for any $x\in X$ and
$r\in(0,\diam X)$ with equivalent positive constants independent of $x$ and $r$. When $(X,d,\mu)$ is
Ahlfors-$n$ regular, upon assuming the quasi-metric $d$ satisfying that there exists
$\thz\in(0,1)$ such that, for any $x,\ x',\ y\in X$,
\begin{equation}\label{regular-d}
|d(x,y)-d(x',y)|\ls[d(x,x')]^{\thz}[d(x,y)+d(x',y)]^{1-\thz},
\end{equation}
Mac\'{i}as and Segovia \cite{MS79b}
characterized Hardy spaces via the grand maximal functions, and Li \cite{Li98} obtained another grand maximal
function characterization
via test functions introduced in \cite{hs94}. Also, Duong and Yan \cite{dy03} characterized Hardy spaces via
the Lusin area function associated with certain semigroup.

Recall that an RD-\emph{space} $(X,d,\mu)$
is a doubling metric measure space with the measure $\mu$ further satisfying the
\emph{reverse doubling condition}, that is, there exist a
positive constant $\wz{C}\in(0,1]$ and $\kz\in(0,\omega]$ such that, for any ball $B(x,r)$ with $x\in X$,
$r\in(0,\diam X/2)$ and $\lz\in[1,\diam X/[2r])$,
$$
\wz C\lz^\kz\mu(B(x,r))\le\mu(B(x,\lz r)).
$$
Indeed, any path connected doubling metric measure space  is an RD-space (see \cite{HMY08,YZ11}).
Characterizations of Hardy spaces on RD-spaces via various Littlewood-Paley functions were established in
\cite{HMY06, HMY08}. Also, characterizations of Hardy spaces on RD-spaces via various maximal functions can be
found in \cite{GLY08,gly09b,YZ10}. It should be mentioned that local Hardy
spaces can be used to characterize more general scale of function spaces like Besov and Triebel-Lizorkin spaces
on RD-spaces (see \cite{YZ11}). For a systematic study of Besov and Triebel-Lizorkin spaces on RD-spaces, we refer the
reader to \cite{HMY08}. More on analysis over Ahlfors-$n$ regular metric measure spaces or RD-spaces can be
found in \cite{glmy, GLY09, kyz10, kyz11, yz08, hyz09,YZ11,dh09, ZSY16}.

The main motivation of studying the real-variable theory of function spaces and the boundedness of operators on
spaces of homogeneous type comes from the celebrated work of Auscher and Hyt\"onen \cite{AH13}, in which they
constructed an orthonormal wavelet basis $\{\psi_\az^k:\ k\in\zz,\ \az\in\CG_k\}$ of $L^2(X)$
with H\"older continuity exponent $\eta\in(0,1)$ and exponential decay by
using the system of random dyadic cubes. The first creative attempt of using the idea of \cite{AH13} to
investigate the real-variable theory of Hardy spaces on spaces of homogeneous type was due to
Han et al.\ \cite{HHL16} (see also Han et al.\ \cite{hhl17}).
Indeed, in \cite{HHL16}, Hardy spaces via wavelets on spaces of homogeneous type were
introduced and then these spaces were proved to have atomic decompositions. The method used in
\cite{HHL16} is based on a new Calder\'{o}n reproducing formula on spaces of homogeneous type (see
\cite[Proposition 2.5]{HHL16}). But there exists an \emph{error} in the proof of
\cite[Proposition 2.5]{HHL16}, namely, since the regularity exponent of the approximations
of the identity in \cite[p.\ 3438]{HHL16} is
$\thz$ [indeed, $\thz$ is from the regularity of the quasi-metric $d$ in \eqref{regular-d}],
it follows that the regularity exponent in \cite[(2.6)]{HHL16} should be
$\min\{\thz,\eta\}$ and hence the correct range of $p$ in \cite[Proposition 2.5]{HHL16} (indeed, all results of
\cite{HHL16}) seems to be $(\omega/[\omega+\min\{\thz,\eta\}],1]$ which is not optimal.
Moreover, the criteria of the boundedness of Calder\'{o}n-Zygmund operators on the dual of Hardy spaces were
established in \cite{HHL16}. Also, Fu and Yang \cite{fy18} obtained an unconditional basis of  $H^1_\cw(X)$
and several equivalent characterizations of $H^1_\cw(X)$ in terms of wavelets.

Another motivation of this article comes from the Calder\'{o}n reproducing formulae established in
\cite{HLYY17}. Indeed, the work of \cite{HLYY17} was partly motivated by the wavelet theory of Auscher and
Hyt\"{o}nen in \cite{AH13} and a corresponding wavelet reproducing formula (which can converge in the
distribution space) in \cite{HLYY17}. The already existing works (see \cite{HMY06,HMY08,GLY08,YZ10,YZ11})
regarding Hardy spaces on RD-spaces show the feasibility of establishing various real-variable
characterizations of the atomic Hardy spaces on spaces of homogeneous type via the Calder\'{o}n
reproducing formulae. It should be mentioned that a characterization of the atomic Hardy spaces via the
Littlewood-Paley functions was established in \cite{HLW16} via the aforementioned wavelet reproducing formula;
see also \cite{HLW16} for some corresponding conclusions of product Hardy spaces on spaces of homogeneous type.

In this article,  motivated by \cite{HHL16, HLYY17},
for the atomic Hardy spaces $H_{\cw}^p(X)$ with any $p\in(\omega/[\omega+\eta],1]$, we establish their various
real-variable characterizations, respectively, in terms of the grand maximal function, the radial maximal
function, the non-tangential maximal function, the various Littlewood-Paley functions and wavelets.
Observe that these characterizations are true for $H_{\cw}^p(X)$ with $p\in(\omega/[\omega+\eta],1]$
and $X$ being any space of homogeneous type \emph{without any additional (geometrical) conditions},
which completely answers the aforementioned question asked by Coifman and Weiss \cite[p.\,642]{CW77}.
As an application, we obtain the finite atomic characterizations of Hardy spaces,
which further induce some criteria for the boundedness of sublinear operators
on Hardy spaces. Compared with the known results, the novelty of this article is that $\mu$ is not
assumed to satisfy the reverse doubling condition and $d$ is only a
quasi-metric. Moreover, the range of $p\in({\omega}/{(\omega+\eta)},1]$ for the various maximal function characterizations
and the Littlewood-Paley function characterizations of the atomic Hardy spaces $H^{p}_\cw(X)$
is natural and optimal. The key tool used through this article is those Calder\'on
reproducing formulae from \cite{HLYY17}.

In addition, we point out that, when $X$ is a doubling metric measure space, the finite atomic
characterizations of Hardy spaces are also useful in establishing the bilinear decomposition of the product
space $H^1_\cw(X)\times\mathop\mathrm{BMO}(X)$ and $H^{p}_\cw(X)\times \CL_{1/p-1}(X)$, with $p\in
(\omega/[\omega+\eta], 1)$ in \cite{fyl17, lyy18, fcy17, fy18},
and also in the study of the endpoint boundedness of commutators generated by
Calder\'on-Zygmund operators and $\mathop\mathrm{BMO}(X)$ functions in \cite{lcfy17, lcfy18}.

The organization of this article is as follows.

In Section \ref{pre}, we recall the notions of the space of test functions and the space of distributions
introduced in \cite{HMY06}, as well as the random dyadic cubes in \cite{AH13} and the approximation of the
identity with exponential decay introduced in \cite{HLYY17}. Then we restate the Calder\'{o}n reproducing
formulae established in \cite{HLYY17}.

Section \ref{max} concerns Hardy spaces defined via the grand maximal function, the radial maximal function and
the non-tangential maximal function. We show that these Hardy spaces are all equivalent
to the Lebesgue space $L^p(X)$ when $p\in(1,\fz]$ (see Section \ref{p>1}), and they are
all mutually equivalent when $p\in(\om,1]$ (see Section \ref{p<1}), all in the sense
of equivalent (quasi-)norms. The proof for the latter borrows some ideas from \cite{YZ10} and uses the
Calder\'{o}n reproducing formulae built in \cite{HLYY17}. Moreover, we prove that the Hardy space
$H^{*,p}(X)$ defined via the grand maximal function is independent of the choices of the distribution space
$(\go{\bz,\gz})'$ whenever $\bz,\ \gz\in(\omega[1/p-1],\eta)$; see Proposition \ref{prop:Hpin} below.

Section \ref{atom} is devoted to the atomic characterization of $H^{*,p}(X)$. Notice that, if a distribution
has an atomic decomposition, then it belongs to $H^{*,p}(X)$ obviously by the definition of atoms; see Section \ref{atsub}. All we
remain to do is to establish the converse relationship. In Section \ref{cz},  by modifying the definition of the grand maximal
function $f^*$ to $f^\star$ so that the level set $\{x\in X\!\!:\ f^\star(x)>\lambda\}$ with $\lz\in(0,\infty)$ is open,
we then apply the partition of unity
to the  open set $\Omega_\lambda$ and obtain a Calder\'{o}n-Zygmund decomposition of $f\in H^{*,p}(X)$.
This is further used in Section \ref{prat} to
construct an atomic decomposition of $f$. In Section \ref{cw}, we compare the atomic Hardy spaces
$H^{p,q}_\at(X)$ with $H^{p,q}_\cw(X)$ and prove that they are exactly the same space in the sense of equivalent
(quasi-)norms.

Section \ref{LP} deals with the Littlewood-Paley theory of Hardy spaces.
In Section \ref{LP1}, we show that the Hardy space $H^p(X)$, defined via the Lusin area function, is
independent of the choices of $\exp$-ATIs. In Section \ref{s5.2}, we use the homogeneous continuous Calder\'on
reproducing formula and the molecular characterizations of the atomic Hardy spaces (see \cite{lcfy18})
to establish the atomic decompositions of elements in
$H^p(X)$, and then we connect $H^p(X)$ with $H^{*,p}(X)$.
In Section \ref{LP2},  we characterize  Hardy spaces $H^p(X)$ via  the Lusin area function with aperture, the
Littlewood-Paley $g$-function and the Littlewood-Paley $g_\lz^*$-function.

In Section \ref{wave}, we consider the Hardy space  $H^p_w(X)$ defined
via wavelets, which was introduced in \cite{HHL16}. We improve the result of \cite[Theorem~4.3]{HLW16} and
prove that $H^p_w(X)$ coincides with  $H^p(X)$ in the sense of equivalent (quasi-)norms.

In Section \ref{bd}, as an application,
we obtain criteria of the boundedness of the sublinear operators from Hardy spaces to quasi-Banach spaces.
To this end, we first establish the finite atomic characterizations, namely, we show that, if
$q\in(p,\infty)\cap [1,\infty)$, then $\|\cdot\|_{H^{p,q}_\fin(X)}$ and  $\|\cdot\|_{H^{p,q}_\at(X)}$
are equivalent  (quasi)-norms on  a dense subspace $H^{p,q}_\fin(X)$ of  $H^{p,q}_\at(X)$;  the above equivalence also holds true on  a dense subspace $H^{p,\fz}_\fin(X)\cap\UC(X)$ of $H^{p,\fz}_\at(X)$,
where $\UC(X)$ denotes the space of all uniformly
continuous functions on $X$.

At the end of this section, we make some conventions on notation. We \emph{always assume} that $\omega$ is as
in \eqref{eq:doub} and $\eta$ is the smoothness index of wavelets
(see \cite[Theorem 7.1]{AH13} or Definition \ref{def:eti} below).
We assume that $\dz$ is a very small positive
number, for example, $\dz\le(2A_0)^{-10}$ in order to construct the dyadic cube system and the wavelet system
on $X$ (see \cite[Theorem 2.2]{HK12} or Lemma \ref{cube} below).
For any $x,\ y\in X$ and $r\in(0,\fz)$, let
$$
V_r(x):=\mu(B(x,r))\quad \mathrm{and}\quad V(x,y):=\mu(B(x,d(x,y))),
$$
where $B(x,r):=\{y\in X:\ d(x,y)<r\}$.
We always let $\nn:=\{1,2,\ldots\}$ and $\zz_+:=\nn\cup\{0\}$.
For any $p\in[1,\fz]$, we use $p'$ to denote its \emph{conjugate index}, namely, $1/p+1/p'=1$.
The symbol $C$ denotes a positive constant which is independent of the main parameters, but it may vary from
line to line. We also use $C_{(\az,\bz,\ldots)}$ to denote a positive constant depending on the indicated
parameters $\az$, $\bz$, \ldots. The symbol $A \ls B$ means that there exists a positive constant $C$ such that
$A \le CB$. The symbol $A \sim B$ is used as an abbreviation of $A \ls B \ls A$. We also use
$A\ls_{\az,\bz,\ldots}B$ to indicate that here the implicit positive constant depends on $\az$, $\bz$, $\ldots$
and, similarly, $A\sim_{\az,\bz,\ldots}B$. For any $s,\ t\in\rr$, denote the \emph{minimum} of $s$ and $t$ by
$s\wedge t$. For any finite set $\CJ$, we use $\#\CJ$ to denote its \emph{cardinality}. Also, for any set
$E$ of $X$, we use $\chi_E$ to denote its characteristic function and $E^\complement$ the set
$X\setminus E$.

\section{Calder\'{o}n reproducing formulae\label{pre}}

This section is devoted to recalling Calder\'{o}n reproducing formulae obtained in \cite{HLYY17}. To this end,
we first recall the notions of both the space of test functions and the distribution space.
\begin{definition}\label{def:test}
Let $x_1\in X$, $r\in(0,\fz)$, $\bz\in(0,1]$ and $\gz\in(0,\fz)$. A function $f$ defined on $X$ is called a
\emph{test function of type $(x_1,r,\bz,\gz)$}, denoted by $f\in\CG(x_1,r,\bz,\gz)$, if there exists a positive
constant $C$ such that
\begin{enumerate}
\item (the \emph{size condition}) for any $x\in X$,
\begin{equation*}
|f(x)|\le C\frac{1}{V_r(x_1)+V(x_1,x)}\lf[\frac r{r+d(x_1,x)}\r]^\gz;
\end{equation*}
\item (the \emph{regularity condition}) for any $x,\ y\in X$ satisfying $d(x,y)\le (2A_0)^{-1}[r+d(x_1,x)]$,
\begin{equation*}
|f(x)-f(y)|\le C\lf[\frac{d(x,y)}{r+d(x_1,x)}\r]^\bz
\frac{1}{V_r(x_1)+V(x_1,x)}\lf[\frac r{r+d(x_1,x)}\r]^\gz.
\end{equation*}
\end{enumerate}
For any $f\in\CG(x_1,r,\bz,\gz)$, define the norm
$$
\|f\|_{\CG(x_1,r,\bz,\gz)}:=\inf\{C\in(0,\fz):\ C \textup{\ satisfies (i) and (ii)}\}.
$$
Define
$$
\mathring{\CG}(x_1,r,\bz,\gz):=\lf\{f\in\CG(x_1,r,\bz,\gz):\ \int_X f(x)\,d\mu(x)=0\r\}
$$
equipped with the norm $\|\cdot\|_{\mathring{\CG}(x_1,r,\bz,\gz)}:=\|\cdot\|_{\CG(x_1,r,\bz,\gz)}$.
\end{definition}

Observe that the above version of $\CG(x_1,r,\bz,\gz)$ was originally introduced by Han et al.\ \cite{HMY08}
(see also \cite{HMY06}).

Fix $x_0\in X$. For any $x\in X$ and $r\in(0,\fz)$, we know that $\CG(x,r,\bz,\gz)=\CG(x_0,1,\bz,\gz)$
with equivalent norms, but the equivalent positive constants depend on $x$ and $r$.
Obviously, $\CG(x_0,1,\bz,\gz)$ is a Banach space. In what follows, we simply write
$\CG(\bz,\gz):=\CG(x_0,1,\bz,\gz)$ and $\GO{\bz,\gz}:=\GO{x_0,1,\bz,\gz}$.

Fix $\epsilon\in(0,1]$ and $\bz,\ \gz\in(0,\epsilon)$. Let $\CG^\epsilon_0(\bz,\gz)$ [resp.,
$\mathring\CG^\epsilon_0(\bz,\gz)$] be the completion of the set $\CG(\epsilon,\epsilon)$ [resp.,
$\mathring\CG(\epsilon,\epsilon)$] in $\CG(\bz,\gz)$, that is, if $f\in\CG_0^\ez(\bz,\gz)$ [resp.,
$f\in\mathring\CG_0^\ez(\bz,\gz)$], then there exists
$\{\phi_j\}_{j=1}^\fz\subset\CG(\ez,\ez)$ [resp., $\{\phi_j\}_{j=1}^\fz\subset\mathring\CG(\ez,\ez)$] such that
$\|\phi_j-f\|_{\CG(\bz,\gz)}\to 0$ as $j\to\fz$. If $f\in\CG^\epsilon_0(\bz,\gz)$ [resp.,
$f\in\mathring\CG^\epsilon_0(\bz,\gz)$], we then let
$$
\|f\|_{\CG^\epsilon_0(\bz,\gz)}:=\|f\|_{\CG(\bz,\gz)}\quad [\textup{resp., }
\|f\|_{\mathring\CG^\epsilon_0(\bz,\gz)}:=\|f\|_{\CG(\bz,\gz)}].
$$
The \emph{dual space} $(\CG^\epsilon_0(\bz,\gz))'$ [resp., $(\mathring{\CG}^\epsilon_0(\bz,\gz))'$] is defined
to be the set of all continuous linear functionals on
$\CG^\epsilon_0(\bz,\gz)$ [resp., $\mathring{\CG}^\epsilon_0(\bz,\gz)$] and
equipped with the weak-$*$ topology. The spaces $(\CG^\epsilon_0(\bz,\gz))'$ and
$(\mathring{\CG}^\epsilon_0(\bz,\gz))'$ are called the \emph{spaces of distributions}.

Let
$L_\loc^1(X)$ be the space of all locally integrable functions on $X$.
Denote by $\CM$ the \emph{Hardy-Littlewood maximal operator}, that is, for any  $f\in L_\loc^1(X)$ and $x\in X$,
\begin{equation*}
\CM(f)(x):=\sup_{B\ni x}\frac 1{\mu(B)}\int_{B} |f(y)|\,d\mu(y),
\end{equation*}
where the supremum is taken over all balls $B$ of $X$ that contain $x$. For any $p\in(0,\fz]$,  the
\emph{Lebesgue space $L^p(X)$} is defined to be the set of all $\mu$-measurable functions $f$ such that
$$
\|f\|_{L^p(X)}:=\lf[\int_X |f(x)|^p\,d\mu(x)\r]^{\frac 1p}<\fz
$$
with the usual modification made when $p=\fz$; the \emph{weak Lebesgue space  $L^{p,\fz}(X)$} is defined to be
the set of all $\mu$-measurable functions $f$ such that
$$
\|f\|_{L^{p,\fz}(X)}:=\sup_{\lz\in(0,\fz)}\lz[\mu(\{x\in X:\ |f(x)|>\lz\})]^{\frac 1p}<\fz.
$$
It is known (see \cite{CW77}) that $\CM$ is bounded on $L^p(X)$ when $p\in(1,\fz]$ and bounded from $L^1(X)$ to
$L^{1,\fz}(X)$. Then we state some estimates from \cite[Lemma~2.1]{HMY08}, which are proved by using
\eqref{eq:doub}.

\begin{lemma}\label{lem-add}
Let $\bz,\ \gz\in(0,\infty)$.
\begin{enumerate}
\item For any $x,\ y\in X$ and $r\in(0,\fz)$, $V(x,y)\sim V(y,x)$ and
$$
V_r(x)+V_r(y)+V(x,y)\sim  V_r(x)+V(x,y)\sim V_r(y)+V(x,y)\sim\mu(B(x,r+d(x,y))),
$$
where the equivalent positive constants are independent of $x$, $y$ and $r$.

\item There exists a positive constant $C$ such that, for any $x_1\in X$ and $r\in(0,\infty)$,
$$
\int_X\frac{1}{V_r(x_1)+V(x_1,x)}\lf[\frac r{r+d(x_1,x)}\r]^\gz\,d\mu(x)\le C.
$$

\item There exists a positive constant $C$ such that, for any $x\in X$ and $R\in(0,\infty)$,
$$
\int_{d(x,y)\le R}\frac 1{V(x,y)}\lf[\frac{d(x,y)}{R}\r]^\bz\,d\mu(y)\le C
\quad \textit{and}\quad\int_{d(x, y)\ge R}\frac{1}{V(x,y)}\lf[\frac R{d(x,y)}\r]^\bz\,d\mu(y)\le C.
$$

\item There exists a positive constant $C$ such that, for any $x_1\in X$ and $R,\ r\in(0,\infty)$,
$$
\int_{d(x, x_1)\ge R}\frac{1}{V_r(x_1)+V(x_1,x)}\lf[\frac r{r+d(x_1,x)}\r]^\gz\,d\mu(x)\
\le C\lf(\frac{r}{r+R}\r)^\gz.
$$

\item There exists a positive constant $C$ such that, for any $r\in(0,\fz)$, $f\in L^1_\loc(X)$ and $x\in X$,
$$
\int_X \frac{1}{V_r(x)+V(x,y)}\lf[\frac{r}{r+d(x,y)}\r]^\gz|f(y)|\,d\mu(y)\le C\CM(f)(x).
$$
\end{enumerate}
\end{lemma}
Next we recall the system of dyadic cubes established in \cite[Theorem 2.2]{HK12} (see also \cite{AH13}),
which is restated in the following version.

\begin{lemma}\label{cube}
Fix constants $0<c_0\le C_0<\fz$ and $\dz\in(0,1)$ such that $12A_0^3C_0\dz\le c_0$. Assume that
a set of points, $\{z_\az^k:\ k\in\zz,\ \az\in\CA_k\}\subset X$ with $\CA_k$ for any $k\in\zz$ being a
countable set of indices, has the following properties: for any $k\in\zz$,
\begin{enumerate}
\item[\rm (i)] $d(z_\az^k,z_\bz^k)\ge c_0\dz^k$ if $\az\neq\bz$;
\item[\rm (ii)] $\min_{\az\in\CA_k} d(x,z_\az^k)\le C_0\dz^k$ for any $x\in X$.
\end{enumerate}
Then there exists a family of sets, $\{Q_\az^k:\  k\in\zz,\ \az\in\CA_k\}$, satisfying
\begin{enumerate}
\item[\rm (iii)] for any $k\in\zz$, $\bigcup_{\az\in\CA_k} Q_\az^k=X$ and $\{ Q_\az^k:\;\; {\az\in\CA_k}\}$ is disjoint;
\item[\rm (iv)] if $k,\ l\in\zz$ and $l\ge k$, then either $Q_\bz^l\subset Q_\az^k$ or
$Q_\bz^l\cap Q_\az^k=\emptyset$;
\item[\rm (v)] for any $k\in\zz$ and $\az\in\CA_k$, $B(z_\az^k,c_\natural\dz^k)\subset Q_\az^k\subset B(z_\az^k,C^\natural\dz^k)$,
where $c_\natural:=(3A_0^2)^{-1}c_0$, $C^\natural:=2A_0C_0$ and $z_\az^k$ is called ``the center'' of
$Q_\az^k$.
\end{enumerate}
\end{lemma}

Throughout this article, we keep the notation used in Lemma \ref{cube}. Moreover, for any $k\in\zz$, let
$$\CX^k:=\{z_\az^k\}_{\az\in\CA_k},\qquad \CG_k:=\CA_{k+1}\setminus\CA_k\qquad \textup{and}\qquad
\CY^k:=\{z_\az^{k+1}\}_{\az\in\CG_k}=:\{y_\az^{k}\}_{\az\in\CG_k}.$$
Next we recall the notion of approximations of the identity with exponential decay introduced in
\cite{HLYY17}.
\begin{definition}\label{def:eti}
A sequence $\{Q_k\}_{k\in\zz}$ of bounded linear integral operators on $L^2(X)$ is called an
\emph{approximation of the identity with exponential decay} (for short, $\exp$-ATI) if there exist constants
$C,\ \nu\in(0,\fz)$, $a\in(0,1]$ and $\eta\in(0,1)$ such that, for any $k\in\zz$, the kernel of operator
$Q_k$, which is still denoted by $Q_k$, satisfying
\begin{enumerate}
\item (the \emph{identity condition}) $\sum_{k=-\fz}^\fz Q_k=I$ in $L^2(X)$, where $I$ is the identity operator
on $L^2(X)$;
\item (the \emph{size condition}) for any $x,\ y\in X$,
\begin{align}\label{eq:etisize}
|Q_k(x,y)|&\le C\frac1{\sqrt{V_{\dz^k}(x)\,V_{\dz^k}(y)}}
\exp\lf\{-\nu\lf[\frac{d(x,y)}{\dz^k}\r]^a\r\}\\
&\quad\times\exp\lf\{-\nu\lf[\frac{\max\{d(x, \CY^k),\,d(y,\CY^k)\}}{\dz^k}\r]^a\r\};\noz
\end{align}
\item (the \emph{regularity condition}) for any $x,\ x',\ y\in X$ with $d(x,x')\le\dz^k$,
\begin{align}\label{eq:etiregx}
&|Q_k(x,y)-Q_k(x',y)|+|Q_k(y,x)-Q_k(y, x')|\\
&\quad\le C\lf[\frac{d(x,x')}{\dz^k}\r]^\eta
\frac1{\sqrt{V_{\dz^k}(x)\,V_{\dz^k}(y)}} \exp\lf\{-\nu\lf[\frac{d(x,y)}{\dz^k}\r]^a\r\}\noz\\
&\qquad\times\exp\lf\{-\nu\lf[\frac{\max\{d(x, \CY^k),\,d(y,\CY^k)\}}{\dz^k}\r]^a\r\};\noz
\end{align}
\item (the \emph{second difference regularity condition}) for any $x,\ x',\ y,\ y'\in X$ with $d(x,x')\le\dz^k$
and $d(y,y')\le\dz^k$, then
\begin{align}\label{eq:etidreg}
&|[Q_k(x,y)-Q_k(x',y)]-[Q_k(x,y')-Q_k(x',y')]|\\
&\quad \le C\lf[\frac{d(x,x')}{\dz^k}\r]^\eta\lf[\frac{d(y,y')}{\dz^k}\r]^\eta
\frac1{\sqrt{V_{\dz^k}(x)\,V_{\dz^k}(y)}} \exp\lf\{-\nu\lf[\frac{d(x,y)}{\dz^k}\r]^a\r\}\noz\\
&\qquad\times\exp\lf\{-\nu\lf[\frac{\max\{d(x, \CY^k),\,d(y,\CY^k)\}}{\dz^k}\r]^a\r\};\noz
\end{align}
\item (the \emph{cancelation condition}) for any $x,\ y\in X$,
\begin{equation*}
\int_X Q_k(x,y')\,d\mu(y')=0=\int_X Q_k(x',y)\,d\mu(x').
\end{equation*}
\end{enumerate}
\end{definition}
\begin{remark}
By \cite[Remark 2.8]{HLYY17}, we know that the factor $\frac 1{\sqrt{V_{\dz^k}(x)V_{\dz^k}(y)}}$ in
\eqref{eq:etisize}, \eqref{eq:etiregx} and \eqref{eq:etidreg} can be replaced by $\frac 1{V_{\dz^k}(x)}$ or
$\frac 1{V_{\dz^k}(y)}$, and $\max\{d(x,\CY^k),d(y,\CY^k)\}$ by $d(x,\CY^k)$ or by $d(y,\CY^k)$, with
$\exp\{-\nu[\frac{d(x,y)}{\dz^k}]^a\}$ replaced by $\exp\{-\nu'[\frac{d(x,y)}{\dz^k}]^a\}$, where
$\nu'\in(0,\nu)$ only depends on $a$ and $A_0$. Moreover, the condition in Definition \ref{def:eti}(iii)
[resp., (iv)] can be replaced by $d(x,x')\le(2A_0)^{-1}[\dz^k+d(x,y)]$ (resp.,
$d(x,x')\le(2A_0)^{-2}[\dz^k+d(x,y)]$ and $d(y,y')\le(2A_0)^{-2}[\dz^k+d(x,y)]$). For their proofs, see
\cite[Proposition 2.9]{HLYY17}.
\end{remark}
With the above $\exp$-ATI, we have the following homogeneous continuous Calder\'on reproducing formula
established in \cite{HLYY17}.
\begin{theorem}\label{thm:hcrf}
Let $\{Q_k\}_{k\in\zz}$ be an $\exp$-{\rm ATI} and $\bz,\ \gz\in(0,\eta)$.
Then there exists a sequence $\{\wz{Q}_k\}_{k\in\zz}$ of bounded linear operators on $L^2(X)$
such that, for any $f\in (\GOO{\bz,\gz})'$,
\begin{equation*}
f=\sum_{k=-\fz}^\fz \wz{Q}_kQ_kf,
\end{equation*}
where the series converges in $(\mathring{\CG}^\eta_0(\bz,\gz))'$.
Moreover, there exists a positive constant $C$ such that, for any $k\in\zz$, the kernel of $\wz{Q}_k$ satisfies
the following conditions:
\begin{enumerate}
\item for any $x,\ y\in X$,
$$
\lf|\wz{Q}_k(x,y)\r|\le C\frac 1{V_{\dz^k}(x)+V(x,y)}\lf[\frac{\dz^k}{\dz^k+d(x,y)}\r]^\gz;
$$
\item for any $x,\ x',\ y\in X$ with $d(x,x')\le(2A_0)^{-1}[\dz^k+d(x,y)]$,
$$
\lf|\wz{Q}_k(x,y)-\wz{Q}_k(x',y)\r|\le C \lf[\frac{d(x,x')}{\dz^k+d(x,y)}\r]^\bz
\frac 1{V_{\dz^k}(x)+V(x,y)}\lf[\frac{\dz^k}{\dz^k+d(x,y)}\r]^\gz;
$$
\item for any $x\in X$,
\begin{equation*}
\int_X \wz{Q}_k(x,y)\,d\mu(y)=0=\int_X\wz{Q}_k(y,x)\,d\mu(y).
\end{equation*}
\end{enumerate}
\end{theorem}

Next, we recall the homogeneous discrete Calder\'on reproducing formulae established in \cite{HLYY17}.
To this end, let $j_0\in\nn$ be a sufficiently large integer such that $\dz^{j_0}\le(2A_0)^{-4}C^\natural$,
where $C^\natural$ is as in Lemma \ref{cube}. Based on Lemma \ref{cube}, for any $k\in\zz$ and $\az\in\CA_k$,
we let
$$
\CN(k,\az):=\{\tau\in\CA_{k+j_0}:\ Q_\tau^{k+j_0}\subset Q_\az^k\}
$$
and $N(k,\az)$ be the cardinality of the set $\CN(k,\az)$. For any $k\in\zz$ and $\az\in\CA_k$, we rearrange
the set $\{Q_\tau^{k+j_0}:\ \tau\in\CN(k,\az)\}$ as $\{Q_\az^{k,m}\}_{m=1}^{N(k,\az)}$, whose centers are
denoted, respectively, by $\{z_\az^{k,m}\}_{m=1}^{N(k,\az)}$.

\begin{theorem}\label{thm:hdrf}
Let $\{Q_k\}_{k\in\zz}$ be an $\exp$-{\rm ATI} and $\bz,\ \gz\in(0,\eta)$.
For any $k\in\zz$, $\az\in\CA_k$ and  $m\in\{1,\ldots,N(k,\az)\}$, suppose that $y_\az^{k,m}$ is an arbitrary
point in $Q_\az^{k,m}$. Then, for any $i\in\{1,2\}$, there exists a  sequence $\{\wz{Q}_k^{(i)}\}_{k=-\fz}^\fz$
of bounded linear operators on $L^2(X)$ such that, for any $f\in(\GOO{\bz,\gz})'$,
\begin{align*}
f(\cdot)&=\sum_{k=-\fz}^\fz\sum_{\az\in\CA_k}\sum_{m=1}^{N(k,\az)}\wz{Q}^{(1)}_k\lf(\cdot,y_\az^{k,m}\r)
\int_{Q_\az^{k,m}}Q_kf(y)\,d\mu(y)\noz\\
&=\sum_{k=-\fz}^\fz\sum_{\az\in\CA_k}\sum_{m=1}^{N(k,\az)}\mu\lf(Q_\az^{k,m}\r)
\wz{Q}^{(2)}_k\lf(\cdot,y_\az^{k,m}\r)Q_kf\lf(y_\az^{k,m}\r),\noz
\end{align*}
where the equalities converge in $(\GOO{\bz,\gz})'$. Moreover, for any $k\in\zz$, the kernels of
$\wz{Q}^{(1)}_k$ and $\wz{Q}^{(2)}_k$ satisfy (i), (ii) and (iii) of Theorem \ref{thm:hcrf}.
\end{theorem}

To recall the inhomogeneous discrete Calder\'{o}n reproducing formulae established in \cite{HLYY17}, we
introduce the following $1$-$\exp$-ATI and $\exp$-IATI.
\begin{definition}\label{def:1eti}
A sequence $\{P_k\}_{k=-\fz}^\fz$ of bounded linear operators on $L^2(X)$ is called an
\emph{approximation of the identity with exponential decay and integration $1$} (for short, $1$-$\exp$-ATI) if
$\{P_k\}_{k=-\fz}^\fz$ has the following properties:
\begin{enumerate}
\item for any $k\in\zz$, $P_k$ satisfies (ii), (iii) and (iv) of Definition \ref{def:eti} but without the
exponential decay factor
$$
\exp\lf\{-\nu\lf[\frac{\max\{d(x,\CY^{k}),d(y,\CY^{k})\}}{\dz^k}\r]^a\r\};
$$
\item for any $k\in\zz$ and $x\in X$, $\int_X P_{k}(x,y)\,d\mu(y)=1=\int_X P_{k}(y,x)\,d\mu(y)$;
\item for any $k\in\zz$, letting $Q_k:=P_k-P_{k-1}$, then $\{Q_k\}_{k\in\zz}$ is an $\exp$-ATI.
\end{enumerate}
\end{definition}

\begin{remark}\label{rem:ieti}
The existence of the $1$-$\exp$-ATI is guaranteed by \cite[Lemma 10.1]{AH13}.
Moreover, by the proofs of \cite[Proposition 2.9]{HLYY17} and \cite[Proposition 2.7(iv)]{HMY08}, we know that,
for any $f\in L^2(X)$, $\lim_{k\to\fz} P_kf=f$ in $L^2(X)$.
\end{remark}

\begin{definition}\label{def:ieti}
A sequence $\{Q_k\}_{k=0}^\fz$ of bounded linear operators on $L^2(X)$ is called an \emph{inhomogeneous
approximation of the identity with exponential decay} (for short, $\exp$-IATI) if
there exists a $1$-$\exp$-ATI $\{P_k\}_{k=-\fz}^\fz$ such that $Q_0=P_0$ and $Q_k=P_k-P_{k-1}$ for any
$k\in\nn$.
\end{definition}

Next we recall the following inhomogeneous discrete reproducing formula established in \cite{HLYY17}.
\begin{theorem}\label{thm:idrf}
Let $\{Q_k\}_{k=0}^\fz$ be an {\rm $\exp$-IATI} and $\bz,\ \gz\in(0,\eta)$.
Then there exists a sequence $\{\wz{Q}_k\}_{k=0}^\fz$ of bounded
linear operators on $L^2(X)$ such that, for any $f\in(\go{\bz,\gz})'$,
\begin{align*}\label{eq:idrf}
f(\cdot)&=\sum_{k=0}^N\sum_{\az\in\CA_k}\sum_{m=1}^{N(0,\az)}
\int_{Q_\az^{k,m}}\wz{Q}_k(\cdot,y)\,d\mu(y)Q_{\az,1}^{k,m}(f)\\
&\quad+\sum_{k=1}^\fz\sum_{\az\in\CA_k}\sum_{m=1}^{N(k,\az)}\mu\lf(Q_\az^{k,m}\r)
\wz{Q}_k\lf(\cdot,y_\az^{k,m}\r)Q_kf\lf(y_\az^{k,m}\r),\noz
\end{align*}
where the equality converges in $(\go{\bz,\gz})'$, every $y_\az^{k,m}$
is an arbitrary point in $Q_\az^{k,m}$ and, for any $k\in\{0,\ldots,N\}$,
\begin{equation*}
Q_{\az,1}^{k,m}(f):=\frac 1{\mu(Q_{\az}^{k,m})}\int_{Q_\az^{k,m}}Q_kf(u)\,d\mu(u).
\end{equation*}
Moreover, for any $k\in\zz_+$, $\wz{Q}_k$ satisfies (i) and (ii) of Theorem \ref{thm:hcrf} and, for any
$x\in X$,
$$
\int_X \wz{Q}_k(x,y)\,d\mu(y)=\int_X\wz{Q}_k(y,x)\,d\mu(y)=\begin{cases}
1 & \textit{if } k\in\{0,\ldots,N\},\\
0 & \textit{if } k\in\{N+1,N+2,\ldots\},
\end{cases}
$$
where $N\in\nn$ is some fixed constant independent of $f$ and $y_\az^{k,m}$.
\end{theorem}

\section{Hardy spaces via various maximal functions}\label{max}

Let
$\bz,\ \gz\in(0,\eta)$ and $f\in(\go{\bz,\gz})'$. Let $\{P_k\}_{k\in\zz}$ be a $1$-$\exp$-ATI as in Definition
\ref{def:1eti}. Define the \emph{radial maximal function $\CM^+(f)$} of $f$ by setting
$$
\CM^+(f)(x):=\sup_{k\in\zz}|P_kf(x)|,\quad \forall\,x\in X.
$$
Define the \emph{non-tangential maximal function $\CM_\thz(f)$ of $f$ with aperture $\thz\in(0,\fz)$} by
setting
$$
\CM_\thz(f)(x):=\sup_{k\in\zz}\sup_{y\in B(x,\thz\dz^k)}|P_kf(y)|, \quad \forall\,x\in X.
$$
Also, define the \emph{grand maximal function $f^*$} of $f$ by setting
$$
f^*(x):=\sup\lf\{|\langle f,\vz \rangle|:\ \vz\in\go{\bz,\gz} \textup{ and } \|\vz\|_{\CG(x,r_0,\bz,\gz)}\le 1
\textup{ for some } r_0\in(0,\fz)\r\},\quad \forall\,x\in X.
$$
Correspondingly, for any $p\in(0,\fz]$, the \emph{Hardy spaces} $H^{+,p}(X)$, $H^p_\thz(X)$ with
$\thz\in(0,\fz)$ and $H^{*,p}(X)$ are defined, respectively, by setting
\begin{align*}
H^{+,p}(X)&:=\lf\{f\in\lf(\go{\bz,\gz}\r)':\ \|f\|_{H^{+,p}(X)}:=\|\CM^+(f)\|_{L^p(X)}<\fz\r\},\\
H^p_\thz(X)&:=\lf\{f\in\lf(\go{\bz,\gz}\r)':\ \|f\|_{H_\thz^{p}(X)}:=\|\CM_\thz(f)\|_{L^p(X)}<\fz\r\}
\end{align*}
and
$$
H^{*,p}(X):=\lf\{f\in\lf(\go{\bz,\gz}\r)':\ \|f\|_{H^{*,p}(X)}:=\|f^*\|_{L^p(X)}<\fz\r\}.
$$

Based on \cite[Remark 2.9(ii)]{GLY08}, we easily observe that, for any $f\in(\go{\bz,\gz})'$ and $x\in X$,
\begin{equation}\label{eq-xxx}
\CM^+f(x)\le\CM_\thz(f)(x)\le Cf^*(x),
\end{equation}
where $C$ is a positive constant only depending on $\thz$.

The aim of this section is to prove that the Hardy spaces $H^{+,p}(X)$, $H^p_\thz(X)$ and $H^{*,p}(X)$
are mutually equivalent when $p\in(\om,\infty]$ in the sense of equivalent (quasi-)norms (see Section
\ref{p<1}); in particular, they all are equivalent to the Lebesgue space $L^p(X)$ when $p\in(1,\fz]$
in the sense of equivalent norms (see Section \ref{p>1}). Moreover, we prove that
$H^{*,p}(X)$ is independent of the choices of the distribution space $(\go{\bz,\gz})'$ whenever
$\bz,\ \gz\in(\omega(1/p-1),\eta)$; see Proposition \ref{prop:Hpin} below.


\subsection{Equivalence to the Lebesgue space $L^p(X)$ when $p\in(1,\fz]$}\label{p>1}

In this section, we show that the Hardy spaces $H^{+,p}(X)$, $H^p_\thz(X)$ and $H^{*,p}(X)$ are all
equivalent to the Lebesgue space $L^p(X)$, when $p\in(1,\fz]$, in the sense of both representing the same
distributions and equivalent norms. First we give some basic properties of $H^{*,p}(X)$.

\begin{proposition}\label{prop:ban}
Let $p\in(0,\fz]$. Then $H^{*,p}(X)$ is a (quasi-)Banach space, which is continuously embedded into
$(\go{\bz,\gz})'$, where $\bz,\ \gz\in(0,\eta)$.
\end{proposition}

\begin{proof}
Let $f\in H^{*,p}(X)$ and $\vz\in\go{\bz,\gz}$
with $\|\vz\|_{\CG(\bz,\gz)}\le 1$. For any $x\in B(x_0,1)$, by Definition \ref{def:test}, we easily know that
$\|\vz\|_{\CG(x,1,\bz,\gz)}\ls 1$  with the implicit positive constant independent of $x$ and hence
$|\langle f,\vz\rangle|\ls f^*(x)$. Therefore, for any $\vz\in\go{\bz,\gz}$ with $\bz,\ \gz\in(0,\eta)$, we
have
$$
|\langle f,\vz\rangle|^p\ls\frac 1{V_1(x_0)}\int_{B(x_0,1)}[f^*(x)]^p\,d\mu(x)\ls \|f^*\|_{L^p(X)}^p
\sim \|f\|_{H^{*,p}(X)}^p.
$$
This implies that $H^{*,p}(X)$ is continuously embedded into $(\go{\bz,\gz})'$.

To see that $H^{*,p}(X)$ is a (quasi-)Banach space, we only prove its completeness. Indeed,
suppose that $\{f_k\}_{k=1}^\fz$ in $H^{*,p}(X)$ is a Cauchy sequence, which is also a Cauchy sequence in
$(\go{\bz,\gz})'$ with $\bz,\ \gz\in(0,\eta)$. By the completeness of $(\go{\bz,\gz})'$, the sequence
$\{f_k\}_{k=1}^\fz$ converges to some element $f\in(\go{\bz,\gz})'$ as $k\to\fz$. If $\vz\in\go{\bz,\gz}$
satisfies $\|\vz\|_{\CG(x,r_0,\bz,\gz)}\le 1$ for\ some $x\in X$ and $r_0\in(0,\fz)$, then
$|\langle f_{k+l}-f_k,\vz\rangle|\le (f_{k+l}-f_k)^*(x)$ for any $k,\ l\in\nn$. Letting $l\to\fz$, we obtain
$$
|\langle f-f_k,\vz\rangle|\le \liminf_{l\to\fz}(f_{k+l}-f_k)^*(x),
$$
which further implies that, for any $x\in X$,
$$
(f-f_k)^*(x)\le \liminf_{l\to\fz}(f_{k+l}-f_k)^*(x).
$$
By the Fatou lemma, we conclude that
$$
\|(f-f_k)^*\|_{L^p(X)}\le\liminf_{l\to\fz}\|(f_{k+l}-f_k)^*\|_{L^p(X)}\to 0
$$
as $k\to\fz$, which, together with the sublinearity of $\|\cdot\|_{H^{*,p}(X)}$, further implies that
$f\in H^{*,p}(X)$ and $\lim_{k\to\fz}\|f-f_k\|_{H^{*,p}(X)}=0$. Therefore, $H^{*,p}(X)$ is complete. This
finishes the proof of Proposition \ref{prop:ban}.
\end{proof}

To show the equivalence of $H^{+,p}(X)$, $H^p_\thz(X)$ and $H^{*,p}(X)$ to the Lebesgue space $L^p(X)$
when $p\in(1,\fz]$ in the sense of both representing the same distributions and equivalent norms, we need the
following technical lemma.

\begin{lemma}\label{lem:Sklim}
Let $\{P_k\}_{k\in\zz}$ be a {\rm $1$-$\exp$-ATI} as in Definition \ref{def:1eti}. Assume that
$\bz,\ \gz\in(0,\eta)$. Then the following statements hold true:
\begin{enumerate}
\item there exists a positive constant $C$ such that, for any $k\in\zz$ and $\vz\in\CG(\bz,\gz)$,
$\|P_k\vz\|_{\CG(\bz,\gz)}\le C\|\vz\|_{\CG(\bz,\gz)}$;
\item for any $f\in\CG(\bz,\gz)$ and $\bz'\in(0,\bz)$,
$\lim_{k\to\fz}P_kf=f$ in $\CG(\bz',\gz)$;
\item if $f\in\go{\bz,\gz}$ [resp., $f\in(\go{\bz,\gz})'$], then $\lim_{k\to\fz}P_kf=f$ in
$\go{\bz,\gz}$ [resp., $(\go{\bz,\gz})'$].
\end{enumerate}
\end{lemma}

\begin{proof}
The proof of (i) can be obtained by the method used in the proof of \cite[Lemma 4.14]{HLYY17}. The proof of
(ii) is given in \cite[Lemma 3.6]{GLY08}, whose proof does not rely on the reverse doubling condition of $\mu$
and the metric $d$. We obtain (iii) directly by (i), (ii) and a standard duality argument. This finishes the
proof of Lemma \ref{lem:Sklim}.
\end{proof}

Then we have the following proposition.

\begin{proposition}\label{prop:HsubL}
Let $p\in[1,\fz]$, $\bz,\ \gz\in(0,\eta)$ and $\{P_k\}_{k\in\zz}$ be a {\rm $1$-$\exp$-ATI}. If
$f\in(\go{\bz,\gz})'$ belongs to $H^{+,p}(X)$, then there exists $\wz{f}\in L^p(X)$ such that, for any
$\vz\in\go{\bz,\gz}$,
\begin{align}\label{eq-add1}
\langle f,\vz\rangle=\int_X\wz{f}(x)\vz(x)\,d\mu(x)
\end{align}
and $\|\wz{f}\|_{L^p(X)}\le\|\CM^+(f)\|_{L^p(X)}$; moreover, if $p\in[1,\fz)$, then, for almost every $x\in X$,
$|\wz f(x)|\le\CM^+(f)(x)$.
\end{proposition}
\begin{proof}
Let $f\in(\go{\bz,\gz})'$ and $\CM^+(f)=\sup_{k\in\zz} |P_k f|\in L^p(X)$, where $\{P_k\}_{k\in\zz}$ is a $1$-$\exp$-ATI as in Definition
\ref{def:1eti}. Then $\{P_kf\}_{k\in\zz}$ is
uniformly bounded in $L^p(X)$. If $p\in(1,\fz]$, then $p'\in[1,\fz)$ and $L^{p'}(X)$ is separable. Thus, by the
Banach-Alaoglu theorem (see, for example, \cite[Theorem 3.17]{Rudin3}), we find a function $\wz{f}\in L^p(X)$
and a sequence $\{k_j\}_{j=1}^\fz\subset\zz$ such that $k_j\to\fz$ and $P_{k_j}f\to\wz{f}$ as $j\to\fz$ in the
weak-$*$ topology of $L^p(X)$. By this and the H\"older inequality, for any $g\in L^{p'}(X)$, we have
\begin{equation*}
\lf|\int_X\wz f(x)g(x)\,d\mu(x)\r|=\lim_{j\to\fz}\lf|\int_XP_{k_j}f(x)g(x)\,d\mu(x)\r|\le
\|\CM^+(f)\|_{L^p(X)}\|g\|_{L^{p'}(X)},
\end{equation*}
which further implies that $\|\wz{f}\|_{L^p(X)}\le\|\CM^+(f)\|_{L^p(X)}$.

If $p=1$, notice that
$
\|\sup_{k\in\zz} |P_kf|\|_{L^1(X)}=\|\CM^+(f)\|_{L^1(X)}<\fz.
$
Then, by the proof of \cite[Theorem III.C.12]{Woj}, $\{P_kf\}_{k\in\zz}$ is relatively compact in $L^1(X)$.
Therefore, by the Eberlin-\u{S}mulian theorem (see \cite[II.C]{Woj}), we know that $\{P_kf\}_{k\in\zz}$
is weakly sequentially compact, that is, there exist a function $\wz{f}\in L^1(X)$ and a
subsequence $\{P_{k_j}f\}_{j=1}^\fz$ such that $\{P_{k_j}f\}_{j=1}^\fz$ converges to $\wz{f}$
weakly in $L^1(X)$. As the arguments for the case $p\in(1,\infty]$, we still have $\|\wz{f}\|_{L^1(X)}\le\|\CM^+(f)\|_{L^1(X)}$.

Moreover, for any $\vz\in\go{\bz,\gz}$, by the fact $\go{\bz,\gz}\subset L^p(X)$ for any $p\in[1,\fz]$ and
Lemma \ref{lem:Sklim}(iii), we conclude that
\begin{equation}\label{3.2x}
\langle f,\vz\rangle=\lim_{k\to\fz}\langle P_{k_j}f,\vz\rangle=\lim_{j\to\fz}\int_X P_{k_j}f(x)\vz(x)\,d\mu(x)
=\int_X\wz{f}(x)\vz(x)\,d\mu(x).
\end{equation}

Let $p\in[1,\fz)$.
For any $j\in\nn$ and $x\in X$, we have $P_{k_j}(x,\cdot)\in\CG(\eta,\eta)$ (see the proof of
\cite[Proposition 2.10]{HLYY17}), which, together with \eqref{3.2x}, implies that
$$
P_{k_j}f(x)=\langle f,P_{k_j}(x,\cdot)\rangle=\int_X P_{k_j}(x,y)\wz{f}(y)\,d\mu(y)=P_{k_j}\wz f(x).
$$
From this and \cite[Proposition 2.7(iv)]{HMY08}, we deduce that
$\{P_{k_j}f\}_{j\in\nn}$ converges to $\wz f$ in  the sense of $\|\cdot\|_{L^p(X)}$.
Then, by the Riesz theorem, we find a subsequence of
$\{P_{k_j}f\}_{j\in\nn}$, still denoted by $\{P_{k_j}f\}_{j\in\nn}$, such that  $P_{k_j}f(x)\to\wz{f}(x)$
as $k_j\to\fz$
for almost every $x\in X$. Therefore, $|\wz{f}(x)|\le\CM^{+}(f)(x)$ for almost every $x\in X$.
This finishes the proof of Proposition \ref{prop:HsubL}.
\end{proof}
Finally, we show the following main result of this section.
\begin{theorem}\label{thm:H=L}
Let $p\in(1,\fz]$ and $\bz,\ \gz\in (0,\eta)$. Then the following hold true:
\begin{enumerate}
\item if $f\in(\go{\bz,\gz})'$ belongs to $H^{+,p}(X)$, then there exists $\wz{f}\in L^p(X)$ such that \eqref{eq-add1} holds true
and
$\|\wz{f}\|_{L^p(X)}\le\|f\|_{H^{+,p}(X)}$;
\item any $f\in L^p(X)$ induces a distribution on $\go{\bz,\gz}$ as in \eqref{eq-add1}, still denoted by $f$, such that
$f\in H^{*,p}(X)$ and $\|f\|_{H^{*,p}(X)}\le C\|f\|_{L^p(X)}$, where $C$ is a positive constant independent of $f$.
\end{enumerate}
Consequently, for any fixed $\thz\in(0,\fz)$, $H^{+,p}(X)=H^p_\thz(X)=H^{*,p}(X)=L^p(X)$ in the sense of
both representing the same distributions and equivalent norms.
\end{theorem}
\begin{proof}
We obtain (i) directly by Proposition \ref{prop:HsubL}. Now we prove (ii). Suppose that $p\in(1,\fz]$
and $f\in L^p(X)$. Clearly, $f$ induces a distribution on $\go{\bz,\gz}$  as in \eqref{eq-add1}.
By \cite[Proposition 3.9]{GLY08}, we find that, for almost every $x\in X$,
$f^*(x)\ls\CM(f)(x)$, with the implicit positive constant independent of $f$ and $x$. Therefore, from the
boundedness of $\CM$ on $L^p(X)$, we deduce that
$\|f^*\|_{L^p(X)}\ls\|\CM(f)\|_{L^p(X)}\ls\|f\|_{L^p(X)}$. This finishes the proof of (ii).

By (i), (ii) and \eqref{eq-xxx}, we obtain
$H^{+,p}(X)=H^p_\thz(X)=H^{*,p}(X)=L^p(X)$, which completes the proof of Theorem \ref{thm:H=L}.
\end{proof}


\subsection{Equivalence of Hardy spaces defined via various maximal functions}\label{p<1}

The main aim of this section concerns the equivalence of Hardy spaces defined via various maximal functions
for the case $p\in({\omega}/{(\omega+\eta)},1]$. Indeed, our goal is to show the following equivalence theorem.

\begin{theorem}\label{coro:H=}
Assume that $p\in({\omega}/{(\omega+\eta)},1]$ and
$\thz\in(0,\fz)$. Then, for any $f\in (\go{\bz,\gz})'$ with $\bz,\ \gz\in(\omega(1/p-1),\eta)$,
$$
\|f\|_{H^{+,p}(X)}\sim \|f\|_{H^p_\thz(X)}\sim \|f\|_{H^{*,p}(X)},
$$
with equivalent positive constants independent of $f$. In other words,
$H^{+,p}(X)=H^p_\thz(X)=H^{*,p}(X)$ with equivalent (quasi-)norms.
\end{theorem}

To prove Theorem \ref{coro:H=}, we borrow some ideas from \cite{YZ10}. To this end, we need the following
two technical lemmas.

\begin{lemma}\label{lem:g0}
Assume that  $\phi\in\go{\bz,\gz}$ with $\bz,\ \gz\in(0,\eta)$. Let $\sigma:=\int_X\phi(x)\,d\mu(x)$. If
$\psi\in\CG(\eta,\eta)$ with $\int_X\psi(x)\,d\mu(x)=1$, then $\phi-\sigma\psi\in\GOO{\bz,\gz}$.
\end{lemma}

\begin{proof}
Since $\phi\in\go{\bz,\gz}$ with $\bz,\ \gz\in(0,\eta)$, it follows that there exists
$\{\phi_n\}_{n=1}^\fz\subset\CG(\eta,\eta)$ such that $\lim_{n\to\fz}\|\phi-\phi_n\|_{\CG(\bz,\gz)}=0$.
Letting $\sigma_n:=\int_X\phi_n(x)\,d\mu(x)$ for any $n\in\nn$, by Definition \ref{def:test} and Lemma
\ref{lem-add}(ii), we conclude that $\lim_{n\to\fz}|\sigma-\sigma_n|=0$, where $\sigma:=\int_X\phi(x)\,d\mu(x)$.
Let $\vz_n:=\phi_n-\sigma_n\psi$ for any $n\in\nn$. Then $\vz_n\in\GO{\eta,\eta}$ and
$$
\|\phi-\sigma\psi-\vz_n\|_{\CG(\bz,\gz)}\le\|\phi-\phi_n\|_{\CG(\bz,\gz)}+|\sigma-\sigma_n|
\|\psi\|_{\CG(\bz,\gz)}\to 0\quad\textup{as}\; n\to\fz.
$$
Thus, $\phi-\sigma\psi\in\GOO{\bz,\gz}$. This finishes the proof of Lemma \ref{lem:g0}.
\end{proof}

The next lemma comes from \cite[Lemma 5.3]{HMY08}, whose proof remains true for a quasi-metric $d$ and also
does not rely on the reverse doubling condition of $\mu$.

\begin{lemma}\label{lem:max}
Let all the notation be  as in Theorem \ref{thm:hdrf}. Let $k,\ k'\in\zz$,
$\{a_\az^{k,m}\}_{k\in\zz,\ \az\in\CA_k,\ m\in\{1,\ldots,N(k,\az)\}}\subset\cc$,
$\gz\in(0,\eta)$ and $r\in({\omega}/{(\omega+\gz)},1]$. Then there exists a positive constant $C$,
independent of $k$, $k'$, $y_\az^{k,m}\in Q_\az^{k,m}$ and $a_\az^{k,m}$
with $k\in\zz$, $\az\in\CA_k$ and $m\in\{1,\ldots,N(k,\az)\}$, such that, for any $x\in X$,
\begin{align*}
&\sum_{\az\in\CA_k}\sum_{m=1}^{N(k,\az)}\mu\lf(Q_\az^{k,m}\r)\frac{1}{V_{\dz^{k\wedge k'}}(x)+V(x,y_\az^{k,m})}
\lf[\frac{\dz^{k\wedge k'}}{\dz^{k\wedge k'}+d(x,y_\az^{k,m})}\r]^\gz\lf|a_\az^{k,m}\r|\\
&\quad\le C\dz^{[k-(k\wedge k')]\omega(1-\frac 1r)}\lf[\CM\lf(\sum_{\az\in\CA_k}\sum_{m=1}^{N(k,\az)}
\lf|a_\az^{k,m}\r|^r\chi_{Q_\az^{k,m}}\r)(x)\r]^{\frac 1r}.
\end{align*}
\end{lemma}

Now we show Theorem \ref{coro:H=}  by using the above two technical lemmas. In what follows, the \emph{symbol
$\ez\to 0^+$} means that $\ez\in(0,\fz)$ and $\ez\to 0$.

\begin{proof}[Proof of Theorem \ref{coro:H=}]
Let $f\in (\go{\bz,\gz})'$ with $\bz,\ \gz\in(\omega(1/p-1),\eta)$. Fix $\thz\in(0,\fz)$. By \eqref{eq-xxx},
we have
$$
\|\CM^+(f)\|_{L^p(X)}\le\|\CM_\thz(f)\|_{L^p(X)}\ls\|f^*\|_{L^p(X)}.
$$
Thus, the proof of Theorem \ref{coro:H=} is reduced to showing
\begin{align}\label{aim}
\|f^*\|_{L^p(X)}\ls\|\CM^+(f)\|_{L^p(X)}.
\end{align}
To obtain \eqref{aim}, it suffices to prove that, for some $r\in(0,p)$ and any $x\in X$,
\begin{equation}\label{eq:H=}
f^*(x)\ls\CM^+(f)(x)+\lf\{\CM\lf(\lf[\CM^+(f)\r]^r\r)(x)\r\}^{\frac 1r}.
\end{equation}
If \eqref{eq:H=} holds true, then, by the boundedness of $\CM$ on $L^{p/r}(X)$, we conclude that
\begin{equation*}
\|f^*\|_{L^p(X)}\ls\|\CM^+(f)\|_{L^p(X)}+\lf\|\CM\lf(\lf[\CM^+(f)\r]^r\r)\r\|_{L^{p/r}(X)}^{\frac 1r}
\sim\|\CM^+(f)\|_{L^p(X)},
\end{equation*}
which proves \eqref{aim}.

We now fix $x\in X$ and show \eqref{eq:H=}. Let $\{P_k\}_{k\in\zz}$ be a $1$-$\exp$-ATI. For any $k\in\zz$,
define $Q_k:=P_k-P_{k-1}$. Then $\{Q_k\}_{k\in\zz}$ is an $\exp$-ATI. Assume for the moment that, for
any $\vz\in\GOO{\bz,\gz}$ with $\|\vz\|_{\CG(x,\dz^l,\bz,\gz)}\le 1$ for some $l\in\zz$,
\begin{equation}\label{eq:H=2}
|\langle f,\vz\rangle|\ls\lf\{\CM\lf(\lf[\CM^+(f)\r]^r\r)(x)\r\}^{\frac 1r}.
\end{equation}
We now use \eqref{eq:H=2} to show \eqref{eq:H=}. For any $\phi\in\go{\bz,\gz}$ with
$\|\phi\|_{\CG(x,r_0,\bz,\gz)}\le 1$ for some $r_0\in(0,\fz)$, choose $l\in\zz$ such that
$\dz^{l+1}\le r_0<\dz^l$. Clearly, $\|\phi\|_{\CG(x,\dz^l,\bz,\gz)}\ls 1$. Let $\sigma:=\int_X\phi(y)\,d\mu(y)$
and $\vz:=\phi-\sigma P_l(x,\cdot)$. Notice that $\int_X P_l(x,y)\,d\mu(y)=1$ and $P_l(x,\cdot)\in\CG(\eta,\eta)$
(see the proof of \cite[Proposition 2.10]{HLYY17}). From Lemma \ref{lem:g0}, it follows  that
$\vz\in\GOO{\bz,\gz}$. Moreover, $\|\vz\|_{\CG(x,\dz^l,\bz,\gz)}\ls\|\phi\|_{\CG(x,\dz^l,\bz,\gz)}
+|\sigma|\|P_l(x,\cdot)\|_{\CG(x,\dz^l,\bz,\gz)}\ls 1$. By \eqref{eq:H=2}, we know that
\begin{align*}
|\langle f,\phi\rangle|&\le|\langle f,\vz\rangle|+|\sigma||\langle f,P_l(x,\cdot)\rangle|\\
&\ls\lf\{\CM\lf(\lf[\CM^+(f)\r]^r\r)(x)\r\}^{\frac 1r}+|P_lf(x)|
\ls\lf\{\CM\lf(\lf[\CM^+(f)\r]^r\r)(x)\r\}^{\frac 1r}+\CM^+(f)(x),
\end{align*}
which is exactly \eqref{eq:H=}.

It remains to prove \eqref{eq:H=2}. For any $\ez\in(0,\fz)$, choose $y_\az^{k,m}\in Q_\az^{k,m}$ such that
$$
\lf|Q_kf\lf(y_\az^{k,m}\r)\r|\le\inf_{z\in Q_{\az}^{k,m}}|Q_kf(z)|+\ez\le 2\inf_{z\in Q_\az^{k,m}}\CM^+(f)(z)
+\ez.
$$
Let $g:=f|_{\GOO{\bz,\gz}}$ be the restriction of $f$ on
$\GOO{\bz,\gz}$. Obviously, $g\in(\GOO{\bz,\gz})'$ and $\|g\|_{(\GOO{\bz,\gz})'}\le\|f\|_{(\go{\bz,\gz})'}$.
By Theorem \ref{thm:hdrf}, we conclude that
\begin{align*}
\langle f,\vz\rangle=\langle g,\vz\rangle
&=\sum_{k=-\fz}^\fz\sum_{\az\in\CA_k}\sum_{m=1}^{N(k,\az)}\mu\lf(Q_\az^{k,m}\r)\wz{Q}_k^*\vz\lf(y_\az^{k,m}\r)
Q_kg\lf(y_\az^{k,m}\r)\\
&=\sum_{k=-\fz}^\fz\sum_{\az\in\CA_k}\sum_{m=1}^{N(k,\az)}\mu\lf(Q_\az^{k,m}\r)\wz{Q}_k^*\vz\lf(y_\az^{k,m}\r)
Q_kf\lf(y_\az^{k,m}\r),
\end{align*}
where $\wz{Q}_k^*$ denotes the \emph{dual operator} of $\wz{Q}_k$. By the proof of \cite[(3.2)]{HMY08}, which
remains true for a quasi-metric $d$ and does not rely on the reverse doubling condition of $\mu$, we find that,
for any fixed $\bz'\in(0,\bz\wedge\gz)$ and any $k\in\zz$,
\begin{equation}\label{eq:wQ*phi}
\lf|\wz{Q}_k^*\vz\lf(y_\az^{k,m}\r)\r|\ls\dz^{|k-l|\bz'}\frac 1{V_{\dz^{k\wedge l}}(x)+V(x,y_\az^{k,m})}
\lf[\frac{\dz^{k\wedge l}}{\dz^{k\wedge l}+d(x,y_\az^{k,m})}\r]^\gz.
\end{equation}
Choose $\bz'\in(0,\bz\wedge\gz)$ such that $\omega/(\omega+\bz')<p$.
From this and Lemma \ref{lem:max}, we deduce that, for any fixed $r\in(\omega/(\omega+\bz'),p)$,
\begin{align}\label{3.5x}
|\langle f,\vz\rangle|
&\ls\sum_{k=-\fz}^\fz\dz^{|k-l|\bz'}\sum_{\az\in\CA_k}\sum_{m=1}^{N(k,\az)}\mu\lf(Q_\az^{k,m}\r)
\frac {\inf_{z\in Q_{\az}^{k,m}}\CM^+(f)(z)+\ez}{V_{\dz^{k\wedge l}}(x)+V(x,y_\az^{k,m})}
\lf[\frac{\dz^{k\wedge l}}{\dz^{k\wedge l}+d(x,y_\az^{k,m})}\r]^\gz\\
&\ls\sum_{k=-\fz}^\fz\dz^{|k-l|\bz'}\dz^{[k-(k\wedge l)]\omega(1-\frac 1r)}
\lf\{\CM\lf(\sum_{\az\in\CA_k}\sum_{m=1}^{N(k,\az)}\lf[\inf_{z\in Q_{\az}^{k,m}}\CM^+(f)(z)+\ez\r]^r
\chi_{Q_\az^{k,m}}\r)(x)\r\}^{\frac 1r}\noz\\
&\ls\sum_{k=-\fz}^\fz\dz^{|k-l|\bz'}\dz^{[k-(k\wedge l)]\omega(1-\frac 1r)}
\lf\{\CM\lf(\lf[\CM^+(f)+\ez\r]^r\r)(x)\r\}^{\frac 1r}\noz\\
&\ls\lf\{\CM\lf(\lf[\CM^+(f)\r]^r\r)(x)+\ez^r\r\}^{\frac 1r}
\to\lf\{\CM\lf(\lf[\CM^+(f)\r]^r\r)(x)\r\}^{\frac 1r}\quad\textup{as }  \ez\to 0^+.\noz
\end{align}
This proves \eqref{eq:H=2} and hence finishes the proof of Theorem \ref{coro:H=}.
\end{proof}

To conclude this section, we show that the Hardy space $H^{*,p}(X)$ is independent of the choices of
$(\go{\bz,\gz})'$ whenever $\bz,\ \gz\in(\omega(1/p-1),\eta)$.

\begin{proposition}\label{prop:Hpin}
Let $p\in(\om,1]$ and $\bz_1,\ \bz_2,\ \gz_1,\ \gz_2\in(\omega(1/p-1),\eta)$. If
$f\in(\go{\bz_1,\gz_1})'$ and $f\in H^{*,p}(X)$, then $f\in(\go{\bz_2,\gz_2})'$.
\end{proposition}

\begin{proof}
Let $f\in(\go{\bz_1,\gz_1})'$ with $\|f\|_{H^{*,p}(X)}<\fz$.
We first prove that there exists $\thz\in(0,\fz)$ such that, for any $\vz\in\CG(\eta,\eta)$ with
$\|\vz\|_{\CG(\bz_2,\gz_2)}\le 1$,
\begin{equation}\label{eq:dual2}
|\langle f,\vz\rangle|\ls\|\CM_\thz(f)\|_{L^p(X)}.
\end{equation}

Notice that $\vz\in\CG(\eta,\eta)\subset \go{\bz_1,\gz_1}$ and $f\in(\go{\bz_1,\gz_1})'$. With all the notation
involved as in Theorem \ref{thm:idrf}, we have
\begin{align*}
\langle f,\vz\rangle&=
\sum_{k=0}^N\sum_{\az\in\CA_0}\sum_{m=1}^{N(k,\az)}\int_{Q_\az^{k,m}}\wz{Q}_k^*\vz(y)\,d\mu(y)
Q_{\az,1}^{k,m}(f)\\
&\quad+\sum_{k=N+1}^\fz\sum_{\az\in\CA_k}\sum_{m=1}^{N(k,\az)}\mu\lf(Q_\az^{k,m}\r)\wz{Q}_k^*\vz
\lf(y_\az^{k,m}\r)Q_kf\lf(y_\az^{k,m}\r)=:{\rm Z}_1+{\rm Z}_2.
\end{align*}
 Choose $\thz:=2A_0C^\natural$ with $C^\natural$ as in Lemma \ref{cube}(v).
By the definition of $Q_\az^{k,m}$ and Lemma \ref{cube}(v), we have
$Q_{\az}^{k,m}\subset B(z_\az^{k,m},C^\natural\dz^{k+j_0})\subset B(z,2A_0C^\natural\dz^k)
= B(z,\thz\dz^k)$ for any $z\in Q_\az^{k,m}$.

Fix $x\in B(x_0,1)$.  Then $\|\vz\|_{\CG(x,1,\bz_2,\gz_2)}\sim \|\vz\|_{\CG(x_0,1,\bz_2,\gz_2)}\ls 1$.
If $k\in\{0,\ldots,N\}$, then we have $\|\vz\|_{\CG(x,\dz^k,\bz_2,\gz_2)} \sim
\|\vz\|_{\CG(x,1,\bz_2,\gz_2)}\ls 1$,
where the implicit constants are independent of $x$ but can depend on  $N$.
Let $\bz_-:=\min\{\bz_1,\gz_1,\bz_2,\gz_2\}$.
By \cite[(3.2)]{HMY08}, we conclude that, for any $y\in Q_\az^{k,m}$,
$$
\lf|\wz{Q}_k^*\vz(y)\r|\ls\frac 1{V_{1}(x)+V(x,y)}\lf[\frac{1}{1+d(x,y)}\r]^{\bz_-}
\sim\frac 1{V_{1}(x)+V(x,y_\az^{k,m})}\lf[\frac{1}{1+d(x,y_\az^{k,m})}\r]^{\bz_-}.
$$
Moreover, for any $k\in\{0,\ldots,N\}$ and $z\in Q_\az^{k,m}$, we have
$$
\lf|Q_{\az,1}^{k,m}(f)\r|\le\frac{1}{\mu(Q_\az^{k,m})}\int_{Q_\az^{k,m}}[|P_kf(y)|+|P_{k-1}f(y)|]\,d\mu(y)
\le 2\CM_\thz(f)(z).
$$
Thus, we obtain
\begin{align}\label{eq:ind1}
|{\rm Z}_1|\ls\sum_{k=0}^N\sum_{\az\in\CA_k}\sum_{m=1}^{N(k,\az)}
\frac 1{V_{1}(x)+V(x,y_\az^{k,m})}\lf[\frac{1}{1+d(x,y_\az^{k,m})}\r]^{\bz_-}
\inf_{z\in Q_\az^{k,m}}\CM_\thz(f)(z).
\end{align}
If $k\in\{N+1,N+2,\ldots\}$, then $|Q_kf(y_\az^{k,m})|\le 2\inf_{z\in Q_\az^{k,m}}\CM_\thz(f)(z)$. Again, by
$\|\vz\|_{\CG(x,1,\bz_2,\gz_2)}\ls 1$ and \cite[(3.2)]{HMY08}, we find that, for any fixed $\bz'\in(0,\bz_-)$,
\begin{align*}
\lf| \wz{Q}_k^*\vz\lf(y_\az^{k,m}\r)\r|\ls  \dz^{k\bz'}
\frac 1{V_{1}(x)+V(x,y_\az^{k,m})}\lf[\frac{1}{1+d(x,y_\az^{k,m})}\r]^{\bz_-},
\end{align*}
because now $k\in\zz_+$ and we do not need the cancelation of $\vz$. Therefore, we have
\begin{align}\label{eq:ind2}
|{\rm Z}_2|\ls\sum_{k=N+1}^\fz\dz^{k\bz'}\sum_{\az\in\CA_k}\sum_{m=1}^{N(k,\az)}
\frac 1{V_{1}(x)+V(x,y_\az^{k,m})}\lf[\frac{1}{1+d(x,y_\az^{k,m})}\r]^{\bz_-}
\inf_{z\in Q_\az^{k,m}}\CM_\thz(f)(z).
\end{align}
Following the estimation of \eqref{3.5x}, from \eqref{eq:ind1} and \eqref{eq:ind2}, we deduce that, for some
$r\in(\omega/(\omega+\eta),p)$,
$$
|\langle f,\vz\rangle|\ls\lf\{\CM\lf(\lf[\CM_\thz(f)\r]^r\r)(x)\r\}^{\frac 1r}.
$$
Notice that the above inequality holds true for any $x\in B(x_0,1)$. Then, by the boundedness of $\CM$ on
$L^{p/r}(X)$, we further conclude that
$$
|\langle f,\vz\rangle|^p\ls\frac 1{V_1(x_0)}\int_X \lf\{\CM\lf(\lf[\CM_\thz(f)\r]^r\r)(x)\r\}^{\frac pr}
\,d\mu(x)\ls\|\CM_\thz(f)\|_{L^p(X)}^p,
$$
which is exactly \eqref{eq:dual2}.

Combining \eqref{eq:dual2} and \eqref{eq-xxx}, we find that, for any $\vz\in\CG(\eta,\eta)$,
\begin{equation}\label{3.x}
|\langle f,\vz\rangle|\ls\|\CM_\thz(f)\|_{L^p(X)}\|\vz\|_{\CG(\bz_2,\gz_2)}\ls\|f\|_{H^{*,p}(X)}\|
\vz\|_{\CG(\bz_2,\gz_2)}.
\end{equation}
Now let $g\in\go{\bz_2,\gz_2}$. By the definition of $\go{\bz_2,\gz_2}$, we know that there exist
$\{\vz_j\}_{j=1}^\fz\subset\CG(\eta,\eta)$ such that $\|g-\vz_j\|_{\CG(\bz_2,\gz_2)}\to 0$ as $j\to\fz$, which
implies that $\{\vz_j\}_{j=1}^\fz$ is a Cauchy sequence in $\CG(\bz_2,\gz_2)$. By
\eqref{3.x}, we find that, for any $j,\ k\in\nn$,
$$
|\langle f,\vz_j-\vz_k\rangle|\ls \|f\|_{H^{*,p}(X)}\|\vz_j-\vz_k\|_{\CG(\bz_2,\gz_2)}.
$$
Therefore, $\lim_{j\to\fz} \langle f,\vz_j\rangle$ exists and the limit is independent of the choice of
$\{\vz_j\}_{j=1}^\fz$. Thus, it is reasonable to define
$\langle f,g\rangle:= \lim_{j\to\fz}\langle f,\vz_j\rangle$. Moreover, by \eqref{3.x}, we conclude that
$$
|\langle f,g\rangle|=\lim_{j\to\fz}|\langle f,\vz_j\rangle|\ls\|f\|_{H^{*,p}(X)}
\liminf_{j\to\fz}\|\vz_j\|_{\CG(\bz_2,\gz_2)}\sim \|f\|_{H^{*,p}(X)}\|g\|_{\go{\bz_2,\gz_2}}.
$$
This implies $f\in(\go{\bz_2,\gz_2})'$ and $\|f\|_{(\go{\bz_2,\gz_2})'}\ls \|f\|_{H^{*,p}(X)}$, which
completes the proof of Proposition \ref{prop:Hpin}.
\end{proof}


\section{Grand maximal function characterizations of atomic Hardy spaces}\label{atom}

In this section, we establish the atomic characterizations of $H^{*,p}(X)$ with $p\in(\om,1]$.

\begin{definition}\label{def:atH}
Let $p\in(\om,1]$, $q\in(p,\fz]\cap[1,\fz]$ and $\bz,\ \gz\in(\omega(1/p-1),\eta)$.
The \emph{atomic Hardy space $H^{p,q}_\at(X)$} is defined to be the set of all $f\in(\go{\bz,\gz})'$ such
that $f=\sum_{j=1}^\fz \lz_j a_j$, where $\{a_j\}_{j=1}^\fz$ is a sequence of $(p,q)$-atoms and
$\{\lz_j\}_{j=1}^\fz\subset\cc$ satisfies $\sum_{j=1}^\fz|\lz_j|^p<\fz$. Moreover, let
$$
\|f\|_{H^{p,q}_\at(X)}:=\inf\lf(\sum_{j=1}^\fz|\lz_j|^p\r)^{\frac 1p},
$$
where the infimum is taken over all the decompositions of $f$ as above.
\end{definition}

Observe that the difference between $H^{p,q}_\cw(X)$ and $H^{p,q}_\at(X)$ mainly lies on the choices of
distribution spaces. When $(X,d,\mu)$ is a doubling metric measure space, it was proved in
\cite[Theorem~4.4]{lyy18} that $H^{p,q}_\cw(X)$ and $H^{p,q}_\at(X)$ coincide  with equivalent (quasi-)norms.
Since now $d$ is a quasi-metric, for the completeness of this article, we include a proof of their
equivalence in Section \ref{cw} below.

The main aim in this section is to prove the following conclusion.
\begin{theorem}\label{thm:atom}
Let $p\in(\om,1]$, $q\in(p,\fz]\cap[1,\fz]$ and $\bz,\ \gz\in(\omega(1/p-1),\eta)$. As subspaces of
$(\go{\bz,\gz})'$, $H^{*,p}(X)=H^{p,q}_\at(X)$ with equivalent (quasi-)norms.
\end{theorem}
We divide the proof of Theorem \ref{thm:atom} into three sections. In Section \ref{atsub}, we prove that
$H^{p,q}_\at(X)\subset H^{*,p}(X)$ directly by the definition of $H^{p,q}_\at(X)$. The next two sections
mainly deal with the proof of $H^{*,p}(X)\subset H^{p,q}_{\at}(X)$. In Section \ref{cz}, we obtain a
Calder\'{o}n-Zygmund decomposition for any $f\in H^{*,p}(X)$. Then, in Section \ref{prat}, we show that any
$f\in H^{*,p}(X)$ has a $(p,\fz)$-atomic decomposition. In Section \ref{cw}, we reveal the
equivalent relationship between $H^{p,q}_\at(X)$ and $H^{p,q}_\cw(X)$.


\subsection{Proof of $H^{p,q}_\at(X)\subset H^{*,p}(X)$}\label{atsub}

In this section, we prove $H^{p,q}_\at(X)\subset H^{*,p}(X)$, as subspaces of $(\go{\bz,\gz})'$ with
$\bz,\ \gz\in(\omega(1/p-1),\eta)$. To do this, we need the following technical lemma.
\begin{lemma}\label{lem:a*}
Let $p\in(\om,1]$ and $q\in(p,\fz]\cap[1,\fz]$.  Then there exists a positive constant $C$ such that, for any
$(p,q)$-atom $a$ supported on $B:=B(x_B,r_B)$, with $x_B\in X$ and $r_B\in(0,\fz)$, and any $x\in X$,
\begin{align}\label{eq-add2}
a^*(x)\le C\CM(a)(x)\chi_{B(x_B,2A_0r_B)}(x)+C\lf[\frac{r_B}{d(x_B,x)}\r]^\bz
\frac{[\mu(B)]^{1-\frac 1p}}{V(x_B,x)}\chi_{[B(x_B,2A_0r_B)]^\complement}(x)
\end{align}
and
\begin{align}\label{eq-add3}
\|a^*\|_{L^p(X)}\le C,
\end{align}
where the atom $a$ is viewed as a distribution on $\go{\bz,\gz}$ with $\bz,\ \gz\in(\omega(1/p-1),\eta)$.
\end{lemma}

\begin{proof}
First, we show \eqref{eq-add2}.
Let $\vz\in\go{\bz,\gz}$ be such that $\|\vz\|_{\CG(x,r,\bz,\gz)}\le 1$ for some $r\in(0,\fz)$, where
$\bz,\ \gz\in(\omega(1/p-1),\eta)$.
When $x\in B(x_B,2A_0r_B)$, by Lemma \ref{lem-add}(v), we find that
$$
|\langle a,\vz\rangle|=\lf|\int_X a(y)\vz(y)\,d\mu(y)\r|\le\int_X |a(y)|\frac{1}{V_r(x)+V(x,y)}
\lf[\frac{r}{r+d(x,y)}\r]^\gz\,d\mu(y)\ls\CM(a)(x),
$$
which consequently implies that $a^*(x)\ls\CM(a)(x)$.

Let $x\notin B(x_B,2A_0r_B)$. Then, for any $y\in B$, we have
$d(x,x_B)\ge 2A_0r_B>2A_0d(x_B,y)$. Therefore, by the definition of $(p,q)$-atoms and Definition \ref{def:test}(ii),
we conclude that
\begin{align*}
|\langle a,\vz\rangle|&=\lf|\int_B a(y)\vz(y)\,d\mu(y)\r|\le\int_B |a(y)||\vz(y)-\vz(x_B)|\,d\mu(y)\\
&\le\int_B |a(y)|\lf[\frac{d(x_B,y)}{r+d(x,x_B)}\r]^\bz\frac{1}{V_r(x)+V(x,x_B)}
\lf[\frac{r}{r+d(x,x_B)}\r]^\gz\,d\mu(y)\\
&\le\lf[\frac{r_B}{d(x_B,x)}\r]^\bz\frac{1}{V(x,x_B)}\|a\|_{L^1(X)}
\ls\lf[\frac{r_B}{d(x_B,x)}\r]^\bz\frac{[\mu(B)]^{1-\frac 1p}}{V(x_B,x)}.
\end{align*}
Taking the supremum over all such $\vz\in\go{\bz,\gz}$ satisfying $\|\vz\|_{\CG(x,r,\bz,\gz)}\le 1$
for some $r\in(0,\fz)$, we obtain
\eqref{eq-add2}.

Now, we use \eqref{eq-add2} to show \eqref{eq-add3}.  When  $q\in(1,\fz]$, from the H\"{o}lder
inequality and the boundedness of $\CM$ on $L^{q}(X)$, we deduce that
\begin{align*}
\int_{B(x_B,2A_0r_B)}[\CM(a)(x)]^p\,d\mu(x)
&\le[\mu(B(x_B,2A_0r_B))]^{1-p/q}\|\CM(a)\|_{L^q(X)}^{p}\ls[\mu(B)]^{1-p/q}\|a\|_{L^q(X)}^{p}\ls 1.
\end{align*}
If $q=1$, then, by $p\in(\om,1)$ and the boundedness of $\CM$ from $L^1(X)$ to $L^{1,\fz}(X)$, we conclude that
\begin{align*}
\int_{B(x_B,2A_0r_B)}[\CM(a)(x)]^p\,d\mu(x)&=\int_0^\fz\mu(\{x\in B(x_B,2A_0r_B):\ \CM(a)(x)>\lz\})\,d\lz^p\\
&\ls\int_0^\fz\min\lf\{\mu(B),\frac{\|a\|_{L^1(X)}}{\lz}\r\}\,d\lz^p\\
&\ls\int_0^{\|a\|_{L^1(X)}/\mu(B)}\mu(B)\,d\lz^p+\int_{\|a\|_{L^1(X)}/\mu(B)}^\fz
\|a\|_{L^1(X)}\lz^{-1}\,d\lz^p\\
&\ls\|a\|_{L^1(X)}^p[\mu(B)]^{1-p}\ls 1.
\end{align*}
By the fact $\bz>\omega(1/p-1)$ and the doubling condition \eqref{eq:doub}, we have
\begin{align*}
&\int_{d(x,x_B)\ge 2A_0r_B}\lf[\frac{r_B}{d(x_B,x)}\r]^{\bz p}
\lf[\frac{1}{\mu(B)}\r]^{1-p}\lf[\frac 1{V(x_B,x)}\r]^p\,d\mu(x)\\
&\quad\ls\sum_{k=1}^\fz 2^{-k\bz p}2^{k\omega(1-p)}\int_{2^kA_0r_B\le d(x,x_B)<2^{k+1}A_0r_B}
\frac 1{V(x_B,x)}\,d\mu(x)\ls 1.
\end{align*}
Combining the last three formulae with \eqref{eq-add2}, we obtain \eqref{eq-add3}, which then completes
the proof of Lemma \ref{lem:a*}.
\end{proof}

\begin{proof}[Proof of $H^{p,q}_\at(X)\subset H^{*,p}(X)$]
Assume that  $f\in (\go{\bz,\gz})'$ is non-zero and it belongs to  $H^{p,q}_\at(X)$  with
$\bz,\ \gz\in(\omega(1/p-1),\eta)$.
Then $f=\sum_{j=1}^\fz\lz_ja_j$, where $\{a_j\}_{j=1}^\fz$ are $(p,q)$-atoms and
$\{\lz_j\}_{j=1}^\fz\subset\cc$ satisfy $\sum_{j=1}^{\fz}|\lz_j|^p\sim \|f\|_{H^{p,q}_\at(X)}^p$.
By the definition of the grand
maximal function, we conclude that, for any $x\in X$,
$$
f^*(x)\le\sum_{j=1}^\fz |\lz_j|a_j^*(x).
$$
From this and \eqref{eq-add3}, we deduce that
$$
\|f^*\|_{L^p(X)}^p\ls\sum_{j=1}^\fz|\lz_j|^p\lf\|a_j^*\r\|_{L^p(X)}\ls
\sum_{j=1}^{\fz}|\lz_j|^p \sim \|f\|_{H^{p,q}_\at(X)}^p.
$$
This finishes the proof of
$H^{p,q}_\at(X)\subset H^{*,p}(X)$.
\end{proof}


\subsection{Calder\'{o}n-Zygmund decomposition of a distribution from $H^{*,p}(X)$}\label{cz}

In this section, we obtain a Calder\'{o}n-Zygmund decomposition of any $f\in H^{*,p}(X)$. First we establish
a partition of unity for an open set $\Omega$ with $\mu(\Omega)<\fz$.
\begin{proposition}\label{prop:ozdec}
Suppose $\Omega\subset X$ is a proper open set with $\mu(\Omega)\in(0,\fz)$ and $A\in[1,\fz)$. For any
$x\in\Omega$, let
$$
r(x):=\frac{d(x,\Omega^{\complement})}{2AA_0}\in(0,\fz).
$$
Then there exist $L_0\in\nn$ and a sequence $\{x_k\}_{k\in I}\subset \Omega$, where $I$ is a countable index
set, such that
\begin{enumerate}
\item $\{B(x_k,r_k/(5A_0^3))\}_{k\in I}$ is disjoint. Here and hereafter, $r_k:=r(x_k)$ for any $k\in I$;
\item $\bigcup_{k\in I} B(x_k,r_k)=\Omega$ and $B(x_k,Ar_k)\subset\Omega$;
\item for any $x\in\Omega$, $Ar_k\le d(x,\Omega^\complement)\le3AA_0^2r_k$ whenever $x\in B(x_k,r_k)$ and
$k\in I$;
\item for any $k\in I$, there exists $y_k\notin \Omega$ such that $d(x_k,y_k)<3AA_0r_k$;
\item for any given $k\in I$, the number of balls $B(x_j,Ar_j)$ that intersect $B(x_k,Ar_k)$ is at most
$L_0$;
\item if, in addition, $\Omega$ is bounded, then, for any $\sigma\in(0,\fz)$, the set $\{k\in I:\ r_k>\sigma\}$
is finite.
\end{enumerate}
\end{proposition}
\begin{proof}
We show this proposition by borrowing some ideas from \cite[pp.\ 15--16]{Stein93}. Let $\ez:=(5A_0^3)^{-1}$
and $\{B(x,\ez r(x))\}_{x\in\Omega}$ be a covering of $\Omega$. Now we pick the maximal disjoint
subcollection of $\{B(x,\ez r(x))\}_{x\in\Omega}$, denoted by $\{B_k\}_{k\in I}$, which is at most
countable, because of \eqref{eq:doub} and $\mu(\Omega)\in(0,\fz)$. For any $k\in I$, denote the center of $B_k$
by $x_k$ and $r(x_k)$ by $r_k$. Then we obtain (i).

Properties (iii) and (iv) can be shown by the definition of $r_k$, the details being omitted.
Now we show (ii). Obviously, $B(x_k,Ar_k)\subset\Omega$ for any $k\in I$. It suffices to prove that
$\Omega\subset\bigcup_{k\in I} B(x_k,r_k)$. For any $x\in\Omega$, since $\{B_k\}_{k\in I}$ is maximal, it
then follows that there exists $k\in I$ such that $B(x_k,\ez r_k)\cap B(x,\ez r(x))\neq\emptyset$. We claim that
$r_k\ge r(x)/(4A_0^2)$. If not, then $r_k<r(x)/(4A_0^2)$. Suppose that
$x_0\in B(x_k,\ez r_k)\cap B(x,\ez r(x))$. Then, for any $y\in B(x_k,3AA_0r_k)$, we have
\begin{align*}
d(y,x)&\le A_0[d(y,x_0)+d(x_0,x)]\le A_0^2[d(y,x_k)+d(x_k,x_0)]+A_0d(x_0,x)\le 6AA_0^3r_k+A_0\ez r(x)\\
&\le \frac 32 AA_0r(x)+\frac 15 AA_0r(x)=\frac{17}{10} AA_0r(x)
\end{align*}
and hence $B(x_k,3AA_0r_k)\subset B(x,\frac{17}{10}AA_0r(x))\subset\Omega$, which contradicts to (iv).
This proves the claim. Further, by the fact that $r(x)\le 4A_0^2r_k$, we have
$$
d(x,x_k)\le A_0[d(x,x_0)+d(x_0,x_k)]<A_0\ez r(x)+A_0\ez r_k\le 5A_0^3\ez r_k=r_k,
$$
that is, $x\in B(x_k,r_k)$. This finishes the proof of (ii).

Now we prove (v). Fix $k\in I$. Suppose that $B(x_j,Ar_j)\cap B(x_k,Ar_k)\neq\emptyset$. We claim that
$r_j\le 8A_0^2r_k$. If not, then $r_j>8A_0^2r_k$. Choose $y_0\in B(x_j,Ar_j)\cap B(x_k,Ar_k)$. For any
$y\in B(x_k,3AA_0r_k)$, we have
\begin{align*}
d(y,x_j)&\le A_0[d(y,y_0)+d(y_0,x_j)]\le A_0^2[d(y,x_k)+d(x_k,y_0)]+A_0d(y_0,x_j)\\
&\le 3AA_0^3r_k+AA_0^2r_k+AA_0r_j\le \frac 32 AA_0r_j,
\end{align*}
which further implies that $y\in B(x_j,\frac 32AA_0r_j)$. Therefore,
$B(x_k,3AA_0r_k)\subset B(x_j,\frac 3 2 AA_0r_j)\subset\Omega$, which contradicts to (iv), Thus, we have
$r_j\le 8A_0^2r_k$. By symmetry, we also have $r_k\le 8A_0^2r_j$. Let
$$
\CJ:=\{j\in I:\ B(x_j,Ar_j)\cap B(x_k,Ar_k)\neq\emptyset\}.
$$
Then, for any $j\in\CJ$, $d(x_j,x_k)<AA_0(r_j+r_k)\le 9AA_0^3r_k$, which further implies that
$$
B\lf(x_j,(5A_0^3)^{-1}r_j\r)\subset B\lf(x_k,A_0\lf[d(x_j,x_k)+(5A_0^3)^{-1}r_j\r]\r)
\subset B(x_k,11AA_0^4r_k).
$$
Then, from the fact $d(x_j,x_k)\ls\min\{r_j,r_k\}$ and \eqref{eq:doub}, we deduce that
$$
\mu\lf(B\lf(x_j,(5A_0^3)^{-1}r_j\r)\r)\sim \mu(B(x_j,r_j))\sim\mu(B(x_k,r_k))\sim\mu(B(x_k,11AA_0^4r_k))
$$
with the equivalent positive constants depending on $A$. Thus, we obtain (v) by (i).

Finally we prove (vi). Since $\Omega$ is bounded, it follows that there exist $x_0\in X$ and $R\in(0,\fz)$
such that $\Omega\subset B(x_0,R)$. If (vi) fails, then there exists $\sigma_0\in(0,\fz)$ such that
$\mathcal{K}:=\{k\in I:\ r_k>\sigma_0 R\}$ is infinite. Then, for any $k\in\mathcal K$,
$$
\mu(B(x_k,r_k/(5A_0^3)))\sim\mu(B(x_k,\ez_0 R))\gtrsim\mu(B(x_0,R))\gtrsim \mu(\Omega)>0.
$$
By this and (i), we have $\mu(\Omega)\ge\sum_{k\in\mathcal{K}} \mu(B(x_k,r_k/(5A_0^3)))=\fz$.
That is a contradiction. This proves (vi) and hence finishes the proof of Proposition \ref{prop:ozdec}.
\end{proof}

\begin{proposition}\label{prop:chidec}
Let $\Omega\subset X$ be an open set and $\mu(\Omega)<\fz$. Suppose that sequences $\{x_k\}_{k\in I}$ and
$\{r_k\}_{k\in I}$ are as in Proposition \ref{prop:ozdec} with $A:=16A_0^4$. Then there exist non-negative
functions $\{\phi_k\}_{k\in I}$ such that
\begin{enumerate}
\item for any $k\in I$, $0\le\phi_k\le 1$ and $\supp\phi_k\subset B(x_k,2A_0r_k)$;
\item $\sum_{k\in I} \phi_k=\chi_\Omega$;
\item for any $k\in I$, $\phi_k\ge L_0^{-1}$ in $B(x_k,r_k)$, where $L_0$ is as in Proposition
\ref{prop:ozdec};
\item there exists a positive constant $C$ such that, for any $k\in I$,
$\|\phi_k\|_{\CG(x_k,r_k,\eta,\eta)}\le CV_{r_k}(x_k)$.
\end{enumerate}
\end{proposition}

\begin{proof}
By \cite[Corollary 4.2]{AH13}, for any $k\in I$, we find a function $\psi_k$ such that
$\chi_{B(x_k,r_k)}\le\psi_k\le\chi_{B(x_k,2A_0r_k)}$ and $\|\psi_k\|_{\dot{C}^\eta(X)}\ls r_k^{-\eta}$.
Here and hereafter, for any $s\in(0,\eta]$ and a measurable function $f$, define
$$
\|f\|_{\dot{C}^s(X)}:=\sup_{x\neq y}\frac{|f(x)-f(y)|}{[d(x,y)]^\bz}.
$$
Since
$A\ge 2A_0$, from (ii) and (v) of Proposition \ref{prop:ozdec}, it follows that, for any $x\in\Omega$,
$1\le\sum_{k\in I}\psi_k(x)\le L_0$. For any $k\in I$ and $x\in X$, let
$$
\phi_k(x):=\begin{cases}
\displaystyle\psi_k(x)\lf[\sum_{j\in I}\psi_j(x)\r]^{-1}&\textup{when } x\in\Omega,\\
0,&\textup{when } x\notin\Omega.
\end{cases}
$$
Then, for any $k\in I$, we have $0\le\phi_k\le 1$, $\supp \phi_k\subset B(x_k,2A_0r_k)$ and $\sum_{k\in I}\phi_k(x)=1$
when $x\in\Omega$. Moreover, for any $k\in I$, we  have  $\phi_k\ge L_0^{-1}$ in $B(x_k,r_k)$. Thus, we prove
(i), (ii) and (iii).

It remains to prove (iv). Fix $k\in I$. For any $y\in X$, we have
$$
|\phi_k(y)|\le\chi_{B(x_k,2A_0r_k)}(y)\ls V_{r_k}(x_k)\frac{1}{V_{r_k}(x_k)+V(x_k,y)}
\lf[\frac{r_k}{r_k+d(x_k,y)}\r]^\eta.
$$
Now we prove that $\phi_k$ satisfies the regularity condition. Suppose that
$d(y,y')\le (2A_0)^{-1}[r_k+d(x_k,y)]$.
If $|\phi_k(y)-\phi_k(y')|\neq 0$, then $d(x_k,y)< (3A_0)^2r_k$. If not, then $d(x_k,y)\ge (3A_0)^2r_k$, so that $\phi_k(y)=0$  and
$$
d(y',x_k)\ge A_0^{-1}d(x_k,y)-d(y,y')\ge (2A_0)^{-1}d(x_k,y)-(2A_0)^{-1}r_k>2A_0r_k
$$
and hence $\phi_k(y')=0$, which contradicts to $|\phi_k(y)-\phi_k(y')|\neq 0$.
Notice that $\psi_k(y')|\psi_j(y)-\psi_j(y')|\neq 0$ implies that $y'\in B(x_k,2A_0r_k)$ and
also $y$ or $y'$ belongs to $B(x_j,2A_0r_j)$, which further implies that
$B(x_k,Ar_k)\cap B(x_j,Ar_j)\neq\emptyset$. Then, by the proof of Proposition \ref{prop:ozdec}(v), the number of $j$ satisfying $\psi_k(y')|\psi_j(y)-\psi_j(y')|\neq 0$ is at most $L_0$ and $r_j\sim r_k$.
Therefore,
\begin{align*}
|\phi_k(y)-\phi_k(y')|&=\lf|\frac{\psi_k(y)}{\sum_{j\in I}\psi_j(y)}
-\frac{\psi_k(y')}{\sum_{j\in I}\psi_j(y')}\r|\\
&\le\frac{|\psi_k(y)-\psi_k(y')|}{\sum_{j\in I}\psi_j(y)}+
\frac{\psi_k(y')\sum_{j\in I}|\psi_j(y)-\psi_j(y')|}{[\sum_{j\in I}\psi_j(y)][\sum_{j\in I}\psi_j(y')]}\\
&\ls\lf[\frac{d(y,y')}{r_k}\r]^{\eta}+\sum_{\{j\in I:\ B(x_k,Ar_k)\cap B(x_j,Ar_j)\neq\emptyset\}}
\lf[\frac{d(y,y')}{r_j}\r]^\eta\\
&\ls\lf[\frac{d(y,y')}{r_k}\r]^\eta\sim V_{r_k}(x_k)\lf[\frac{d(y,y')}{r_k+d(x_k,y)}\r]^\eta
\frac{1}{V_{r_k}(x_k)+V(x_k,y)}\lf[\frac{r_k}{r_k+d(x_k,y)}\r]^\eta.
\end{align*}
Then we obtain the desired regularity condition of $\phi_k$. This finishes the proof of (iv) and hence of
Proposition \ref{prop:chidec}.
\end{proof}

Assume that $f\in (\go{\bz,\gz})'$ belongs to $f\in H^{*,p}(X)$, where $p\in (\omega/(\omega+\eta), 1]$ and $\bz,\ \gz\in(\omega(1/p-1),\eta)$.
To obtain the Calder\'{o}n-Zygmund decomposition of $f$,
we apply Propositions \ref{prop:ozdec} and \ref{prop:chidec} to the
level set $\{x\in X:\ f^*(x)>\lz\}$ with $\lz\in(0,\infty)$. The encountering problem is that such a level set
may not be  open even in the case that $d$ is a metric. To solve this problem
in the case that $d$ is a metric, a variant of the notion of the space of test functions is adopted  in
\cite[Definition 2.5]{GLY08} so that to ensure that the level set is open (see \cite[Remark 2.9]{GLY08}).
Here, we borrow some idea from \cite{GLY08}.

By the proof of
\cite[Theorem 2]{MS79}, we know that there exist $\thz\in(0,1)$ and a metric $d'$ such that $d'\sim d^\thz$.
For any $x\in X$ and $r\in(0,\fz)$, define the \emph{$d'$-ball} $B'(x,r):=\{y\in X:\ d'(x,y)<r\}$.
Then $(X,d',\mu)$ is a doubling metric measure space. Moreover, for any
$x,\ y\in X$ and $r\in(0,\fz)$, we have
$$
\mu(B(y,r+d(x,y)))\sim\mu\lf(B'\lf(y,\lf[r+d(x,y)\r]^\thz\r)\r)\sim \mu\lf(B'\lf(y,r^\thz+d'(x,y)\r)\r),
$$
where the equivalent positive constants are independent of $x$ and $r$.
Using the metric $d'$, we introduce a variant of the space of test functions as follows.

\begin{definition}
For any $x\in X,\ \rho\in(0,\fz)$ and
$\bz',\ \gz'\in(0,\fz)$,  define $G(x,\rho,\bz',\gz')$ to be the set of all functions $f$ satisfying that there
exists a positive constant $C$ such that
\begin{enumerate}
\item (the \emph{size condition}) for any $y\in X$,
$$
|f(y)|\le C\frac 1{\mu(B'(y,\rho+d'(x,y)))}\lf[\frac{\rho}{\rho+d'(x,y)}\r]^{\gz'};
$$
\item (the \emph{regularity condition}) for any $y,\ y'\in X$ satisfying $d(y,y')\le[\rho+d'(x,y)]/2$, then
$$
|f(y)-f(y')|\le C\lf[\frac{d'(y,y')}{\rho+d'(y,y')}\r]^{\bz'}\frac 1{\mu(B'(y,\rho+d'(x,y)))}
\lf[\frac{\rho}{\rho+d'(x,y)}\r]^{\gz'}.
$$
\end{enumerate}
Also, define
$$
\|f\|_{G(x,\rho,\bz',\gz')}:=\inf\{C\in(0,\fz):\ \textup{(i) and (ii) hold true}\}.
$$
\end{definition}

By the previous argument, we find that $\CG(x,r,\bz,\gz)=G(x,r^\thz,\bz/\thz,\gz/\thz)$ with equivalent
norms, where the equivalent positive constants are independent of $x$ and $r$. For any
$\bz,\ \gz\in(0,\eta)$ and $f\in(\go{\bz,\gz})'$, define the
\emph{modified grand maximal function} of $f$ by setting, for any $x\in X$,
$$
f^\star(x):=\sup\lf\{\langle f,\vz\rangle:\ \vz\in\go{\bz,\gz}\textup{ with }
\|\vz\|_{G(x,r^\thz,\bz/\thz,\gz/\thz)}\le 1\textup{ for some } r\in(0,\fz)\r\}.
$$
Then $f^\star\sim f^*$ pointwisely on $X$. For any $\lz\in(0,\fz)$ and $j\in\zz$,
define
$$
\Omega_\lz:=\{x\in X:\ f^\star(x)>\lz\}\qquad \textup{and}\qquad \Omega^j:=\Omega_{2^j}.
$$
By the argument used in \cite[Remark 2.9(ii)]{GLY08}, we find that $\Omega_\lz$ is \emph{open} under the
topology induced by $d'$, so is it under the topology induced by $d$.

Now suppose that $p\in(\om,1]$, $\bz,\ \gz\in(\omega(1/p-1),\eta)$ and $f\in H^{*,p}(X)$.
Then $f^\star\in L^p(X)$ and every $\Omega^j$ with $j\in\zz$ has finite measure. Consequently, there exist
$\{x^j_k\}_{k\in I_j}\subset X$ with $I_j$ being a countable index set, $\{r^j_k\}_{k\in I_j}\subset(0,\fz)$,
$L_0\in\nn$ and a sequence $\{\phi^j_k\}_{k\in I_j}$ of non-negative functions satisfying all the conclusions
of Propositions \ref{prop:ozdec} and \ref{prop:chidec}. For any $j\in\zz$ and $k\in I_j$, define $\Phi^j_k$ by
setting, for any $\vz\in\go{\bz,\gz}$ and $x\in X$,
$$
\Phi^j_k(\vz)(x):=\phi^j_k(x)\lf[\int_X\phi^j_k(z)\,d\mu(z)\r]^{-1}\int_X[\vz(x)-\vz(z)]\phi^j_k(z)\,d\mu(z).
$$
It can be seen that $\Phi^j_k$ is bounded on $\go{\bz,\gz}$ with operator norm depending on $j$ and
$k$; see \cite[Lemma 4.9]{GLY08}. Thus, it makes sense to define a distribution $b^j_k$ on $\go{\bz,\gz}$ by
setting, for any $\vz\in\go{\bz,\gz}$,
\begin{equation}\label{eq:defbjk}
\lf\langle b^j_k,\vz\r\rangle:=\lf\langle f,\Phi^j_k(\vz)\r\rangle.
\end{equation}
To estimate $(b^j_k)^*$, we have the following result. For its proof, see, for example, \cite[Lemma 3.7]{Li98}.
\begin{proposition}\label{prop:bjk*}
For any $j\in\zz$ and $k\in I_j$, $b^j_k$ is defined as in \eqref{eq:defbjk}. Then there exists a positive
constant $C$ such that, for any $j\in\zz$, $k\in I_j$ and $x\in X$,
$$
\lf(b^j_k\r)^*(x)\le C2^j\frac{\mu(B(x^j_k,r^j_k))}{\mu(B(x^j_k,r^j_k))+V(x^j_k,x)}
\lf[\frac{r^j_k}{r^j_k+d(x^j_k,x)}\r]^\bz\chi_{[B(x^j_k,16A_0^4r^j_k)]^\complement}(x)
+Cf^*(x)\chi_{B(x^j_k,16A_0^4r^j_k)}(x).
$$
\end{proposition}
The next lemma is exactly \cite[Lemma 4.10]{GLY08}. The proof remains true if $d$ is a quasi-metric and $\mu$
does not satisfy the reverse doubling condition.
\begin{lemma}\label{lem:Lpest}
Let $\bz\in(0,\fz)$, $p\in(\omega/(\omega+\bz),\fz)$, $L_0\in\nn$ and $I$ be a countable index set.
Then there exists a positive constant $C$ such that, for any sequences $\{x_k\}_{k\in I}\subset X$ and
$\{r_k\}_{k\in I}\subset(0,\fz)$ satisfying $\sum_{k\in I}\chi_{B(x_k,r_k)}\le L_0$,
$$
\int_X \lf\{\sum_{k\in I}\frac{V_{r_k}(x_k)}{V_{r_k}(x_k)+V(x_k,x)}\lf[\frac{r_k}{r_k+d(x_k,x)}\r]^\bz\r\}^p
\,d\mu(x)\le C\mu\lf(\bigcup_{k\in I}B(x_k,r_k)\r).
$$
\end{lemma}
Then, by Proposition \ref{prop:bjk*} and Lemma \ref{lem:Lpest}, we have the following result.
\begin{proposition}\label{prop:bgj}
Let $p\in(\om,1]$. For any $j\in\zz$ and $k\in I_j$, let $b^j_k$ be as in
\eqref{eq:defbjk}. Then there exists a positive constant $C$ such that, for any $j\in\zz$,
\begin{equation}\label{4.1x}
\int_X\sum_{k\in I_j}\lf[\lf(b^j_k\r)^*(x)\r]^p\,d\mu(x)\le C\lf\|f^*\chi_{\Omega^j}\r\|_{L^p(X)}^p;
\end{equation}
moreover, there exists $b^j\in H^{*,p}(X)$ such that $b^j=\sum_{k\in I_j} b^j_k$ in $H^{*,p}(X)$ and, for any
$x\in X$,
\begin{equation}\label{eq:bj*}
(b^j)^*(x)\le C2^j\sum_{k\in I_j}\frac{\mu(B(x^j_k,r^j_k))}{\mu(B(x^j_k,r^j_k))+V(x^j_k,x)}
\lf[\frac{r^j_k}{r^j_k+d(x^j_k,x)}\r]^\bz+Cf^*(x)\chi_{\Omega^j}(x);
\end{equation}
if $g^j:=f-b^j$ for any $j\in\zz$, then, for any $x\in X$,
\begin{equation}\label{eq:gj*}
(g^j)^*(x)\le C2^j\sum_{k\in I_j}\frac{\mu(B(x^j_k,r^j_k))}{\mu(B(x^j_k,r^j_k))+V(x^j_k,x)}
\lf[\frac{r^j_k}{r^j_k+d(x^j_k,x)}\r]^\bz+Cf^*(x)\chi_{(\Omega^j)^\complement}(x).
\end{equation}
\end{proposition}
\begin{proof}	
Fix $j\in\zz$. We first prove \eqref{4.1x}. Indeed, by Proposition \ref{prop:bjk*}, we find that
\begin{align*}
\int_X\sum_{k\in I_j}\lf[\lf(b^j_k\r)^*(x)\r]^p\,d\mu(x)&\ls 2^{jp}\int_X\sum_{k\in I_j}
\lf\{\frac{\mu(B(x^j_k,r^j_k))}{\mu(B(x^j_k,r^j_k))+V(x^j_k,x)}\lf[\frac{r^j_k}{r^j_k+d(x^j_k,x)}\r]^\bz\r\}^p
\,d\mu(x)\\
&\quad+\int_{\bigcup_{k\in I_j} B(x^j_k,16A_0^4r^j_k)} [f^*(x)]^p\,d\mu(x).
\end{align*}
By Proposition \ref{prop:ozdec}(ii), we have
$
\Omega^j=\bigcup_{k\in I_j} B(x^j_k,16A_0^4r^j_k).
$
Applying this and Lemma \ref{lem:Lpest}, the first term in  the right-hand side of the above formula is
bounded by a harmlessly positive constant multiple of $2^{jp}\mu(\Omega^j)$. Combining this with $f^\ast\sim
f^\star$ implies that
$$
\int_X\sum_{k\in I_j}\lf[\lf(b^j_k\r)^*(x)\r]^p\,d\mu(x)\ls 2^{jp}\mu\lf(\Omega^j\r)+\int_{\Omega^j}[f^*(x)]^p\,d\mu(x)
\ls\lf\|f^\ast \chi_{\Omega^j}\r\|_{L^p(X)}^p,
$$
which proves \eqref{4.1x}.

Next we prove \eqref{eq:bj*}. By \eqref{4.1x}, the dominated convergence theorem and the completeness of
$H^{*,p}(X)$ (see Proposition \ref{prop:ban}), we know that there exists $b^j\in H^{*,p}(X)$ such that
$b^j=\sum_{k\in I_j} b^j_k$ in $H^{*,p}(X)$. Moreover, from Proposition \ref{prop:bjk*} and
$\Omega^j=\bigcup_{k\in I_j} B(x^j_k,16A_0^4r^j_k)$,  we deduce that, for any $x\in X$,
$$
(b^j)^*(x)\le\sum_{k\in I_j}\lf(b^j_k\r)^*(x)\ls 2^j\sum_{k\in I_j}
\frac{\mu(B(x^j_k,r^j_k))}{\mu(B(x^j_k,r^j_k))+V(x^j_k,x)}\lf[\frac{r^j_k}{r^j_k+d(x^j_k,x)}\r]^\bz
+f^*(x)\chi_{\Omega^j}(x).
$$
This finishes the proof of \eqref{eq:bj*}.

It remains to prove \eqref{eq:gj*}. If $x\in(\Omega^j)^\complement$, then, by \eqref{eq:bj*}, we conclude that
$$
(g^j)^*(x)\le f^*(x)+(b^j)^*(x)\ls 2^j\sum_{k\in I_j}\frac{\mu(B(x^j_k,r^j_k))}{\mu(B(x^j_k,r^j_k))+V(x^j_k,x)}
\lf[\frac{r^j_k}{r^j_k+d(x^j_k,x)}\r]^\bz+f^*(x),
$$
as desired.

Now we consider the case $x\in\Omega^j$.  According to Proposition \ref{prop:ozdec}(v), for any $n\in I_j$, we
choose a point $y^j_n\notin\Omega^j$
satisfying $32A_0^5r^j_n\le d(x^j_n,y^j_n)<48A_0^5r^j_n$. Since $x\in\Omega^j$, it follows that there exists
$k_0\in I_j$ such that $x\in B(x^j_{k_0},r^j_{k_0})$. Let $\CJ$ be the set of all $n\in I_j$ such that
$B(x^j_n,16A_0^4r^j_n)\cap B(x^j_{k_0},16A_0^4r^j_{k_0})\neq\emptyset$. Then, by the proof of Proposition
\ref{prop:ozdec}(v), $\#\CJ\le L_0$ and $r^j_n\sim r^j_{k_0}$ whenever $n\in\CJ$.

Suppose that $\vz\in\go{\bz,\gz}$ with
$\|\vz\|_{\CG(x,r,\bz,\gz)}\le 1$ for some $r\in(0,\fz)$.
We then estimate $\langle g^j,\vz\rangle $  by considering the  cases $r\le r^j_{k_0}$ and $r> r^j_{k_0}$,
respectively.

{\it Case 1) $r\le r^j_{k_0}$.} In this case, we write
$$
\langle g^j,\vz\rangle=\langle f,\vz\rangle-\sum_{n\in I_j}\langle b^j_n,\vz\rangle
=\langle f,\vz\rangle-\sum_{n\in\CJ}\langle b^j_n,\vz\rangle-\sum_{n\notin\CJ}\langle b^j_n,\vz\rangle
=\lf\langle f,\wz{\vz}\r\rangle-\sum_{n\in\CJ}\lf\langle f,\wz{\vz}_n\r\rangle
-\sum_{n\notin\CJ}\langle b^j_n,\vz\rangle,
$$
where $\wz{\vz}:=(1-\sum_{n\in\CJ}\phi^j_n)\vz$ and, for any $n\in\CJ$,
$$
\wz{\vz}_n:=\phi^j_n\lf[\int_X\phi^j_n(z)\,d\mu(z)\r]^{-1}\int_X\vz(z)\phi^j_n(z)\,d\mu(z).
$$

We first consider the term $\sum_{n\notin\CJ}\langle b^j_n,\vz\rangle$. Indeed, from $x\in B(x^j_{k_0},r^j_{k_0})$, it follows that $x\notin B(x^j_n,16A_0^4x^j_n)$ when $n\notin\CJ$. Applying Proposition \ref{prop:bjk*} implies that
$$
\lf|\left\langle b^j_n,\vz\right\rangle\r|\le \lf|\lf(b^j_n\r)^*(x)\r|\ls 2^j
\frac{\mu(B(x^j_n,r^j_n))}{\mu(B(x^j_n,r^j_n))+V(x^j_n,x)}\lf[\frac{r^j_n}{r^j_n+d(x^j_n,x)}\r]^\bz,
$$
and hence
\begin{align*}
\sum_{n\notin\CJ}\lf|\left\langle b^j_n,\vz\right\rangle\r|\ls 2^j\sum_{n\notin\CJ}
\frac{\mu(B(x^j_n,r^j_n))}{\mu(B(x^j_n,r^j_n))+V(x^j_n,x)}\lf[\frac{r^j_n}{r^j_n+d(x^j_n,x)}\r]^\bz,
\end{align*}
 as desired.

Next we consider the term $\sum_{n\in\CJ}\lf\langle f,\wz{\vz}_n\r\rangle$.
Notice that $\|\wz\vz_n\|_{\CG(x^j_n,r^j_n,\bz,\gz)}\ls 1$. By $d(x^j_n, y^j_n)\sim r^j_n$, we then have
$\|\wz\vz_n\|_{\CG(y^j_n,r^j_n,\bz,\gz)}\ls 1$. Therefore,
$$
\lf|\lf\langle f,\wz{\vz}_n\r\rangle\r|\ls f^*\lf(y^j_n\r)\sim f^\star\lf(y^j_n\r)\ls 2^j
\sim 2^j\frac{\mu(B(x^j_n,r^j_n))}{\mu(B(x^j_n,r^j_n))+V(x^j_n,x)}\lf[\frac{r^j_n}{r^j_n+d(x^j_n,x)}\r]^\bz,
$$
where, in the last step, we used the facts that  $x\in B(x^j_{k_0},r^j_{k_0})$ and $d(x^j_n,x^j_{k_0})\ls
r^j_n+ r^j_{k_0}\sim r^j_n$ whenever $n\in\CJ$. Then, summing all $n\in\CJ$, we obtain the desired estimate.

Finally, we consider the  term $\langle f,\wz{\vz}\rangle$. Since ${\vz}\in\go{\bz,\gz}$, it is easy to see that
$\wz{\vz}\in\go{\bz,\gz}$. Once we have proved that
\begin{equation}\label{4.x}
\lf\|\wz{\vz}\r\|_{\CG(y^j_{k_0},r^j_{k_0},\bz,\gz)}\ls 1,
\end{equation}
then
$$
\lf|\lf\langle f,\wz{\vz}\r\rangle\r|\ls f^*\lf(y^j_{k_0}\r)\sim f^\star\lf(y^j_{k_0}\r)\ls 2^j
\sim 2^j\frac{\mu(B(x^j_{k_0},r^j_{k_0}))}{\mu(B(x^j_{k_0},r^j_{k_0}))+V(x^j_{k_0},x)}
\lf[\frac{r^j_{k_0}}{r^j_{k_0}+d(x^j_{k_0},x)}\r]^\bz,
$$
as desired.

To prove \eqref{4.x}, we first consider the size condition.
For any $z\in B(x^j_{k_0},16A_0^4r^j_{k_0})$, by  Proposition \ref{prop:chidec},  we have
$\sum_{n\in\CJ}\phi^j_n(z)=\sum_{n\in I_j}\phi^j_n(z)=1$ and hence $\wz{\vz}(z)=0$. When
$d(z,x^j_{k_0})\ge 16A_0^4r^j_{k_0}$, by the fact $d(x^j_{k_0},z)\ge 2A_0d(x,x^j_{k_0})$, we have
\begin{align}\label{eq-x1}
r^j_{k_0}+d\lf(z,y^j_{k_0}\r)&\le r^j_{k_0}+A_0\lf[d\lf(z,x^j_{k_0}\r)+d\lf(x^j_{k_0},y^j_{k_0}\r)\r]
\le (2A_0)^6\lf[r^j_{k_0}+d\lf(z,x^j_{k_0}\r)\r]\\
&\le(2A_0)^7d\lf(z,x^j_{k_0}\r)\le(2A_0)^8d(x,z)\le(2A_0)^8[r+d(x,z)]\notag
\end{align}
and hence $\mu(B(y^j_{k_0},r^j_{k_0}))+V(y^j_{k_0},z)\ls V_r(x)+V(x,z)$,
which, together with the size condition of $\vz$ and the fact
that $r\le r^j_{k_0}$, further implies that
$$
\lf|\wz{\vz}(z)\r|\le|\vz(z)|\le\frac 1{V_r(x)+V(x,z)}\lf[\frac r{r+d(x,z)}\r]^\gz
\ls\frac 1{\mu(B(y^j_{k_0},r^j_{k_0}))+V(y^j_{k_0},z)}\lf[\frac {r^j_{k_0}}{r^j_{k_0}+d(y^j_{k_0},z)}\r]^\gz.
$$
This finishes the proof of the size condition.

Now we consider the regularity of $\wz\vz$. Suppose that $z,\ z'\in X$ with
$d(z,z')\le (2A_0)^{-1}[r^j_{k_0}+d(z,y^j_{k_0})]$. Due to the size condition, we only need to consider the case
$d(z,z')\le (2A_0)^{-9}[r^j_{k_0}+d(z,y^j_{k_0})]$. If
$\wz\vz(z)-\wz\vz(z')\neq 0$, then either  $d(z,x^j_{k_0})\ge 16A_0^4r^j_{k_0}$ or  $d(z',x^j_{k_0})\ge 16A_0^4r^j_{k_0}$,
which always implies that  $d(z,x^j_{k_0})\ge 8A_0^3r^j_{k_0}$.

Indeed, if $d(z,x^j_{k_0})<8A_0^3r^j_{k_0}$, then $d(z,y^j_{k_0})\le A_0[d(z,x^j_{k_0})+d(x^j_{k_0},y^j_{k_0})]<(2A_0)^6r^j_{k_0}$
and hence $d(z,z')\le (2A_0)^3r^j_{k_0}$, which further implies that
$d(z',x^j_{k_0})\le A_0[d(z',z)+d(z,x^j_{k_0})]<16A_0^4r^j_{k_0}$ and it is a contraction.

Notice that $d(z,x^j_{k_0})\ge 8A_0^3r^j_{k_0}$, which, together with an argument as in the estimation of
\eqref{eq-x1}, implies $r^j_{k_0}+d(z,y^j_{k_0})\le(2A_0)^8[r+d(z,x)]$, so that $d(z,z')\le
(2A_0)^{-1}[r+d(z,x)]$. By the definition of $\wz\vz$, we find that
\begin{align*}
\lf|\wz\vz(z)-\wz\vz(z')\r|
&\le \lf(1-\sum_{n\in\CJ}\phi^j_n(z)\r)|\vz(z)-\vz(z')|+
|\vz(z')|\sum_{n\in\CJ}\lf|\phi^j_n(z)-\phi^j_n(z')\r|.
\end{align*}

Using the regularity condition of $\vz$ and the fact $d(z,z')\le (2A_0)^{-1}[r+d(z,x)]$, we obtain
\begin{align*}
\lf(1-\sum_{n\in\CJ}\phi^j_n(z)\r)|\vz(z)-\vz(z')|
&\ls \lf[\frac{d(z,z')}{r+d(z,x)}\r]^\beta\frac 1{V_r(x)+V(x,z)}\lf[\frac r{r+d(x,z)}\r]^\gz\\
&
\ls \lf[\frac{d(z,z')}{r^j_{k_0}+d(z,y^j_{k_0})}\r]^\beta\frac 1{\mu(B(y^j_{k_0},r^j_{k_0}))+V(y^j_{k_0},z)}\lf[\frac {r^j_{k_0}}{r^j_{k_0}+d(y^j_{k_0},z)}\r]^\gz,
\end{align*}
where, in the last step, we used $r^j_{k_0}+d(z,y^j_{k_0})\ls r+d(z,x)$, $r\le r^j_{k_0}$, $x\in
B(x^j_{k_0},r^j_{k_0})$ and $d(y^j_{k_0},x^j_{k_0})\sim r^j_{k_0}$.

We now estimate $|\vz(z')|\sum_{n\in\CJ}|\phi^j_n(z)-\phi^j_n(z')|$.
If $\vz(z')|\phi^j_n(z)-\phi^j_n(z')|\neq 0$, then $z'\notin B(x^j_{k_0},16A_0^4r^j_{k_0})$ and
$z$ or $z'$ belongs to $B(x^j_n,2A_0r^j_n)$.
When $n\in\CJ$, we have $r^j_n\sim r^j_{k_0}\sim r^j_{k_0}+d(y^j_{k_0},z)$.
Also, $r^j_{k_0}+d(z,y^j_{k_0})\ls r+d(z,x)\sim r+d(z',x)$.
By these, $\#\CJ\le L_0$ and $r\le r^j_{k_0}$, we
conclude that
\begin{align*}
|\vz(z')|\sum_{n\in\CJ}\lf|\phi^j_n(z)-\phi^j_n(z')\r|
&\ls \frac{1}{V_r(x)+V(x,z')}\lf[\frac{r}{r+d(z,x)}\r]^\gz\sum_{n\in\CJ}
\lf[\frac{d(z,z')}{r^j_n}\r]^\eta\\
&\ls\lf[\frac{d(z,z')}{r^j_{k_0}+d(y^j_{k_0},z)}\r]^\bz\frac 1{\mu(B(y^j_{k_0},r^j_{k_0}))+V(y^j_{k_0},z)}
\lf[\frac {r^j_{k_0}}{r^j_{k_0}+d(y^j_{k_0},z)}\r]^\gz.
\end{align*}
This finishes the proof of the regularity condition  and hence of \eqref{4.x}. Thus, we complete the proof of
Case 1).

{\it Case 2) $r>r^j_{k_0}$.} In this case, we write
$$
\lf|\left\langle g^j,\vz\right\rangle\r|\le|\langle f,\vz\rangle|
+\sum_{n\in\CJ}\lf|\left\langle b^j_n,\vz\right\rangle \r|
+\sum_{n\notin\CJ}\lf|\left\langle b^j_n,\vz\right\rangle \r|.
$$
The estimation of $\sum_{n\notin\CJ}|\langle b^j_n,\vz\rangle |$ has already been given in Case 1).

From  $x\in B(x^j_{k_0},r^j_{k_0})$ and $d(y^j_{k_0},x^j_{k_0})\sim r^j_{k_0}\ls r$, it follows that
$\|\vz\|_{\CG(y^j_{k_0},r,\bz,\gz)}\ls 1$ and hence
$$
|\langle f,\vz\rangle|\ls f^*\lf(y^j_{k_0}\r)\ls 2^j\sim 2^j
\frac{\mu(B(x^j_{k_0},r^j_{k_0}))}{\mu(B(x^j_{k_0},r^j_{k_0}))+V(x^j_{k_0},x)}
\lf[\frac{r^j_{k_0}}{r^j_{k_0}+d(x^j_{k_0},x)}\r]^\bz.
$$

If $n\in\CJ$, then $r^j_n\sim r^j_{k_0}$ and hence $d(y^j_n,x^j_{k_0})\ls r^j_{k_0}$. This,
together with the fact $r^j_{k_0}<r$ and $x\in B(x^j_{k_0},r^j_{k_0})$, implies that
$\|\vz\|_{\CG(y^j_n,r,\bz,\gz)}\ls 1$. Thus, by Proposition \ref{prop:bjk*}, we have
\begin{equation*}
\sum_{n\in\CJ}\lf|\lf\langle b^j_n,\vz\r\rangle\r|\ls\sum_{n\in\CJ}\lf(b^j_n\r)^*\lf(y^j_n\r)
\ls 2^j\sum_{n\in\CJ}\frac{\mu(B(x^j_n,r^j_n))}{\mu(B(x^j_n,r^j_n))+V(x^j_n,x)}
\lf[\frac{r^j_n}{r^j_n+d(x^j_n,x)}\r]^\bz.
\end{equation*}
Then we obtain the desired estimate for  $\langle g^j,\vz\rangle $  in the case  $r>r^j_{k_0}$.

Combining the two cases above, we find that, for any $x\in\Omega^j$,
$$
(g^j)^*(x)\ls 2^j\sum_{k\in I_j}\frac{\mu(B(x^j_k,r^j_k))}{\mu(B(x^j_k,r^j_k))+V(x^j_k,x)}
\lf[\frac{r^j_k}{r^j_k+d(x^j_k,x)}\r]^\bz.
$$
Thus, \eqref{eq:gj*} holds true. This finishes the proof of Proposition \ref{prop:bgj}.
\end{proof}


\subsection{Atomic characterization of $H^{*,p}(X)$}\label{prat}

In this section, we prove $H^{*,p}(X)\subset H^{p,q}_\at(X)$ and complete the proof of Theorem \ref{thm:atom}.
First, we obtain dense subspaces of $H^{*,p}(X)$.
\begin{lemma}[{\cite[Proposition 4.12]{GLY08}}]\label{lem:dense}
Let $p\in(\om,1]$, $\bz,\ \gz\in(\omega(1/p-1),\eta)$ and $q\in[1,\fz)$.
If regard $H^{*,p}(X)$ as a subspace of $(\go{\bz,\gz})'$, then $L^q(X)\cap H^{*,p}(X)$ is dense in
$H^{*,p}(X)$.
\end{lemma}
In the next two lemmas, we suppose that $f\in L^2(X)\cap H^{*,p}(X)$. Based on Proposition
\ref{prop:HsubL} and \eqref{eq-xxx}, we may follow \cite[Remark 4.14]{GLY08} and assume that there exists a
positive constant $C$ such that, for any $x\in X$, $|f(x)|\le Cf^*(x)$.
With all the notation as in the previous section, for any $j\in\zz$ and $k\in I_j$, define
\begin{equation}\label{x}
m^j_k:=\frac 1{\|\phi^j_k\|_{L^1(X)}}\int_X f(\xi)\phi^j_k(\xi)\,d\mu(\xi)\quad\textup{and}\quad
b^j_k:=\lf(f-m^j_k\r)\phi^j_k.
\end{equation}
Then we have the following technical lemma.
\begin{lemma}[{\cite[Proposition 4.13]{GLY08}}]\label{lem:funbgj}
For any $j\in\zz$ and $k\in I_j$, let $m^j_k$ and $b^j_k$ be  as in \eqref{x}. Then
\begin{enumerate}
\item there exists a positive constant $C$, independent of $j$ and $k\in I_j$, such that $|m^j_k|\le C2^j$;
\item $b^j_k$ induces the same distribution as defined in \eqref{eq:defbjk};
\item $\sum_{k\in I_j} b^j_k$ converges to some function $b^j$ in $L^2(X)$, which induces a distribution that
coincides with $b^j$ as in Proposition \ref{prop:bgj};
\item let $g^j:=f-b^j$. Then $g^j=f\chi_{(\Omega^j)^\complement}+\sum_{k\in I_j}m^j_k\phi^j_k$. Moreover,
there exists a positive constant $C$, independent of $j$, such that, for any $x\in X$, $|g^j(x)|\le C2^j$.
\end{enumerate}
\end{lemma}
For any $j\in\zz$, $k\in I_j$ and $l\in I_{j+1}$, define
\begin{equation}\label{x1}
 L^{j+1}_{k,l}:=\frac 1{\|\phi^{j+1}_l\|_{L^1(X)}}\int_X \lf[f(\xi)-m^{j+1}_l\r]\phi^j_k(\xi)
\phi^{j+1}_l(\xi)\,d\mu(\xi)
\end{equation}
Then $ L^{j+1}_{k,l}$ has the following properties.
\begin{lemma}\label{lem:Psi}
For any $j\in\zz$, $k\in I_j$ and $l\in I_{j+1}$, let $ L^{j+1}_{k,l}$ be as in \eqref{x1}. Then
\begin{enumerate}
\item there exists a positive constant $C$, independent of $j$, $k$ and $l$, such that
$$
\sup_{x\in X}\lf| L^{j+1}_{k,l}\phi^{j+1}_l(x)\r|\le C2^j;
$$
\item $\sum_{k\in I_j}\sum_{l\in I_{j+1}} L^{j+1}_{k,l}\phi^{j+1}_l=0$ both in $(\go{\bz,\gz})'$ and
everywhere.
\end{enumerate}
\end{lemma}
\begin{proof}
We first show (i). Indeed, for any $j\in\zz$, $k\in I_j$, $l\in I_{j+1}$ and $x\in X$,
$$
\lf| L^{j+1}_{k,l}\phi^{j+1}_l(x)\r|\le\lf|m^{j+1}_l\r|\phi^{j+1}_l(x)
+\phi^{j+1}_l(x)\lf|\int_Xf(\xi)\frac{\phi^j_k(\xi)\phi^{j+1}_l(\xi)}{\|\phi^{j+1}_l\|_{L^1(X)}}\,d\mu(\xi)\r|
=:\RY_1+\RY_2.
$$
By Lemma \ref{lem:funbgj}(i) and the definition of $\phi^{j+1}_l$, it is easy to obtain $\RY_1\ls 2^j$.

Now we consider $\RY_2$. If $\phi^j_k\phi^{j+1}_l$ is a non-zero function, then $B(x^j_k,2A_0r^j_k)\cap
B(x^{j+1}_l,2A_0r^{j+1}_l)\neq\emptyset$, which further implies that $r^{j+1}_l\le 3A_0r^j_k$. Otherwise, if
$r^{j+1}_l>3A_0r^j_k$,
then, for any $y\in B(x^j_k,48A_0^5r^j_k)$,
\begin{align*}
d\lf(y,x^{j+1}_l\r)&\le A_0\lf[d\lf(y,x^{j}_k\r)+d\lf(x^{j}_k,x^{j+1}_l\r)\r]
<48A_0^6r^j_k+A_0^2\lf(2A_0r^j_k+2A_0r^{j+1}_l\r)\\
&<16A_0^5r^{j+1}_l+\frac 23 A_0^2r^{j+1}_l+2A_0^3r^{j+1}_l<20A_0^5r^{j+1}_l,
\end{align*}
which implies that $B(x^j_k,48A_0^5r^j_k)\subset B(x^{j+1}_l,20A_0^5r^{j+1}_l)\subset \Omega^{j+1}\subset
\Omega^j$ and hence contradicts to Proposition \ref{prop:ozdec}(v).

Define $\vz:=\phi^j_k\phi^{j+1}_l/\|\phi^{j+1}_l\|_{L^1(X)}$.
According to Proposition \ref{prop:ozdec}(iv) with $A:=16A_0^4$, we can choose
$y^{j+1}_l\in(\Omega^{j+1})^\complement$ such that $d(y^{j+1}_l,x^{j+1}_l)\le 48A_0^5r^{j+1}_l$.
We now show $\vz\in\CG(y^{j+1}_l,r^{j+1}_l,\eta,\eta)$ and $\|\vz\|_{\CG(y^{j+1}_l,r^{j+1}_l,\eta,\eta)}\ls 1$.
Notice that $\supp \vz\subset B(x^{j+1}_l,2A_0r^{j+1}_l)$.
Moreover, by this and the choice of $y^{j+1}_l$, we conclude that, for any $x\in B(x^{j+1}_l,2A_0r^{j+1}_l)$,
\begin{align*}
|\vz(x)|\ls |\phi^{j+1}_l(x)|&\ls \frac 1{\mu(B(x^{j+1}_l,r^{j+1}_l))+V(x^{j+1}_l,x)}\lf[\frac{r^{j+1}_l}
{r^{j+1}_l+d(x^{j+1}_l,x)}\r]^\eta\\
&\sim\frac{1}{\mu(B(y^{j+1}_l,r^{j+1}_l))+V(y^{j+1}_l,x)}
\lf[\frac{r^{j+1}_l}{r^{j+1}_l+d(y^{j+1}_l,x)}\r]^\eta.
\end{align*}
This shows the size condition of $\vz$.

To consider the regularity condition of $\vz$, we suppose that $x,\ x'\in X$ satisfying
$d(x,x')\le(2A_0)^{-1}[r^{j+1}_l+d(y^{j+1}_l,x)]$. Due to the size condition, we may assume
$d(x,x')\le(2A_0)^{-3}[r^{j+1}_l+d(y^{j+1}_l,x)]$. We claim that $\vz(x)-\vz(x')\neq 0$ implies that
$d(x,x^{j+1}_l)\le 96A_0^6r^{j+1}_l$.

Indeed, if $d(x,x^{j+1}_l)>96A_0^6r^{j+1}_l$, then $\vz(x)=0$.
By
$d(x^{j+1}_l,y^{j+1}_l)\le 48A_0^5r^{j+1}_l$, we find that $d(x,y^{j+1}_l)>48A_0^5r^{j+1}_l$ and hence
$d(x,x')\le (2A_0)^{-2}d(x,y^{j+1}_l)\le(2A_0)^{-1}d(x,x^{j+1}_l)$.
Consequently, $d(x',x^{j+1}_l)\ge A_0^{-1}d(x,x^{j+1}_l)-d(x,x')> 48A_0^5r^{j+1}_l$ and $\vz(x')=0$. This contradicts to $\vz(x)-\vz(x')\neq 0$.

By the above claim, $r^{j+1}_l\le 3A_0r^j_k$ and $d(y^{j+1}_l,x^{j+1}_j)\sim r^{j+1}_l$, we know that
\begin{align*}
|\vz(x)-\vz(x')|&\ls\frac{1}{\mu(B(x^{j+1}_l,r^{j+1}_l))}\lf[\phi^j_k(x)\lf|\phi^{j+1}_l(x)-\phi^{j+1}_l(x')\r|
+\lf|\phi^{j}_k(x)-\phi^{j}_k(x')\r|\phi^{j+1}_l(x')\r]\\
&\ls\frac{1}{\mu(B(x^{j+1}_l,r^{j+1}_l))}\lf\{\lf[\frac{d(x,x')}{r^{j+1}_l}\r]^\eta
+\lf[\frac{d(x,x')}{r^j_k}\r]^\eta\r\}\\
&\sim\lf[\frac{d(x,x')}{r^{j+1}_l+d(y^{j+1}_l,x)}\r]^\eta
\frac{1}{\mu(B(y^{j+1}_l,r^{j+1}_l))+V(y^{j+1}_l,x)}
\lf[\frac{r^{j+1}_l}{r^{j+1}_l+d(y^{j+1}_l,x)}\r]^\eta.
\end{align*}
Thus, we obtain $\vz\in\CG(y^{j+1}_l,r^{j+1}_l,\eta,\eta)$ and $\|\vz\|_{\CG(y^{j+1}_l,r^{j+1}_l,\eta,\eta)}\ls
1$, which further implies that $\|\vz\|_{\CG(y^{j+1}_l,r^{j+1}_l,\bz,\gz)}\ls 1$ and hence
$$
\RY_2=|\langle f,\vz\rangle|\ls f^*\lf(y^{j+1}_l\r)\ls 2^j.
$$
This finishes the proof of (i).

Next we prove (ii). If $ L^{j+1}_{k,l}\neq 0$, then the proof in (i) implies $B(x^j_k,2A_0r^j_k)\cap
B(x^{j+1}_l,2A_0r^{j+1}_l)\neq\emptyset$ and $r^{j+1}_l\le 3A_0r^j_k$.  Further, for any  $y\in B(x^{j+1}_l,2A_0r^{j+1}_l)$, we have
\begin{align*}
d\lf(y,x^{j}_k\r)&\le A_0\lf[d\lf(y,x^{j+1}_l\r)+d\lf(x^{j}_k,x^{j+1}_l\r)\r]
<2A_0^2r^{j+1}_l+A_0^2\lf(2A_0r^j_k+2A_0r^{j+1}_l\r)\\
&<6A_0^3r^{j}_k+2A_0^3r^{j}_k+6A_0^4r^{j}_k\le 14A_0^4r^{j}_k<16A_0^4r^j_k,
\end{align*}
which implies that $B(x^{j+1}_l,2A_0r^{j+1}_l)\subset B(x^j_k,16A_0^4r^j_k)\subset\Omega^j$ by
Proposition \ref{prop:ozdec}(v). Thus, for any $k\in I_j$ and $x\in X$, we find that
\begin{equation}\label{xx}
\sum_{l\in I_{j+1}}\lf| L^{j+1}_{k,l}\phi^{j+1}_l\r|\ls 2^j\chi_{B(x^j_k,16A_0^4r^j_k)}(x)
\end{equation}
and hence
$$
\sum_{k\in I_j}\sum_{l\in I_{j+1}}\lf| L^{j+1}_{k,l}\phi^{j+1}_l(x)\r|\ls 2^j
\sum_{k\in I_j}\chi_{B(x^j_k,16A_0^4r^j_k)}(x)\ls 2^j\chi_{\Omega^j}(x).
$$
Consequently,
\begin{align*}
\sum_{k\in I_j}\sum_{l\in I_{j+1}} L^{j+1}_{k,l}\phi^{j+1}_l
&=\sum_{l\in I_{j+1}}\lf(\sum_{k\in I_j} L^{j+1}_{k,l}\r)\phi^{j+1}_l\\
&=\sum_{l\in I_{j+1}}\frac{\phi^{j+1}_l}{\|\phi^{j+1}_l\|_{L^1(X)}}\int_X
\lf[f(\xi)-m^{j+1}_l\r]\phi^{j+1}_l(\xi)\sum_{k\in I_j}\phi^j_k(\xi)\,d\mu(\xi)\\
&=\sum_{l\in I_{j+1}}\frac{\phi^{j+1}_l}{\|\phi^{j+1}_l\|_{L^1(X)}}
\int_X \lf[f(\xi)-m^{j+1}_l\r]\phi^{j+1}_l(\xi)\,d\mu(\xi)\\
&=\sum_{l\in I_{j+1}}\frac{\phi^{j+1}_l}{\|\phi^{j+1}_l\|_{L^1(X)}}\int_X b^{j+1}_l(\xi)\,d\mu(\xi)=0.
\end{align*}
By the fact that
$\sum_{k\in I_j}\sum_{l\in I_{j+1}}\int_X | L^{j+1}_{k,l}\phi^{j+1}_l(\xi)|\,d\mu(\xi)\ls 2^j\mu(\Omega^j)<\fz$
and the dominated convergence theorem, we find that
$\sum_{k\in I_j}\sum_{l\in I_{j+1}} L^{j+1}_{k,l}\phi^{j+1}_l=0$ in $L^1(X)$ and hence in $(\go{\bz,\gz})'$.
This finishes the proof of Lemma \ref{lem:Psi}.
\end{proof}

Now we show another side of Theorem \ref{thm:atom}.
\begin{proof}[Proof of $H^{*,p(X)}\subset H^{p,q}_\at(X)$]
By Lemma \ref{lem:dense}, we first suppose $f\in L^2(X)\cap H^{*,p}(X)$. We may also assume
$|f(x)|\ls f^*(x)$ for any $x\in X$. We use the same notation as in Lemmas \ref{lem:funbgj} and \ref{lem:Psi}.
For any $j\in\nn$, let $h^j:=g^{j+1}-g^j=b^j-b^{j+1}$. Then $f-\sum_{j=-m}^m h^j=b^{m+1}-g^m$. For any
$m\in\zz$, by Lemma \ref{lem:funbgj}, we conclude that $\|g^{-m}\|_{L^\fz(X)}\ls 2^{-m}$. Moreover, by
\eqref{eq:bj*}, we find
that $\|(b^{m+1})^*\|_{L^p(X)}\ls\|f^*\chi_{(\Omega^{m+1})^\complement}\|_{L^p(X)}\to 0$ as $m\to\fz$. Thus,
$f=\sum_{j=-\fz}^\fz h^j$ in $(\go{\bz,\gz})'$. Besides, by the definition of $b^m_k$, we know that
$\supp b^{m+1}\subset\Omega^{m+1}$, which then implies that $\sum_{j=-\fz}^\fz h^j$ converges almost
everywhere.
Notice that, by Lemma \ref{lem:Psi}(ii), for any $j\in\zz$, we have
\begin{align}\label{xxx}
h^j&=b^j-b^{j+1}=\sum_{k\in I_j}b^j_k-\sum_{l\in I_{j+1}}b^{j+1}_l+\sum_{k\in I_j}\sum_{l\in I_{j+1}}
 L^{j+1}_{k,l}\phi^{j+1}_l\\
&=\sum_{k\in I_j}\lf[b^j_k-\sum_{l\in I_{j+1}}\lf(b^{j+1}_l\phi^j_k
-L^{j+1}_{k,l}\phi^{j+1}_l\r)\r]=:\sum_{k\in I_j} h^j_k\noz,
\end{align}
which converges in $(\go{\bz,\gz})'$ and almost everywhere. Moreover, for any $j\in\zz$ and $k\in\nn$,
\begin{align*}
h^j_k&=b^j_k-\sum_{l\in I_{j+1}}\lf(b^{j+1}_l\phi^j_k- L^{j+1}_{k,l}\phi^{j+1}_l\r)
=\lf(f-m^j_k\r)\phi^j_k-\sum_{l\in I_{j+1}}\lf[\lf(f-m^{j+1}_l\r)\phi^j_k- L^{j+1}_{k,l}\r]\phi^{j+1}_l\\
&=f\phi^j_k\chi_{(\Omega^{j+1})^\complement}-m^j_k\phi^j_k+\phi^j_k\sum_{l\in I_{j+1}} m^{j+1}_l\phi^{j+1}_l
+\sum_{l\in I_{j+1}} L^{j+1}_{k,l}\phi^{j+1}_l.
\end{align*}
The fourth term is supported on $B^j_k:=B(x^j_k,16A_0^4r^j_k)$, which is deduced from \eqref{xx}. Thus,
$\supp h^j_k\subset B^j_k$. Moreover, by Lemmas \ref{lem:funbgj}(i) and
\ref{lem:Psi}(i), we conclude that there exists a positive constant $C$, independent of $j$ and $k$, such that
$\|h^j_k\|_{L^\fz(X)}\le C2^j$. Now, let
\begin{equation}\label{x4}
\lz^j_k:=C2^j\lf[\mu\lf(B^j_k\r)\r]^{\frac 1p}\quad\textup{and}\quad a^j_k:=\lf(\lz^j_k\r)^{-1}h^j_k.
\end{equation}
Then $a^j_k$ is a $(p,\fz)$-atom supported on $B^j_k$ and $f=\sum_{j=-\fz}^\fz\sum_{k\in I_j}\lz^j_ka^j_k$
in $(\go{\bz,\gz})'$. Moreover, we have
$$
\sum_{j=-\fz}^\fz\sum_{k\in I_j}\lf|\lz^j_k\r|^p
\ls\sum_{j=-\fz}^\fz 2^{-jp}\sum_{k\in I_j}\mu\lf(B^j_k\r)\ls\sum_{j=-\fz}^\fz 2^{-jp}\mu\lf(\Omega^j\r)
\sim \lf\|f^\star\r\|_{L^p(X)}^p\sim\lf\|f^*\r\|_{L^p(X)}^p,
$$
which further implies that $\|f\|_{H^{p,\fz}_\at(X)}\ls\|f\|_{H^{*,p}(X)}$.

When $f\in H^{*,p}(X)$, using Lemma \ref{lem:dense} and a standard density argument and following the proof in
\cite[pp.\ 301--302]{MS79b}, we obtain the atomic decomposition of $f$, the details being omitted. This
finishes the proof of $H^{*,p}(X)\subset H^{p,q}_\at(X)$ and hence of Theorem \ref{thm:atom}.
\end{proof}

\begin{remark}\label{rem:r4.2}
By the argument used in the proof of $H^{*,p}(X)\subset H^{p,q}_\at(X)$, we find that, if
$f\in L^q(X)\cap H^{*,p}(X)$ with $q\in[1,\fz]$, then $f=\sum_{j=1}^\fz\sum_{k\in I_j} h^j_k$ in
$(\go{\bz,\gz})'$ and almost everywhere, where, for any $j\in\zz$ and $k\in I_j$, $h^j_k$ is as in \eqref{xxx}.
\end{remark}


\subsection{Relationship between $H^{p,q}_\at(X)$ and $H^{p,q}_\cw(X)$}\label{cw}

In this section, we consider the relationship between $H^{p,q}_\at(X)$ and $H^{p,q}_\cw(X)$.
To see this, we need the following two technical lemmas.
\begin{lemma}[{\cite[p.\ 592]{CW77}}]\label{lem:lipd}
Let $p\in(0,1)$, $q\in(p,\fz]\cap[1,\fz]$ and $a$ be a $(p,q)$-atom. Then,
for any $\vz\in\CL_{1/p-1}(X)$, $|\langle a,\vz\rangle|\le \|\vz\|_{\CL_{1/p-1}(X)}$.
\end{lemma}
\begin{lemma}\label{lem:gsublip}
Let $\bz\in(0,\eta]$ and $\gz\in(0,\fz)$. If $\vz\in\CG(\bz,\gz)$, then $\vz\in\CL_{\bz/\omega}(X)$ and there
exists a positive constant $C$, independent of $\vz$, such that
$\|\vz\|_{\CL_{\bz/\omega}(X)}\le C\|\vz\|_{\CG(\bz,\gz)}$.
\end{lemma}

\begin{proof}
Suppose that $\|\vz\|_{\CG(\bz,\gz)}\le 1$.
If $d(x,y)\le(2A_0)^{-1}[1+d(x_0,x)]$, then, by the regularity condition of $\vz$ and \eqref{eq:doub}, we have
\begin{align*}
|\vz(x)-\vz(y)|&\le \lf[\frac{d(x,y)}{1+d(x_0,x)}\r]^\bz\frac{1}{V_1(x_0)+V(x_0,x)}
\lf[\frac{1}{1+d(x_0,x)}\r]^\gz\\
&\ls\lf[\frac{\mu(B(x,d(x,y)))}{\mu(B(x,1+d(x_0,x)))}\r]^{\bz/\omega}\ls[V(x,y)]^{\bz/\omega}.
\end{align*}
If $d(x,y)>(2A_0)^{-1}[1+d(x_0,x)]$, then, from the size condition of $\vz$, we deduce that
\begin{align*}
|\vz(x)-\vz(y)|&\ls 1\sim[\mu(B(x_0,1))]^{\bz/\omega}\ls[\mu(B(x_0,1+d(x_0,x)))]^{\bz/\omega}\\
&\sim[\mu(B(x,1+d(x_0,x)))]^{\bz/\omega}\ls[V(x,y)]^{\bz/\omega}.
\end{align*}
Thus, for any $x,\ y\in X$, we always have $|\vz(x)-\vz(y)|\ls\|\vz\|_{\CG(\bz,\gz)}[V(x,y)]^{\bz/\omega}$.
This implies $\vz\in\CL_{\bz/\omega}(X)$ and $\|\vz\|_{\CL_{\bz/\omega}(X)}\ls\|\vz\|_{\CG(\bz,\gz)}$, which
completes the proof of Lemma \ref{lem:gsublip}.
\end{proof}

Now we establish the relationship between two kinds of atomic Hardy spaces.
\begin{theorem}\label{thm:cw}
Let $p\in(\om,1]$, $q\in(p,\fz]\cap[1,\fz]$ and $\bz,\ \gz\in(\omega(1/p-1),\eta)$. If regard $H^{p,q}_\at(X)$
as a subspace of $(\go{\bz,\gz})'$, then $H^{p,q}_\cw(X)=H^{p,q}_\at(X)$ with equal (quasi-)norms.
\end{theorem}

\begin{proof}
We only consider the case $p\in(\om,1)$. The proof of $p=1$ is similar and the details are omitted.

We first prove $H^{p,q}_\cw(X)\subset H^{p,q}_\at(X)$. By Lemma \ref{lem:gsublip}, we have
$\go{\bz,\gz}\subset\CG(\omega(1/p-1),\gz)\subset \CL_{1/p-1}(X)$ and hence $(\CL_{1/p-1}(X))'\subset
(\go{\bz,\gz})'$. For any $f\in H^{p,q}_\cw(X)$, by Definition \ref{def:atom}, we know that there exist
$(p,q)$-atoms $\{a_j\}_{j=1}^\fz$ and $\{\lz_j\}_{j=1}^\fz\subset\cc$ with $\sum_{j=1}^\fz|\lz_j|^p<\fz$ such
that $f=\sum_{j=1}^\fz\lz_j a_j$ in $(\CL_{1/p-1}(X))'$ and hence in $(\go{\bz,\gz})'$.
Let $g:=f|_{\go{\bz,\gz}}$. Then,  for any $\vz\in\go{\bz,\gz}\subset \CL_{1/p-1}(X)$, we have
$$
\langle g,\vz\rangle=\langle f,\vz\rangle=\sum_{j=1}^\fz\lz_j\langle a_j,\vz\rangle.
$$
Thus, $g=\sum_{j=1}^\fz\lz_j a_j$ in $(\go{\bz,\gz})'$ and
$\|g\|_{H^{p,q}_\at(X)}\le(\sum_{j=1}^\fz|\lz_j|^p)^{\frac 1p}$. If we take the infimum over all the atomic
decompositions of $f$ as above, we obtain $\|g\|_{H^{p,q}_\at(X)}\le\|f\|_{H^{p,q}_\cw(X)}$.
Thus, $H^{p,q}_\cw(X)\subset H^{p,q}_\at(X)$.

To show  $H^{p,q}_\cw(X)\supset H^{p,q}_\at(X)$, following the proof of \cite[p.\ 593, Theorem B]{CW77}, we
conclude that the dual space of
$H^{p,q}_\at(X)$ is $\CL_{1/p-1}(X)$ in the following sense: each bounded linear functional on $H^{p,q}_\at(X)$
is a mapping of the form
$$
f\mapsto\sum_{j=1}^\fz\lz_j\int_X a_j(x)g(x)\,d\mu(x),
$$
where $g\in\CL_{1/p-1}(X)$ and $f$ has an atomic decomposition
\begin{equation}\label{eq:atom}
f=\sum_{j=1}^\fz\lz_ja_j
\end{equation}
in $(\go{\bz,\gz})'$ with $(p,q)$-atoms $\{a_j\}_{j=1}^\fz$ and $\{\lz_j\}_{j=1}^\fz\subset\cc$ satisfying
$\sum_{j=1}^\fz|\lz_j|^p<\fz$. Therefore, it is reasonable to define the pair $\langle f,g\rangle$ as follows:
$$
\langle f,g\rangle:=\sum_{j=1}^\fz\lz_j\int_X a_j(x)g(x)\,d\mu(x).
$$
In this way, we find that \eqref{eq:atom} also converges in $(\CL_{1/p-1}(X))'$, and hence
$f\in H^{p,q}_\cw(X)$ and $\|f\|_{H^{p,q}_\cw(X)}\le(\sum_{j=1}^\fz|\lz_j|^p)^{\frac 1p}$. Taking the infimum
over all the atomic decompositions of $f$ as above, we obtain
$\|f\|_{H^{p,q}_\cw(X)}\le\|f\|_{H^{p,q}_\at(X)}$. Thus, $H^{p,q}_\at(X)\subset H^{p,q}_\cw(X)$, which completes
the proof of Theorem \ref{thm:cw}.
\end{proof}


\section[Littlewood-Paley function characterizations of atomic Hardy spaces]
{Littlewood-Paley function characterizations of atomic Hardy \\ spaces} \label{LP}

In this section, we consider the Littlewood-Paley function characterizations of Hardy spaces. Differently from
Sections \ref{max} and \ref{atom}, we use $(\GOO{\bz,\gz})'$ as underlying spaces to introduce Hardy spaces. Let
$p\in(\om,1]$, $\bz,\ \gz\in(\omega(1/p-1),\eta)$, $f\in(\GOO{\bz,\gz})'$ and $\{Q_k\}_{k\in\zz}$ be an
$\exp$-ATI.
For any $\thz\in(0,\fz)$, define the \emph{Lusin area function of $f$, with aperture $\thz$, $\CS_\thz(f)$,} by
setting, for any $x\in X$,
\begin{equation}\label{5.z}
\CS_\thz(f)(x):=\lf[\sum_{k=-\fz}^\fz\int_{B(x,\thz\dz^k)}|Q_kf(y)|^2\,\frac{d\mu(y)}{V_{\thz\dz^k}(x)}\r]
^{\frac 12}.
\end{equation}
In particular, when $\thz=1$, we write $\CS_\thz$ simply as $\CS$. Define the \emph{Hardy space} $H^{p}(X)$ via
the Lusin area function by setting
$$
H^p(X):=\lf\{f\in\lf(\GOO{\bz,\gz}\r)':\ \|f\|_{H^p(X)}:=\|\CS(f)\|_{L^p(X)}<\fz\r\}.
$$
In Section \ref{LP1}, we show that $H^p(X)$ is independent of the choices of $\exp$-ATIs.
In Section \ref{s5.2}, we connect $H^p(X)$ with $H^{*,p}(X)$ by considering the molecular and the atomic
characterizations of elements in $H^p(X)$.
Section \ref{LP2} deals with equivalent characterizations of $H^p(X)$ via  the \emph{Littlewood-Paley $g$-function}
\begin{equation}\label{eq:defg}
g(f)(x):=\lf[\sum_{k=-\fz}^\fz|Q_kf(x)|^2\r]^{\frac 12}
\end{equation}
and the \emph{Littlewood-Paley $g_\lz^*$-function}
\begin{equation}\label{5.y}
g_\lz^*(f)(x):=\lf\{\sum_{k=-\fz}^\fz\int_{X}|Q_kf(y)|^2\lf[\frac{\dz^k}{\dz^k+d(x,y)}\r]^\lz
\,\frac{d\mu(y)}{V_{\dz^k}(x)+V_{\dz^k}(y)}\r\}^{\frac 12}.
\end{equation}
where $f\in (\GOO{\bz,\gz})'$ with $\bz,\ \gz\in(\omega(1/p-1),\eta)$, $x\in X$ and $\lz\in(0,\fz)$.


\subsection{Independence of $\exp$-ATIs }\label{LP1}

In this section,
we show that $H^p(X)$ is independent of the choices of $\exp$-ATIs.
If $\mathcal E:=\{E_k\}_{k\in\zz}$ and $\CQ:=\{Q_k\}_{k\in\zz}$ are two $\exp$-ATIs, then we denote by
$\CS_{\mathcal E}$ and $\CS_\CQ$ the Lusin area functions via ${\mathcal E}$ and $\CQ$, respectively.

\begin{theorem}\label{thm:LA=}
Let ${\mathcal E}:=\{E_k\}_{k\in\zz}$ and $\CQ:=\{Q_k\}_{k\in\zz}$ be two {\rm $\exp$-ATIs}. Suppose that $p\in(\om,1]$
and $\bz,\ \gz\in(\omega(1/p-1),\eta)$. Then there exists a positive constant $C$ such that, for any
$f\in(\GOO{\bz,\gz})'$,
\begin{equation*}
C^{-1}\|\CS_\CQ(f)\|_{L^p(X)}\le\|\CS_{\mathcal E}(f)\|_{L^p(X)}\le C\|\CS_\CQ(f)\|_{L^p(X)}.
\end{equation*}
\end{theorem}

To show  Theorem \ref{thm:LA=}, the Fefferman-Stein vector-valued maximal inequality is necessary.

\begin{lemma}[{\cite[Theorem 1.2]{GLY09}}]\label{lem:FSin}
Suppose that $p\in(1,\fz)$ and $u\in(1,\fz]$. Then there exists a positive constant $C$ such that, for any
sequence $\{f_j\}_{j=1}^\fz$ of measurable functions,
\begin{equation*}
\lf\|\lf\{\sum_{j=1}^\fz[\CM(f_j)]^u\r\}^{\frac 1u}\r\|_{L^p(X)}\le
C\lf\|\lf(\sum_{j=1}^\fz|f_j|^u\r)^{\frac 1u}\r\|_{L^p(X)}
\end{equation*}
with the usual modification made when $u=\fz$.
\end{lemma}

\begin{proof}[Proof of Theorem \ref{thm:LA=}]
By symmetry, we only need to prove $\|\CS_{\mathcal E}(f)\|_{L^p(X)} \ls\|\CS_\CQ(f)\|_{L^p(X)}$. For any
$k\in\zz$, $f\in(\GOO{\bz,\gz})'$ with $\bz,\ \gz$ as in Theorem \ref{thm:LA=}, and $z\in X$, define
$$
m_k(f)(z):=\lf[\frac 1{V_{\dz^k}(z)}\int_{B(z,\dz^k)} |Q_kf(u)|^2\,d\mu(u)\r]^{\frac 12}.
$$

Now suppose that $l\in\zz$, $x\in X$ and $y\in B(x,\dz^l)$. By Theorem \ref{thm:hdrf}, we conclude that
$$
E_lf(y)=\sum_{k=-\fz}^\fz\sum_{\az\in\CA_k}\sum_{m=1}^{N(k,\az)}E_l\wz{Q}_k\lf(y,y_\az^{k,m}\r)
\int_{Q_\az^{k,m}}Q_kf(u)\,d\mu(u),
$$
where all the notation is as in Theorem \ref{thm:hdrf} and $\{\wz{Q}_k\}_{k=-\fz}^\fz$ satisfy the conditions
of Theorem \ref{thm:hdrf}. Notice that, if $z\in Q_\az^{k,m}$, then
$Q_\az^{k,m}\subset B(z,\dz^k)$ and $\mu(Q_\az^{k,m})\sim V_{\dz^k}(z)$. Therefore, we have
$$
\lf|\frac 1{\mu(Q_\az^{k,m})}\int_{Q_\az^{k,m}}Q_kf(u)\,d\mu(u)\r|
\ls\lf[\frac 1{V_{\dz^k}(z)}\int_{B(z,\dz^k)}|Q_kf(u)|^2\,d\mu(y)\r]^{\frac 12}\sim m_k(f)(z),
$$
which further implies that
$$
\lf|\frac 1{\mu(Q_\az^{k,m})}\int_{Q_\az^{k,m}}Q_kf(u)\,d\mu(u)\r|
\ls\inf_{z\in Q_\az^{k,m}}m_k(f)(z).
$$
Moreover, by the proof of \eqref{eq:wQ*phi}, we find that, for any fixed $\bz'\in(0,\bz)$,
\begin{align*}
\lf|E_l\wz{Q}_k\lf(y,y_\az^{k,m}\r)\r|&\ls\dz^{|k-l|\bz'}\frac 1{V_{\dz^{k\wedge l}}(y)+V(y,y_\az^{k,m})}
\lf[\frac{\dz^{k\wedge l}}{\dz^{k\wedge l}+d(y,y_\az^{k,m})}\r]^\gz\\
&\sim\dz^{|k-l|\bz'}\frac 1{V_{\dz^{k\wedge l}}(x)+V(x,y_\az^{k,m})}
\lf[\frac{\dz^{k\wedge l}}{\dz^{k\wedge l}+d(x,y_\az^{k,m})}\r]^\gz,
\end{align*}
where only the regularity condition of $\wz{Q}_k$ on the first variable is used. Therefore, by
Lemma \ref{lem:max}, for any fixed $r\in(\omega/(\omega+\gz),1]$, we have
\begin{align*}
|E_lf(y)|&\ls\sum_{k=-\fz}^\fz\dz^{|k-l|\bz'}\sum_{\az\in\CA_k}\sum_{m=1}^{N(k,\az)}\mu\lf(Q_\az^{k,m}\r)
\frac 1{V_{\dz^{k\wedge l}}(x)+V(x,y_\az^{k,m})}
\lf[\frac{\dz^{k\wedge l}}{\dz^{k\wedge l}+d(x,y_\az^{k,m})}\r]^\gz\inf_{z\in Q_\az^{k,m}}m_k(f)(z)\\
&\ls\sum_{k=-\fz}^\fz\dz^{|k-l|\bz'}\dz^{[k-(k\wedge l)]\omega(1-\frac 1r)}
\lf\{\CM\lf(\sum_{\az\in\CA_k}\sum_{m=1}^{N(k,\az)}\inf_{z\in Q_\az^{k,m}}[m_k(f)(z)]^r\chi_{Q_\az^{k,m}}\r)
(x)\r\}^{\frac 1r}.
\end{align*}
Choose $\bz'$ and $r$ such that $r\in(\omega/(\omega+\bz'),p)$. Then, by the H\"{o}lder inequality, we
conclude that
\begin{align*}
\lf[\CS_{\mathcal E}(f)(x)\r]^2&=\sum_{l=-\fz}^\fz\int_{B(x,\dz^l)}|E_lf(y)|^2\,\frac{dy}{V_{\dz^l}(x)}\\
&\ls\sum_{l=-\fz}^\fz\lf[\sum_{k=-\fz}^\fz\dz^{|k-l|\bz'}\dz^{[k-(k\wedge l)]\omega(1-\frac 1r)}
\lf\{\CM\lf(\sum_{\az\in\CA_k}\sum_{m=1}^{N(k,\az)}\inf_{z\in Q_\az^{k,m}}[m_k(f)(z)]^r\chi_{Q_\az^{k,m}}\r)
(x)\r\}^{\frac 1r}\r]^2\\
&\ls\sum_{l=-\fz}^\fz\sum_{k=-\fz}^\fz\dz^{|k-l|\bz'}\dz^{[k-(k\wedge l)]\omega(1-\frac 1r)}
\lf\{\CM\lf(\sum_{\az\in\CA_k}\sum_{m=1}^{N(k,\az)}\inf_{z\in Q_\az^{k,m}}[m_k(f)(z)]^r\chi_{Q_\az^{k,m}}\r)
(x)\r\}^{\frac 2r}\\
&\ls\sum_{k=-\fz}^\fz\lf\{\CM\lf(\sum_{\az\in\CA_k}
\sum_{m=1}^{N(k,\az)}\inf_{z\in Q_\az^{k,m}}[m_k(f)(z)]^r\chi_{Q_\az^{k,m}}\r)(x)\r\}^{\frac 2r}
\ls\sum_{k=-\fz}^\fz\lf\{\CM\lf([m_k(f)]^r\r)(x)\r\}^{\frac 2r}.
\end{align*}
Therefore, from Lemma \ref{lem:FSin}, we deduce that
\begin{equation*}
\|\CS_{\mathcal E}(f)\|_{L^p(X)}\ls\lf\|\lf(\sum_{k=-\fz}^\fz\lf\{\CM\lf([m_k(f)]^r\r)\r\}
^{\frac 2r}\r)^{\frac r2}\r\|_{L^{p/r}(X)}^{\frac 1r}
\ls\lf\|\lf\{\sum_{k=-\fz}^\fz[m_k(f)]^2\r\}^{\frac 12}\r\|_{L^p(X)}
\sim\|\CS_\CQ(f)\|_{L^p(X)}.
\end{equation*}
This finishes the proof of Theorem \ref{thm:LA=}.
\end{proof}

\subsection{Atomic characterizations of $H^p(X)$}\label{s5.2}

The main aim of this section is to obtain  the atomic characterizations of $H^p(X)$ when $p\in(\om,1]$.

For any
$p\in(\om,1]$, $q\in(p,\fz]\cap[1,\fz]$ and $\bz,\ \gz\in(\omega(1/p-1),\eta)$,  we define the
\emph{homogeneous atomic Hardy space} $\OH^{p,q}_\at(X)$ in the same way of $H^{p,q}_\at(X)$,
but with the distribution space
$(\go{\bz,\gz})'$ replaced by $(\GOO{\bz,\gz})'$.
Then the following relationship between $H^{p,q}_\at(X)$ and $\OH^{p,q}_\at(X)$ can be found in
\cite[Theorem 5.4]{GLY08}.

\begin{proposition}\label{prop:h=oh}
Suppose $p\in(\om,1]$, $\bz,\ \gz\in(\omega(1/p-1),\eta)$ and
$q\in(p,\fz]\cap[1,\fz]$. Then $\mathring{H}^{p,q}_\at(X)=H^{p,q}_\at(X)$ with equivalent (quasi)-norms.
More precisely, if $f\in H^{p,q}_\at(X)$, then the restriction of $f$ on $\GOO{\bz,\gz}$ belongs to
$\mathring{H}^{p,q}_\at(X)$; Conversely, if $f\in \OH^{p,q}_\at(X)$, then there exists a unique
$\wz{f}\in H^{p,q}_\at(X)$ such that $\wz{f}=f$ in $(\GOO{\bz,\gz})'$.
\end{proposition}

Due to the fact that the kernels $\wz Q_k$ in the homogeneous continuous Calder\'on formula in Theorem \ref{thm:hcrf}
has no compact support, we can only use  Theorem \ref{thm:hcrf} to decompose an element of $H^p(X)$ into a
linear combination of the following \emph{molecules}.

\begin{definition}\label{def:mol}
Suppose that $p\in(0,1]$, $q\in(p,\fz]\cap[1,\fz]$ and $\vec{\ez}:=\{\ez_m\}_{m=1}^\fz\subset [0,\fz)$
satisfying
\begin{equation}\label{eq:epcon}
\sum_{m=1}^\fz m[\ez_m]^p<\fz.
\end{equation}
A function $M\in L^q(X)$ is called a \emph{$(p,q,\vec{\ez})$-molecule} centered at a ball $B:=B(x_0,r_0)$
for some $x_0\in X$ and $r\in(0,\fz)$ if $m$ has the following properties:
\begin{enumerate}
\item $\|M\chi_B\|_{L^q(X)}\le [\mu(B)]^{\frac 1q-\frac 1p}$;
\item for any $m\in\nn$,
$\|M\chi_{B(x_0,\dz^{-m}r_0)\setminus B(x_0,\dz^{-m+1}r_0)}\|_{L^q(X)}\le
\ez_m[\mu(B(x_0,\dz^{-m}r_0))]^{\frac 1q-\frac 1p}$;
\item $\int_X M(x)\,d\mu(x)=0$.
\end{enumerate}
\end{definition}

By (i) and (ii) of Definition \ref{def:mol}, the
H\"{o}lder inequality, \eqref{eq:epcon} and the fact $p\in(0,1]$, we find that, if $M$ satisfies (i) and (ii)
of Definition \ref{def:mol}, then $M\in L^1(X)$ and hence Definition \ref{def:mol}(iii) makes sense.

After carefully checking the proof of \cite[Theorem 3.4]{lcfy18}, we obtain the following molecular
characterization of the atomic Hardy space $H^{p,q}_\cw(X)$ of Coifman and Weiss \cite{CW77}, the details being
omitted.

\begin{proposition}\label{thm:mol}
Suppose that $p\in(0,1]$, $q\in(p,\fz]\cap[1,\fz]$ and
$\vec{\ez}:=\{\ez_l\}_{l=1}^\fz$ satisfying \eqref{eq:epcon}. Then $f\in H^{p,q}_\cw(X)$ if and only if there
exist $(p,q,\vec{\ez})$-molecules $\{M_j\}_{j=1}^\fz$ and $\{\lz_j\}_{j=1}^\fz\subset\cc$, with
$\sum_{j=1}^\fz|\lz_j|^p<\fz$, such that
\begin{equation}\label{eq:mold}
f=\sum_{j=1}^\fz\lz_jM_j
\end{equation}
converges in $(\CL_{1/p-1}(X))'$ when $p\in(0,1)$ or in $L^1(X)$ when $p=1$. Moreover, there exists a positive
constant $C$, independent of $f$, such that, for any $f\in H^{p,q}_\cw(X)$,
$$
C^{-1}\|f\|_{H^{p,q}_\cw(X)}\le\inf \lf(\sum_{j=1}^\fz|\lz_j|^p\r)^{\frac 1p}\le C\|f\|_{H^{p,q}_\cw(X)},
$$
where the infimum is taken over all the molecular decompositions of $f$ as in \eqref{eq:mold}.
\end{proposition}

Let $p\in(\om,1]$ and $q\in(p,\fz]\cap[1,\fz]$.
By Proposition \ref{prop:h=oh},  $\OH^{p,q}_\at(X)=H^{p,q}_{\rm cw}(X)$ and the already known fact that
$H^{p,q}_{\rm cw}(X)$ is independent of the choice of $q\in(p,\fz]\cap[1,\fz]$,
we know that  $\mathring{H}^{p,q}_\at(X)=\mathring H^{p,2}_\at(X)$. With this observation, we show
$\OH^{p,q}_\at(X)\subset H^p(X)$ as follows.

\begin{proposition}\label{p1-add}
Let $p\in(\om,1]$, $\bz,\ \gz\in(\omega(1/p-1),\eta)$, $q\in(p,\fz]\cap[1,\fz]$
and $\{Q_k\}_{k\in\zz}$ be an {\rm $\exp$-ATI}. Let $\thz\in(0,\infty)$ and $\CS_\thz$ be as in \eqref{5.z}.
Then there exists a positive constant $C$, independent of $\thz$, such that, for any distribution
$f\in (\GOO{\bz,\gz})'$ belonging to  $\mathring{H}^{p,2}_\at(X)$,
\begin{equation}\label{eq-x2}
\|\CS_\thz(f)\|_{L^p(X)}\le C\max\lf\{\thz^{-\omega/2},\thz^{\omega/p}\r\}\|f\|_{\mathring{H}^{p,2}_\at(X)}.
\end{equation}
In particular, $\mathring{H}^{p,q}_\at(X)=\mathring{H}^{p,2}_\at(X)\subset H^p(X)$.
\end{proposition}

\begin{proof}
Let $\bz,\ \gz\in(\omega(1/p-1),\eta)$.
It suffices to show \eqref{eq-x2} for the case $\thz\in[1,\infty)$, because both \eqref{eq-x2} with
$\thz=1$ and $\CS_\thz (f)\ls \thz^{-\omega/2}\CS (f)$ for any $f\in(\GOO{\bz,\gz})'$  whenever $\thz\in(0,1)$
imply that \eqref{eq-x2} also holds true for any $\thz\in (0,1)$.

We start with the proof of the fact that the Littlewood-Paley $g$-function as in \eqref{eq:defg} is bounded on
$L^2(X)$. Indeed, for any $h\in L^2(X)$, we write
$$
\lf\| g(h)\r\|_{L^2(X)}^2=\sum_{k=-\fz}^\fz\int_X \lf|Q_kh(z)\r|^2\,d\mu(z)
=\sum_{k=-\fz}^\fz \lf\langle Q_k^\ast Q_k h,h \r\rangle.
$$
By Theorem \ref{thm:hcrf} and the proof of \cite[(3.2)]{HMY08}, we find that, for any fixed
$\bz'\in(0,\bz\wedge\gz)$, any $k_1,\ k_2\in\zz$ and $x,\ y\in X$, we have
\begin{equation}\label{eq:g1}
\lf|Q_{k_1}Q_{k_2}^\ast(x,y)\r|\ls\dz^{|k_1-k_2|\bz'}\frac{1}{V_{\dz^{k_1\wedge k_2}}(x)+V(x,y)}
\lf[\frac{\dz^{k_1\wedge k_2}}{\dz^{k_1\wedge k_2}+V(x,y)}\r]^\gz.
\end{equation}
Notice that, in \eqref{eq:g1}, only the regularity of $Q_k$ with respect to the second variable is used.
Thus, by Lemma \ref{lem-add}(v) and the boundedness of $\CM$ on $L^2(X)$, we conclude that, for any
$k_1,\ k_2\in\zz$,
$$
\lf\|\lf(Q_{k_1}^\ast{Q}_{k_1}\r)\lf(Q_{k_2}^\ast Q_{k_2}\r)\r\|_{L^2(X)\to L^2(X)}
\ls\lf\|Q_{k_1}Q_{k_2}^\ast\r\|_{L^2(X)\to L^2(X)}\ls\dz^{|k_1-k_2|\bz'}.
$$
Therefore, by the fact that $ Q_k^*Q_k$ is self-adjoint and the Cotlar-Stein lemma
(see \cite[pp.\ 279--280]{Stein93} and \cite[Lemma 4.5]{HLYY17}), we
obtain the boundedness of $\sum_{k=-\fz}^\fz Q_k^\ast Q_k$ on $L^2(X)$ and hence the boundedness of
$g$ on $L^2(X)$.

Suppose that $a$ is a $(p,2)$-atom supported on a ball $B:=B(x_0,r_0)$
with $x_0\in X$ and $r_0\in(0,\fz)$.
By the Fubini theorem and the boundedness of $g$ on $L^2(X)$,
we find that
$$
\|\CS_\thz(a)\|_{L^2(X)}\ls \lf\| \lf\{\sum_{k\in\zz} |Q_k a|^2\r\}^{1/2}\r\|_{L^2(X)}\sim \|g(a)\|_{L^2(X)}
\ls \|a\|_{L^2(X)}\ls [\mu(B)]^{\frac 12-\frac 1p},
$$
which further implies that
\begin{equation}\label{eq:Sa1}
\int_{B(x_0,4A_0^2\thz r_0)}[\CS_\thz(a)(x)]^p\,d\mu(x)\le\|\CS_\thz(a)\|_{L^2(X)}^p
\lf[\mu\lf(B\lf(x_0,4A_0^2\thz r_0\r)\r)\r]^{1-\frac p2}\ls\thz^{\omega(1-\frac p2)}.
\end{equation}

Let $x\notin B(x_0,4A_0^2\thz r_0)$ and $y\in B(x,\thz\dz^k)$.
Since now $\thz\in[1,\infty)$, for any $u\in B=B(x_0,r_0)$, we have
$d(u,x_0)<(4A_0^2\thz)^{-1}d(x_0,x)<(2A_0)^{-1}[\dz^k+d(x_0,y)]$ and hence
\begin{align*}
|Q_ka(y)|&=\lf|\int_X Q_k(y,u)a(u)\,d\mu(u)\r|\le
\int_{B}|Q_k(y,u)-Q_k(y,x_0)||a(u)|\,d\mu(u)\\
&\ls\int_{B}\lf[\frac{d(x_0,u)}{\dz^k+d(x_0,y)}\r]^\eta\frac 1{V_{\dz^k}(x_0)+V(x_0,y)}
\lf[\frac{\dz^k}{\dz^k+d(x_0,y)}\r]^\gz|a(u)|\,d\mu(u)\\
&\ls[\mu(B)]^{1-\frac 1p}\lf[\frac{r_0}{\dz^k+d(x_0,y)}\r]^\eta\frac 1{V_{\dz^k}(x_0)+V(x_0,y)}
\lf[\frac{\dz^k}{\dz^k+d(x_0,y)}\r]^\gz.
\end{align*}
On the one hand, if $\dz^k<(4A_0^2\thz)^{-1}d(x_0,x)$, then
$d(x_0,y)\ge(4A_0)^{-1}d(x_0,x)$ and hence
$$
|Q_ka(y)|\ls[\mu(B)]^{1-\frac 1p}\lf[\frac{r_0}{d(x_0,x)}\r]^\eta\frac 1{V(x_0,x)}
\lf[\frac{\dz^k}{d(x_0,x)}\r]^\gz,
$$
which further implies that
\begin{align*}
&\sum_{\dz^k<(4A_0^2\thz)^{-1}d(x_0,x)}\int_{d(x,y)<\thz\dz^k}|Q_ka(y)|^2\,\frac{d\mu(y)}{V_{\thz\dz^k}(x)}\\
&\quad\ls[\mu(B)]^{2-\frac 2p}\lf[\frac{r_0}{d(x_0,x)}\r]^{2\eta}\lf[\frac 1{V(x_0,x)}\r]^2
\sum_{\dz^k<(4A_0^2\thz)^{-1}d(x_0,x)}\lf[\frac{\dz^k}{d(x_0,x)}\r]^{2\gz}\\
&\quad\ls [\mu(B)]^{2-\frac 2p}\lf[\frac{r_0}{d(x_0,x)}\r]^{2\eta}\lf[\frac 1{V(x_0,x)}\r]^2.
\end{align*}
On the other hand, if $\dz^k\ge(4A_0^2\thz)^{-1}d(x_0,x)$, then $V(x_0,x)\ls \mu(B(x_0, \thz\dz^k))\ls \thz^\omega V_{\dz^k}(x_0)$ and
$$
|Q_ka(y)|\ls \thz^\omega [\mu(B)]^{1-\frac 1p}\lf(\frac{r_0}{\dz^k}\r)^\eta\frac 1{V(x_0,x)},
$$
which further implies that
\begin{align*}
\sum_{\dz^k\ge(4A_0^2\thz)^{-1}d(x_0,x)}\int_{d(x,y)<\thz \dz^k}|Q_ka(y)|^2\,\frac{d\mu(y)}{V_{\thz \dz^k}(x)}
&\ls\thz^{2\omega}[\mu(B)]^{2-\frac 2p}\lf[\frac 1{V(x_0,x)}\r]^2\sum_{\dz^k\ge(4A_0^2\thz)^{-1}d(x_0,x)}
\lf(\frac{r_0}{\dz^k}\r)^{2\eta}\\
&
\sim \thz^{2\omega+2\eta}[\mu(B)]^{2-\frac 2p}\lf[\frac{r_0}{d(x_0,x)}\r]^{2\eta}\lf[\frac 1{V(x_0,x)}\r]^2.
\end{align*}
Therefore, when $x\notin B(x_0,4A_0^2\thz r_0)$, we have
$$
\CS_\thz(a)(x)\ls \thz^{\omega+\eta}[\mu(B)]^{1-\frac 1p}\lf[\frac{r_0}{d(x_0,x)}\r]^\eta\frac 1{V(x_0,x)}.
$$
Consequently, using $p\in(\eta/(\omega+\eta),1]$, $B=B(x_0, r_0)$ and \eqref{eq:doub}, we obtain
\begin{align}\label{5.x2}
&\int_{[B(x_0,4A_0^2\thz r_0)]^\complement}[\CS_\thz(a)(x)]^p\,d\mu(x)\\
&\quad\ls\thz^{(\omega+\eta)p}[\mu(B)]^{p-1}\int_{[B(x_0,4A_0^2\thz r_0)]^\complement}
\lf[\frac{r_0}{d(x_0,x)}\r]^{p\eta}\lf[\frac 1{V(x_0,x)}\r]^p\,d\mu(x)\noz\\
&\quad\ls\thz^{p\omega}[\mu(B)]^{p-1}\sum_{j=2}^\fz 2^{-jp\eta}\int_{(2A_0)^j\thz  r_0\le d(x_0,x)<(2A_0)^{j+1}\thz r_0}
\lf[\frac 1{\mu(B(x_0,(2A_0)^j\thz r_0))}\r]^p\,d\mu(x)\noz\\
&\quad\ls\thz^\omega\sum_{j=2}^\fz 2^{-j[p\eta-(1-p)\omega]}\ls \thz^\omega.\noz
\end{align}
Combining \eqref{eq:Sa1} and \eqref{5.x2} implies that, when $\thz\in[1,\fz)$,
\begin{equation}\label{eq:Sa}
\|\CS_\thz(a)\|_{L^p(X)}\ls \thz^{\omega/p}.
\end{equation}

Let $f\in\OH^{p,2}_\at(X)$. By the definition of $\OH^{p,2}_\at(X)$, we know that, for any $\ez\in(0,\fz)$,
there exist $(p,2)$-atoms $\{a_j\}_{j=1}^\fz$ and $\{\lz_j\}_{j=1}^\fz\subset\cc$ such that
$f=\sum_{j=1}^\fz\lz_ja_j$ in $(\GOO{\bz,\gz})'$ and
$\sum_{j=1}^\fz|\lz_j|^p\le\|f\|_{\OH^{p,2}_\at(X)}^p+\ez$. By \eqref{eq:Sa} and the fact
$\CS_\thz(f)\le\sum_{j=1}^\fz|\lz_j|\CS_\thz(a_j)$,
we conclude that
\begin{align*}
\|\CS_\thz(f)\|_{L^p(X)}^p&\le\sum_{j=1}^\fz|\lz_j|^p\|\CS_\thz(a_j)\|_{L^p(X)}^p\ls\thz^{\omega}
\sum_{j=1}^\fz|\lz_j|^p
\ls\thz^{\omega}[\|f\|_{\OH^{p,2}_\at(X)}^p+\ez]\to\thz^{\omega}\|f\|_{\OH^{p,2}_\at(X)}^p
\end{align*}
as $\ez\to 0^+$. This finishes the proof of \eqref{eq-x2} and hence of Proposition \ref{p1-add}.
\end{proof}

Next, we use Proposition \ref{thm:mol} to show the following converse of Proposition \ref{p1-add}.

\begin{proposition}\label{p2-add}
Let $p\in (\om,1]$, $\bz,\ \gz\in(\omega(1/p-1),\eta)$ and $f\in(\GOO{\bz,\gz})'$ belong to $H^p(X)$.
Then there exist a sequence $\{a_j\}_{j=1}^\fz$ of $(p,2)$-atoms and
$\{\lz_j\}_{j=1}^\fz\subset\cc$ such that $f=\sum_{j=1}^\fz \lz_j a_j$ in $(\GOO{\bz,\gz})'$
and  $\sum_{j=1}^\fz|\lz_j|^p\le C \|f\|_{H^p(X)}^p$,
where $C$ is a positive constant independent of $f$. Consequently, $H^p(X)\subset \OH^{p,2}_\at(X)$.
\end{proposition}
\begin{proof}
Assume that $f\in(\GOO{\bz,\gz})'$ belongs to $H^p(X)$.
To avoid the confusion of notation, we use $\{E_k\}_{k\in\zz}$ to denote an $\exp$-ATI and then define
$\CS(f)$ as in \eqref{5.z} but with $Q_k$ therein replaced by $E_k$. Denote by ${\mathcal D}$  the set of all dyadic cubes.
For any $k\in\zz$, we define $\Omega_k:=\{x\in X:\ \CS(f)(x)>2^k\}$ and
$$
{\mathcal D}_k:=\lf\{Q\in{\mathcal D}:\ \mu(Q\cap\Omega_k)>\frac 12\mu(Q)\;\textup{and}\;
\mu(Q\cap\Omega_{k+1})\le\frac 12\mu(Q)\r\}.
$$
It is easy to see that, for any $Q\in{\mathcal D}$, there exists a unique $k\in\zz$ such that $Q\in{\mathcal D}_k$. A dyadic cube
$Q\in{\mathcal D}_k$ is called a \emph{maximal cube in ${\mathcal D}_k$} if $Q'\in{\mathcal D}$ and
$Q'\supset Q$, then
$Q'\notin{\mathcal D}_k$. Denote the set of all maximal cubes in ${\mathcal D}_k$ at level $j\in\zz$ by
$\{Q_{\tau,k}^j\}_{\tau\in I_{j,k}}$, where $I_{j,k}\subset \CA_j$ may be empty.
The center of $Q_{\tau,k}^j$ is denoted by $z_{\tau,k}^j$. Then
${\mathcal D}=\bigcup_{j,\ k\in\zz}\bigcup_{\tau\in I_{j,k}}\{Q\in{\mathcal D}_k:\ Q\subset Q_{\tau,k}^j \}$.

From now on, we adopt the notation $E_Q:=E_l$ and $\wz E_Q:=\wz E_l$ whenever $Q=Q_\az^{l+1}$ for some $l\in\zz$ and $\az\in\CA_{l+1}$.
Then, by Theorem \ref{thm:hcrf}, we find that
\begin{align}\label{eq-x4}
f(\cdot)
=\sum_{l=-\fz}^\fz \wz E_lE_lf(\cdot)
&=\sum_{l=-\fz}^\fz\sum_{\az\in\CA_{l+1}}\int_{Q_\az^{l+1}}\wz E_l(\cdot,y)E_lf(y)\,d\mu(y)\\
&=\sum_{Q\in{\mathcal D}}\int_{Q}\wz E_Q(\cdot,y)E_Qf(y)\,d\mu(y)\notag\\
&=\sum_{k=-\fz}^\fz\sum_{j=-\fz}^\fz\sum_{\tau\in I_{j,k}}
\sum_{Q\in{\mathcal D}_k,\ Q\subset Q_{\tau,k}^j}\int_{Q}\wz E_Q(\cdot,y)E_Qf(y)\,d\mu(y) \notag\\
&
=:\sum_{k=-\fz}^\fz\sum_{j=-\fz}^\fz\sum_{\tau\in I_{j,k}}\lz_{\tau,k}^j b_{\tau,k}^j(\cdot), \notag
\end{align}
where all the equalities converge in $(\GOO{\bz,\gz})'$,
\begin{equation*}
\lz_{\tau,k}^j:=\lf[\mu\lf(Q_{\tau,k}^j\r)\r]^{\frac 1p-\frac 12}
\lf[\sum_{Q\in{\mathcal D}_k,\ Q\subset Q_{\tau,k}^j}\int_{Q}|E_Qf(y)|^2\,d\mu(y)\r]^{\frac 12}
\end{equation*}
and
\begin{equation}\label{eq:defbB}
b_{\tau,k}^j(\cdot):=\frac 1{\lz_{\tau,k}^j}\sum_{Q\in{\mathcal D}_k,\ Q\subset Q_{\tau,k}^j}
\int_{Q}\wz E_Q(\cdot,y)E_Qf(y)\,d\mu(y).
\end{equation}

For any $Q\in{\mathcal D}_k$ and $Q\subset Q_{\tau,k}^j$, assume that $Q=Q_\az^{l+1}$
for some $l\in\zz$ and $\az\in\CA_{l+1}$.
Since $\dz$ is assumed to satisfy $\dz<(2A_0)^{-10}$, it then follows that $2A_0C^\natural \dz<1$ so that
 $Q=Q_\az^{l+1}\subset B(y,\dz^l)$ for any $y\in Q$.
 By this and the fact that
$\mu(Q\cap\Omega_{k+1})\le\frac 12\mu(Q)$, we obtain
$$
\mu(B(y,\dz^l)\cap[Q_{\tau,k}^j\setminus\Omega_{k+1}])\ge \mu(B(y,\dz^l)\cap[Q\setminus\Omega_{k+1}]) =
\mu(Q\setminus\Omega_{k+1}) \ge \frac 12\mu(Q)\sim  V_{\dz^l}(y).
$$
Thus, we have
\begin{align*}
&\sum_{Q\in{\mathcal D}_k,\ Q\subset Q_{\tau,k}^j}\int_{Q}|E_Qf(y)|^2\,d\mu(y)\\
&\quad\ls \sum_{l=j-1}^\fz\;\sum_{\az\in\CA_{l+1},\;{\mathcal D}_k\ni Q_\az^{l+1}\subset Q_{\tau,k}^j}\;
\int_{Q_\az^{l+1}}\frac{\mu(B(y,\dz^l)\cap(Q_{\tau,k}^j\setminus\Omega_{k+1}))}{V_{\dz^l}(y)}
|E_lf(y)|^2\,d\mu(y)\\
&\quad\ls \sum_{l=j-1}^\fz\int_{Q_{\tau,k}^j}\frac{\mu(B(y,\dz^l)\cap(Q_{\tau,k}^j\setminus\Omega_{k+1}))}{V_{\dz^l}(y)}
|E_lf(y)|^2\,d\mu(y)\\
&\quad\sim\int_X\sum_{l=j-1}^\fz\int_{B(y,\dz^l)\cap(Q_{\tau,k}^j\setminus \Omega_{k+1})}
|E_lf(y)|^2\,\frac{d\mu(x)}{V_{\dz^l}(y)}\,d\mu(y)\\
&\quad\ls\int_{Q_{\tau,k}^j\setminus\Omega_{k+1}}[\CS(f)(x)]^2\,d\mu(x)\ls 2^{2k}\mu\lf(Q_{\tau,k}^{j}\r).
\end{align*}
From  this and the fact $\mu(Q_{\tau,k}^j)<2 \mu(Q_{\tau,k}^j\cap\Omega_k)$, it follows that
\begin{align}\label{eq-add4}
\sum_{k=-\fz}^{\fz}\sum_{j=-\fz}^\fz\sum_{\tau\in I_{j,k}}\lf(\lz_{\tau,k}^j\r)^p
&\ls\sum_{k=-\fz}^\fz 2^{kp}\sum_{j=-\fz}^\fz\sum_{\tau\in I_{j,k}}\mu\lf(Q_{\tau,k}^j\r)\\
&\ls\sum_{k=-\fz}^\fz 2^{kp}\sum_{j=-\fz}^\fz\sum_{\tau\in I_{j,k}}\mu\lf(Q_{\tau,k}^j\cap\Omega_k\r)
\ls\sum_{k=-\fz}^\fz 2^{kp}\mu\lf(\Omega_k\r)\sim\|\CS(f)\|_{L^p(X)}^p.\notag
\end{align}

Choose $\gz'\in(\omega(1/p-1),\gz)$ and let $\vec\ez:=\{\dz^{m[\gz'-\omega(1/p-1)]}\}_{m\in\nn}$.
Assume for the moment that every $b_{\tau, k}^j$ as in \eqref{eq:defbB} is a $(p,2,\vec \ez)$-molecule
centered at a ball $B_{\tau,k}^j:=B(z_{\tau,k}^j,4A_0^2\dz^{j-1})$, whose proof is given in Lemma \ref{lem:bmol} below.
Further, applying Proposition \ref{thm:mol}, we conclude that
$\|b_{\tau, k}^j \|_{H_\cw^{p, 2}(X)}\ls 1$. Thus, $b_{\tau, k}^j$ can be written as a linear combination of
$(p,2)$-atoms which converges in $(\CL_{1/p-1}(X))'$ when $p\in(\omega/(\omega+\eta), 1)$ or in $L^1(X)$ when
$p=1$, and hence converges in $(\GOO{\bz,\gz})'$ because $\GOO{\bz,\gz}\subset \CL_{1/p-1}(X)$ (see Lemma
\ref{lem:gsublip}). Invoking this, \eqref{eq-x4} and \eqref{eq-add4}, we find that $f\in \OH_\at^{p,2}(X)$
and $\|f\|_{\OH_\at^{p,2}(X)}\ls \|\CS(f)\|_{L^p(X)}$. This finishes the proof of Proposition \ref{p2-add}.
\end{proof}

\begin{lemma}\label{lem:bmol}
Let all the notation be as in the proof of  Proposition \ref{p2-add}.
Then every $b_{\tau, k}^j$ as in \eqref{eq:defbB} is a harmlessly positive constant multiple of a
$(p,2,\vec \ez)$-molecule centered at the ball $B_{\tau,k}^j:=B(z_{\tau,k}^j,4A_0^2\dz^{j-1})$, where
$\vec\ez:=\{\dz^{m[\gz'-\omega(1/p-1)]}\}_{m\in\nn}$ and $\gz'\in(\omega(1/p-1),\gz)$.
\end{lemma}
\begin{proof}
Let  $b_{\tau, k}^j$ be as in \eqref{eq:defbB}. For any $h\in L^{2}(X)$ with
$\|h\|_{L^{2}(X)}\le 1$, by the Fubini theorem and the H\"{o}lder inequality, we
conclude that
\begin{align*}
&\lf|\int_X b_{\tau,k}^j(x)h(x)\,d\mu(x)\r|\\
&\quad\le\frac 1{\lz_{\tau,k}^j}\sum_{Q\in{\mathcal D}_k,\ Q\subset Q_{\tau,k}^j}\int_{Q} |E_Qf(y)|
\lf|\int_X\wz E_Q(x,y)h(x)\,d\mu(x)\r|\,d\mu(y)\\
&\quad\le\frac 1{\lz_{\tau,k}^j}\lf[\sum_{Q\in{\mathcal D}_k,\ Q\subset Q_{\tau,k}^j}
\int_Q|E_Qf(y)|^2\,d\mu(y)\r]^{\frac 12}\lf[\sum_{Q\in{\mathcal D}_k,\ Q\subset Q_{\tau,k}^j}
\int_X \lf|\wz E_Q^\ast h(y)\r|^2\,d\mu(y)\r]^{\frac 12}\\
&\quad\le\lf[\mu\lf(Q_{\tau,k}^j\r)\r]^{\frac 12-\frac 1p}
\|\wz g(h)\|_{L^2(X)},
\end{align*}
where $\wz g(h):=[\sum_{l=-\fz}^\fz |\wz E_l^*h|^2]^{1/2}$.
Noticing that the kernel of $\wz E_l^*$ has the regularity with respect to the second variable, we follow the
argument used in the beginning of the proof of Proposition \ref{p1-add} to deduce that
$\wz{g}$ is bounded on $L^2(X)$. Thus, we have
$$
\lf|\int_X b_{\tau,k}^j(x)h(x)\,d\mu(x)\r|\ls \lf[\mu\lf(Q_{\tau,k}^j\r)\r]^{\frac 12-\frac 1p}
\|h\|_{L^2(X)} \ls
\lf[\mu\lf(B_{\tau,k}^j\r)\r]^{\frac 12-\frac 1p}.
$$
Taking supremum over all $h\in L^{2}(X)$ with $\|h\|_{L^{2}(X)}\le 1$, we further find that
$$
\lf\|b_{\tau,k}^j\r\|_{L^2(X)}\ls\lf[\mu\lf(B_{\tau,k}^j\r)\r]^{\frac 12-\frac 1p}.
$$

Let $\gz'\in(\omega(1/p-1),\gz)$. Fix $m\in\nn$ and let
$R_m:= (\dz^{-m}B_{\tau,k}^j)\setminus(\dz^{-m+1}B_{\tau,k}^j)$. Then, for any $x\in R_m$, by the H\"{o}lder
inequality and the size condition of $\{\wz E_l\}_{l\in\zz}$, we conclude that
\begin{align*}
\lf|b_{\tau,k}^j(x)\r|&\le\frac 1{\lz_{\tau,k}^j}\sum_{Q\in{\mathcal D}_k,\ Q\subset Q_{\tau,k}^j}
\int_{Q}\lf|\wz E_Q(x,y)E_Qf(y)\r|\,d\mu(y)\\
&\ls\frac 1{\lz_{\tau,k}^j}\sum_{l=j-1}^\fz\sum_{\az\in\CA_{l+1},\ \mathcal D_k\ni Q_\az^{l+1}\subset
Q_{\tau,k}^j}\int_{Q_\az^{l+1}}
\frac{1}{V_{\dz^{l}}(x)+V(x,y)}\lf[\frac{\dz^{l}}{\dz^{l}+d(x,y)}\r]^{\gz}
\lf|E_lf(y)\r|\,d\mu(y)\\
&\ls\frac 1{\lz_{\tau,k}^j}\lf\{\sum_{l=j-1}^\fz\sum_{\az\in\CA_{l+1},\ Q_\az^{l+1}\subset Q_{\tau,k}^j}
\int_{Q_\az^{l+1}}\frac{1}{V_{\dz^{l}}(x)+V(x,y)}\lf[\frac{\dz^{l}}{\dz^{l}+d(x,y)}\r]^{2\gz'}
\,d\mu(y)\r\}^{\frac 12}\\
&\quad\times\lf\{\sum_{l=j-1}^\fz\sum_{\gfz{\az\in\CA_{l+1}}{\mathcal D_k\ni Q_\az^{l+1}\subset Q_{\tau,k}^j}}
\int_{Q_\az^{l+1}}\frac{1}{V_{\dz^{l}}(x)+V(x,y)}\lf[\frac{\dz^{l}}{\dz^{l}+d(x,y)}\r]^{2(\gz-\gz')}
\lf|E_lf(y)\r|^2\,d\mu(y)\r\}^{\frac 12}\\
&=:\frac 1{\lz_{\tau,k}^j}\RY(x)\RZ(x).
\end{align*}
Notice that, for any
$x\in R_m$,
we have $4A_0^2\dz^{j-m-1}\le d(x,z_{\tau,k}^j)<4A_0^2\dz^{j-m-2}$ and,
for any $y\in Q_\az^{l+1}\subset Q_{\tau,k}^j$, we have
$\dz^l+d(x,y)\sim d(x,y)\sim\dz^{-m+j}$ and hence
\begin{align*}
\RY(x)&\ls\lf[\sum_{l=j-1}^\fz\;\sum_{{\az\in\CA_{l+1}},\,{Q_\az^{l+1}\subset Q_{\tau,k}^j}}
\int_{Q_\az^{l+1}}\frac{1}{\mu(B(y, \dz^{-m+j}))}\lf(\frac{\dz^{l}}{\dz^{-m+j}}\r)^{2\gz'}
\,d\mu(y)\r]^{\frac 12}\\
&\ls\lf[\sum_{l=j-1}^\fz\lf(\frac{\dz^{l}}{\dz^{-m+j}}\r)^{2\gz'}
\int_{Q_{\tau,k}^j}\frac{1}{\mu(B(z_{\tau,k}^j, \dz^{-m+j}))}\,d\mu(y)\r]^{\frac 12}
\ls\dz^{m\gz'}\lf[\frac{\mu(B_{\tau,k}^j)}{\mu(\dz^{-m}B_{\tau,k}^j)}\r]^{\frac 12}.
\end{align*}
Thus, for any $x\in R_m$, we have
$$
\lf|b_{\tau,k}^j(x)\r|\ls\frac 1{\lz_{\tau,k}^j}
\dz^{m\gz'}\lf[\frac{\mu(B_{\tau,k}^j)}{\mu(\dz^{-m}B_{\tau,k}^j)}\r]^{\frac 12}\RZ(x),
$$
which, together with the Fubini theorem and Lemma \ref{lem-add}(ii), implies
that
\begin{align*}
\lf\|b_{\tau,k}^j\chi_{R_m}\r\|_{L^2(X)}&\ls\frac 1{\lz_{\tau,k}^j}
\dz^{m\gz'}\lf[\frac{\mu(B_{\tau,k}^j)}{\mu(\dz^{-m}B_{\tau,k}^j)}\r]^{\frac 12}
\lf\{\int_{R_m}[\RZ(x)]^2\,d\mu(x)\r\}^{\frac 12}\\
&\ls\frac 1{\lz_{\tau,k}^j}
\dz^{m\gz'}\lf[\frac{\mu(B_{\tau,k}^j)}{\mu(\dz^{-m}B_{\tau,k}^j)}\r]^{\frac 12}
\lf\{\sum_{Q\in{\mathcal D}_k,\ Q\subset Q_{\tau,k}^j}\int_Q
\lf|E_Qf(y)\r|^2\,d\mu(y)\r\}
^{\frac 12}\\
&\ls\dz^{m\gz'}\lf[\frac{\mu(B_{\tau,k}^j)}{\mu(\dz^{-m}B_{\tau,k}^j)}\r]^{\frac 12}
\lf[\mu\lf(B_{\tau,k}^j\r)\r]^{\frac 12-\frac 1p}
\ls \dz^{m[\gz'-\omega(\frac 1p-1)]}\lf[\mu\lf(\dz^{-m}B_{\tau,k}^j\r)\r]^{\frac 12-\frac 1p}.
\end{align*}
The cancelation of $b_{\tau,k}^j$ follows directly from that of $\wz{E}_l$, the details being omitted.

Letting $\ez_m:=\dz^{m[\gz'-\omega(\frac 1p-1)]}$ for any $m\in\nn$, we find that $\{\ez_m\}_{m=1}^\fz$
satisfies \eqref{eq:epcon} and  $b_{\tau,k}^j$ is a harmlessly positive constant multiple of a
$(p,2,\vec{\ez})$-molecule. This finishes the proof of Lemma \ref{lem:bmol}.
\end{proof}

Combining Propositions \ref{p1-add} and \ref{p2-add}, we immediately obtain the following main result of this
section, the details being omitted.
\begin{theorem}\label{thm:at2}
Suppose that $p\in(\om,1]$, $\bz,\ \gz\in(\omega(1/p-1),\eta)$
and $q\in(p,\fz]\cap[1,\fz]$. As subspaces of $(\GOO{\bz,\gz})'$, it holds true that $\mathring{H}^{p,q}_\at(X)=H^p(X)$ with
equivalent (quasi-)norms.
\end{theorem}


\subsection{Hardy spaces via various Littlewood-Paley functions}\label{LP2}

In this section, we characterize  Hardy spaces $H^p(X)$ via  the Lusin area functions with apertures, the
Littlewood-Paley $g$-functions and the Littlewood-Paley $g_\lz^*$-functions, respectively.

\begin{theorem}\label{thm:gl*}
Let $p\in(\om,1]$ and $\bz,\ \gz\in(\omega(1/p-1),\eta)$. Assume that $\thz\in(0,\infty)$ and
$\lz\in(\omega[1+2/p],\infty)$.
Then, for any
$f\in(\GOO{\bz,\gz})'$, it holds true that
\begin{align}\label{eq-x5}
\|f\|_{H^p(X)} \sim \|\CS_\thz (f)\|_{L^p(X)} \sim \|g_\lz^*(f)\|_{L^p(X)} \sim \|g(f)\|_{L^p(X)},
\end{align}
provided that either one in \eqref{eq-x5} is finite. Here, the positive equivalent constants in
\eqref{eq-x5} are independent of $f$.
\end{theorem}

\begin{proof}
Let $f\in(\GOO{\bz,\gz})'$ with $\bz,\ \gz\in(\omega(1/p-1),\eta)$.
With $\{Q_k\}_{k\in\zz}$ being an $\exp$-{\rm ATI}, we  define $\CS_\thz(f)$, $g_\lz^*(f)$ and $g(f)$,
respectively, as in \eqref{5.z}, \eqref{eq:defg} and \eqref{5.y}, where $\thz\in(0,\fz)$ and
$\lz\in(\omega[1+2/p],\infty)$.

By Proposition \ref{p1-add} and Theorem \ref{thm:at2}, if $f\in H^p(X)$, then
$\|\CS_\thz(f)\|_{L^p(X)}\ls\|f\|_{\OH^{p,2}_\at(X)}\sim \|f\|_{L^p(X)}$. Conversely, if
$\|\CS_\thz(f)\|_{L^p(X)}<\fz$, then we proceed as the proof of Proposition \ref{p2-add}
to deduce that $f=\sum_{j=1}^\fz\lz_ja_j$ in
$(\GOO{\bz,\gz})'$, where $\{a_j\}_{j=1}^\fz$ are
$(p,2)$-atoms  and $\{\lz_j\}_{j=1}^\fz\subset\cc$ satisfying
$\sum_{j=1}^\fz|\lz_j|^p\ls\|\CS_\thz(f)\|_{L^p(X)}^p$. Combining this with Theorem \ref{thm:at2} implies that
$$
\|f\|_{H^p(X)}=\|\CS(f)\|_{L^p(X)}\sim \|f\|_{\OH^{p,2}_\at(X)}\ls\|\CS_\thz(f)\|_{L^p(X)}.
$$
Therefore, we have $\|f\|_{H^p(X)} \sim \|\CS_\thz (f)\|_{L^p(X)}$ whenever $\|f\|_{H^p(X)}$ or $\|\CS_\thz (f)\|_{L^p(X)}$ is finite.

Noticing that
$
\CS(f)\ls g_\lz^*(f)\ls\sum_{j=1}^\fz 2^{j(\omega-\lz)/2}\CS_{2^j}(f),
$
we then apply \eqref{eq-x2} and $\lz\in(\omega[1+2/p],\infty)$ to obtain
\begin{align*}
\|\CS(f)\|_{L^p(X)}^p &\ls \|g_\lz^*(f)\|_{L^p(X)} ^p
\ls\sum_{j=1}^\fz 2^{j(\omega-\lz)p/2}\|\CS_{2^j}(f)\|_{L^p(X)}^p\\
&\ls \sum_{j=1}^\fz 2^{j(\omega-\lz)p/2}2^{j\omega} \|f\|_{\OH^{p,2}_\at(X)}^p\ls  \|f\|_{\OH^{p,2}_\at(X)}^p.
\end{align*}
Invoking Theorem \ref{thm:at2}, we then obtain $\|f\|_{H^p(X)}\sim\|g_\lz^*(f)\|_{L^p(X)}$ whenever
$\|f\|_{H^p(X)}$ or $\|g_\lz^*(f)\|_{L^p(X)}$ is finite.

If $f\in H^p(X)=\OH^{p,2}_\at(X)$, then, by following the proof of  \eqref{eq-x2}, we also obtain
$$\|g(f)\|_{L^p(X)}\ls \|f\|_{\OH^{p,2}_\at(X)} \sim \|f\|_{H^p(X)}.$$
To finish the proof of \eqref{eq-x5}, it remains to prove $\|f\|_{H^p(X)} \ls \|g(f)\|_{L^p(X)}$. Indeed,  for any $x\in X$, we have
\begin{align}\label{5.14x}
\CS(f)(x)&=\lf[\sum_{k\in\zz}\sum_{\az\in\CA_k}\sum_{m=1}^{N(k,\az)}\int_{d(x,y)<\dz^k}|Q_kf(y)|^2
\chi_{Q_\az^{k,m}}(x)\,\frac{d\mu(y)}{V_{\dz^k}(x)}\r]^{\frac 12}\\
&\ls\lf\{\sum_{k\in\zz}\sum_{\az\in\CA_k}\sum_{m=1}^{N(k,\az)}\lf[\sup_{z\in B(z_\az^{k,m},\dz^{k-1})}
|Q_kf(z)|^2\r]\chi_{Q_\az^{k,m}}(x)\r\}^{\frac 12},\noz
\end{align}
where $Q_\az^{k,m}$ is as in Section \ref{pre} and $z_\az^{k,m}$ the center of $Q_\az^{k,m}$.
With all the notation as in Theorem \ref{thm:hdrf}, we know that, for any $z\in B(z_\az^{k,m},\dz^{k-1})$,
$$
Q_kf(z)=\sum_{k'\in\zz}\sum_{\az'\in\CA_{k'}}\sum_{m'=1}^{N(k',\az')}\mu\lf(Q_{\az'}^{k',m'}\r)
Q_k\wz{Q}_{k'}\lf(z,y_{\az'}^{k',m'}\r)Q_{k'}f\lf(y_{\az'}^{k',m'}\r),
$$
where $y_{\az'}^{k',m'}$ is an arbitrary point in $Q_{\az'}^{k',m'}$. Fix $\bz'\in(0,\bz\wedge\gz)$.
Then, similarly to the proof of \eqref{eq:wQ*phi} (see also \cite[(3.2)]{HMY08}), we conclude that, for any
$z\in B(z_\az^{k,m},\dz^{k-1})$
\begin{equation}\label{5.14z}
\lf|Q_k\wz{Q}_{k'}\lf(z,y_{\az'}^{k',m'}\r)\r|\ls\dz^{|k-k'|\bz'}
\frac{1}{V_{\dz^{k\wedge k'}}(z)+V(z,y_{\az'}^{k',m'})}
\lf[\frac 1{\dz^{k\wedge k'}+d(z,y_{\az'}^{k',m'})}\r]^\gz.
\end{equation}
The variable $z$ in \eqref{5.14z} can be replaced by any $x\in Q_\az^{k,m}$, because
$\max\{d(z,x),d(z,z_\az^{k,m})\}\ls\dz^k\ls\dz^{k\wedge k'}$.
Further, from Lemma \ref{lem:max}, we deduce that, for any fixed $r\in(\om,1]$, any $k'\in\zz$
and $z\in B(z_\az^{k,m},\dz^{k-1})$,
\begin{align*}
&\lf|\sum_{\az'\in\CA_{k'}}\sum_{m'=1}^{N(k',\az')}\mu\lf(Q_{\az'}^{k',m'}\r)
Q_k\wz{Q}_{k'}\lf(z,y_{\az'}^{k',m'}\r)Q_{k'}f\lf(y_{\az'}^{k',m'}\r)\r|\\
&\quad\ls\dz^{(k\wedge k'-k)\omega(\frac 1r-1)}
\lf[\CM\lf(\sum_{\az'\in\CA_{k'}}\lf|Q_{k'}f\lf(y_{\az'}^{k',m'}\r)\r|^r
\chi_{Q_{\az'}^{k',m'}}\r)(x)\r]^{\frac 1r}
\end{align*}
and hence
\begin{equation}\label{5.14y}
|Q_kf(z)|\ls\sum_{k'\in\zz}\dz^{|k-k'|\bz'}\dz^{(k\wedge k'-k)\omega(\frac 1r-1)}
\lf[\CM\lf(\sum_{\az'\in\CA_{k'}}\sum_{m'=1}^{N(k',\az')}\lf|Q_{k'}f\lf(y_{\az'}^{k',m'}\r)\r|^r
\chi_{Q_{\az'}^{k',m'}}\r)(x)\r]^{\frac 1r}.
\end{equation}
Combining \eqref{5.14x} and \eqref{5.14y}, choosing $r$ and $\bz'$ such that $r\in(\omega/(\omega+\bz'),p)$ and
applying the H\"older inequality, we further conclude that, for any $x\in X$,
\begin{align*}
[\CS(f)(x)]^2&\ls\sum_{k\in\zz}\sum_{\az\in\CA_k}\sum_{m=1}^{N(k,\az)}
\lf\{\sum_{k'\in\zz}\dz^{|k-k'|\bz'}\dz^{(k\wedge k'-k)\omega(\frac 1r-1)}
\vphantom{\times\lf[\CM\lf(\sum_{\az'\in\CA_{k'}}\sum_{m'=1}^{N(k',\az')}
	\lf|Q_{k'}f\lf(y_{\az'}^{k',m'}\r)\r|^r\chi_{Q_{\az'}^{k',m'}}\r)(x)\r]^{\frac 1r}\chi_{Q_\az^{k,m}}(x)}\r.\\
&\quad\lf.\times\lf[\CM\lf(\sum_{\az'\in\CA_{k'}}\sum_{m'=1}^{N(k',\az')}
\lf|Q_{k'}f\lf(y_{\az'}^{k',m'}\r)\r|^r\chi_{Q_{\az'}^{k',m'}}\r)(x)\r]^{\frac 1r}\chi_{Q_\az^{k,m}}(x)\r\}^2\\
&\ls\sum_{k\in\zz}\sum_{\az\in\CA_k}\sum_{m=1}^{N(k,\az)}
\sum_{k'\in\zz}\dz^{|k-k'|\bz'}\dz^{(k\wedge k'-k)\omega(\frac 1r-1)}\\
&\quad\times
\lf[\CM\lf(\sum_{\az'\in\CA_{k'}}\sum_{m'=1}^{N(k',\az')}
\lf|Q_{k'}f\lf(y_{\az'}^{k',m'}\r)\r|^r\chi_{Q_{\az'}^{k',m'}}\r)(x)\r]^{\frac 2r}\chi_{Q_\az^{k,m}}(x)\\
&\ls\sum_{k\in\zz}\sum_{k'\in\zz}\dz^{|k-k'|[\bz'-\omega(\frac 1r-1)]}
\lf[\CM\lf(\sum_{\az'\in\CA_{k'}}\sum_{m'=1}^{N(k',\az')}
\lf|Q_{k'}f\lf(y_{\az'}^{k',m'}\r)\r|^r\chi_{Q_{\az'}^{k',m'}}\r)(x)\r]^{\frac 2r}\\
&\ls\sum_{k'\in\zz}\lf[\CM\lf(\sum_{\az'\in\CA_{k'}}\sum_{m'=1}^{N(k',\az')}
\lf|Q_{k'}f\lf(y_{\az'}^{k',m'}\r)\r|^r\chi_{Q_{\az'}^{k',m'}}\r)(x)\r]^{\frac 2r}.
\end{align*}
From this and Lemma \ref{lem:FSin}, we deduce that
\begin{align*}
\|f\|_{H^{p}(X)}&=\|[\CS(f)]^r\|_{L^{p/r}(X)}^{\frac 1r}
\ls\lf\|\lf\{\sum_{k'\in\zz}
\lf[\CM\lf(\sum_{\az'\in\CA_{k'}}\sum_{m'=1}^{N(k',\az')}\lf|Q_{k'}f\lf(y_{\az'}^{k',m'}\r)\r|^r
\chi_{Q_{\az'}^{k',m'}}\r)\r]^{\frac 2r}\r\}^{\frac r2}\r\|_{L^{p/r}(X)}^{\frac 1r}\\
&\ls\lf\|\lf\{\sum_{k'\in\zz}\lf[\sum_{\az'\in\CA_{k'}}\sum_{m'=1}^{N(k',\az')}
\lf|Q_{k'}f\lf(y_{\az'}^{k',m'}\r)\r|^r
\chi_{Q_{\az'}^{k',m'}}\r]^{\frac 2r}\r\}^{\frac r2}\r\|_{L^{p/r}(X)}^{\frac 1r}\\
&\sim\lf\|\lf[\sum_{k'\in\zz}\sum_{\az'\in\CA_{k'}}\sum_{m'=1}^{N(k',\az')}
\lf|Q_{k'}f\lf(y_{\az'}^{k',m'}\r)\r|^2\chi_{Q_{\az'}^{k',m'}}\r]^{\frac 12}\r\|_{L^p(X)}.
\end{align*}
By this and the arbitrariness of $y_{\az'}^{k',m'}$, we finally conclude that
$$
\|f\|_{H^{p}(X)}\ls\lf\|\lf[\sum_{k'\in\zz}\sum_{\az'\in\CA_{k'}}\sum_{m'=1}^{N(k',\az')}
\inf_{z\in Q_{\az'}^{k',m'}}\lf|Q_{k'}f(z)\r|^2\chi_{Q_{\az'}^{k',m'}}\r]^{\frac 12}\r\|_{L^p(X)}
\ls\|g(f)\|_{L^p(X)}.
$$
This finishes the proof of $\|f\|_{H^p(X)}\ls \|g(f)\|_{L^p(X)}$ and hence of Theorem \ref{thm:gl*}.
\end{proof}

\begin{remark}\label{rem5.11x}
If $X$ is a homogeneous group, Folland and Stein \cite{fs82} showed that, for any given $p\in(0,2]$ and any
$f\in\mathscr S'(X)$, $\|g_\lz^*(f)\|_{L^p(X)}\ls\|\CS(f)\|_{L^p(X)}$ whenever $\lz\in(2\omega/p,\fz)$, where
$\mathscr S'(X)$ denotes the space of tempered distributions on
$X$ (see \cite[Corollary 7.4]{fs82} by observing that $\lz$ in \eqref{5.y}
be equal to $2\lz$ with $\lz$ as in
the Littlewood-Paley $g_\lz^*$-function in \cite{fs82}).
Comparing with this, the range of $\lz$ in Theorem \ref{thm:gl*}
is narrower, this is because it was proved in \cite[Theorem 7.1]{fs82} that, for any given $p\in(0,2]$, any
$\thz\in[1,\fz)$ and $f\in\mathscr S'(X)$,
\begin{equation}\label{eq:group}
\|\CS_\thz(f)\|_{L^p(X)}\ls_p \thz^{\omega(1/p-1/2)}\|\CS(f)\|_{L^p(X)}
\end{equation}
while, in the proof of Theorem \ref{thm:gl*}, we only show that \eqref{eq:group}
for an arbitrary space of homogeneous type $X$ holds true,
with $\omega(1/p-1/2)$ replaced by $\omega/p$, when $p\in (\om,1]$
and $f\in(\GOO{\bz,\gz})'$ with $\bz,\ \gz\in(\omega(1/p-1),\eta)$.
However, it is still \emph{unclear} whether or not \eqref{eq:group}
for an arbitrary space of homogeneous type $X$ (and hence Theorem \ref{thm:gl*}
with $\lz\in(2\omega/p,\omega(1+2/p)]$) holds true.
\end{remark}


\section{Wavelet characterizations of Hardy spaces}\label{wave}

In this section, we characterize the Hardy space via the wavelet orthogonal system
$\{\psi_\az^k:\ k\in\zz,\ \az\in\CG_k\}$ introduced in \cite[Theorem 7.1]{AH13}.
The sequence  $\{D_k\}_{k\in\zz}$ of operators on $L^2(X)$ associated with integral kernels
\begin{equation}\label{eq:Dk}
D_k(x,y):=\sum_{\az\in\CG_k}\psi_\az(x)\psi_\az(y),\qquad\forall\; x,\ y\in X
\end{equation}
turns out to be an $\exp$-ATI; see \cite{HLW16, HLYY17}.
Thus, all the conclusions in Section \ref{LP} hold true for $\{D_k\}_{k\in\zz}$.

For any $f\in(\GOO{\bz,\gz})'$ with $\bz,\ \gz\in(0,\eta)$, define the \emph{wavelet Littlewood-Paley
function} $S(f)$ by setting, for any $x\in X$,
\begin{equation*}
S(f)(x):=\lf\{\sum_{k\in\zz}\sum_{\az\in\CG_k}\lf[\mu\lf(Q_\az^{k+1}\r)\r]^{-1}\lf|\lf<\psi_\az^k,f\r>\r|^2
\chi_{Q_\az^{k+1}}(x)\r\}^{\frac 12}.
\end{equation*}
For any $p\in(0,\fz)$, define the corresponding \emph{wavelet Hardy space} $H^p_w(X)$ by
$$
H^p_w(X):=\lf\{f\in\lf(\GOO{\bz,\gz}\r)':\ \|f\|_{H^p_w(X)}:=\|S(f)\|_{L^p(X)}<\fz\r\}.
$$

For any $p\in(\om,\fz)$, the $L^p(X)$-norm equivalence between the wavelet Littlewood-Paley function
$S(f)$ and the Littlewood-Paley $g$-function
$g(f)$ was proved in \cite[Theorem 4.3]{HLW16} whenever $f$ is a distribution.
The proof of \cite[Theorem 4.3]{HLW16} seems \emph{problematic} because the authors therein used an unknown
fact that, when $f\in(\mathring\CG(\bz,\gz))'$ and $n\in\nn$,
\begin{align}\label{eq-x6}
\sum_{|k|\le n}\sum_{\az\in\CG_k}\lf\langle f,\psi_\az^k\r\rangle\psi_\az^k \in L^2(X).
\end{align}
Although \eqref{eq-x6} may not be true for distributions, it is obviously true when $f\in L^2(X)$.
Indeed, the argument used in the proof of \cite[Theorem 4.3]{HLW16} proves
the following result.

\begin{theorem}\label{thm:wLP}
Suppose $p\in(\om,\fz)$ and $\bz,\ \gz\in(0,\eta)$. Then there exists a positive
constant $C$ such that, for any $f\in(\GOO{\bz,\gz})'$,
\begin{equation}\label{eq:less}
\|\CG(f)\|_{L^p(X)}\le C\|S(f)\|_{L^p(X)}
\end{equation}
and, if $f\in L^2(X)$, then
\begin{equation}\label{eq:sim}
C^{-1}\|S(f)\|_{L^p(X)}\le\|\CG(f)\|_{L^p(X)}\le C\|S(f)\|_{L^p(X)}.
\end{equation}
Here and hereafter, $\CG(f)$ is defined as in \eqref{eq:defg}, but with $Q_k$ therein replaced by $D_k$ in
\eqref{eq:Dk}.
\end{theorem}

To show that \eqref{eq:sim} holds true for all distributions, we need the following basic property of
$H^p_w(X)$.

\begin{proposition}\label{prop:bwave}
Let $p\in(\om,1]$ and $\bz,\ \gz\in(\omega(1/p-1),\eta)$. Then $H^p_w(X)$ is a (quasi-) Banach space
that can be continuously embedded into $(\GOO{\bz,\gz})'$.
\end{proposition}

\begin{proof}
Assume that  $f\in(\GOO{\bz,\gz})'$ belongs to $H^p_w(X)$. By \eqref{eq:less}, Theorems \ref{thm:gl*} and \ref{thm:at2}, we have
$\|f\|_{\OH_\at^{p,2}(X)}\ls \|f\|_{H^p_w(X)}$.
Consequently,
for
any $\ez\in(0,\fz)$, there exist $(p,2)$-atoms $\{a_j\}_{j=1}^\fz$ and $\{\lz_j\}_{j=1}^\fz\subset\cc$
satisfying $(\sum_{j=1}^\fz|\lz_j|^p)^{\frac 1p}\le\|f\|_{\OH_\at^{p,2}(X)}+\ez$ such that
$f=\sum_{j=1}^\fz\lz_ja_j$ in $(\GOO{\bz,\gz})'$. Combining this with Lemmas \ref{lem:lipd} and
\ref{lem:gsublip}, we find that, for any $\vz\in\GOO{\bz,\gz}$,
\begin{align*}
|\langle f,\vz\rangle|&\le\sum_{j=1}^\fz |\lz_j||\langle a_j,\vz\rangle|
\ls\sum_{j=1}^\fz |\lz_j|\|\vz\|_{\CL_{1/p-1}(X)}\ls\|\vz\|_{\GOO{\bz,\gz}}
\lf[\sum_{j=1}^\fz |\lz_j|^p\r]^{1/p}\\
&\ls\|\vz\|_{\GOO{\bz,\gz}}[\|f\|_{H^p_w(X)}+\ez].
\end{align*}
Letting $\ez\to 0^+$, we obtain $\|f\|_{(\GOO{\bz,\gz})'}\ls \|f\|_{H^p_w(X)}$. Thus, $H^p_w(X)$ can be
continuously embedded into $(\GOO{\bz,\gz})'$.

To prove that $H^p_w(X)$ is a (quasi-)Banach space, we only prove its completeness.
Let   $\{f_n\}_{n=1}^\fz$  be a Cauchy
sequence in $H^p_w(X)$. Then  $\{f_n\}_{n=1}^\fz$  is also a Cauchy sequence in $(\GOO{\bz,\gz})'$, so  it converges to
some element $f$ in $(\GOO{\bz,\gz})'$.
For any $n\in\nn$ and $x\in X$, applying the Fatou lemma twice, we conclude that
\begin{align*}
S(f-f_n)(x)&=S\lf(\lim_{m\to\fz}[f_m-f_n]\r)(x)
=\lf[\sum_{k\in\zz}\sum_{\az\in\CG_k}\lf|\lf<\psi_\az^k,\lim_{m\to\fz}[f_m-f_n]\r>
\wz{\chi}_{Q_\az^{k+1}}(x)\r|^2\r]^{\frac 12}\\
&=\lf[\sum_{k\in\zz}\sum_{\az\in\CG_k}\lim_{m\to\fz}\lf|\lf<\psi_\az^k,f_m-f_n\r>
\wz{\chi}_{Q_\az^{k+1}}(x)\r|^2\r]^{\frac 12}\\
&\le\liminf_{m\to\fz}\lf[\sum_{k\in\zz}\sum_{\az\in\CG_k}\lf|\lf<\psi_\az^k,f_m-f_n\r>
\wz{\chi}_{Q_\az^{k+1}}(x)\r|^2\r]^{\frac 12}=\liminf_{m\to\fz} S(f_m-f_n)(x)
\end{align*}
and hence
\begin{align*}
\|f-f_n\|_{H^p_w(X)}^p&=\int_X \lf[S(f-f_n)(x)\r]^p\,d\mu(x)\\
&\le \int_X\liminf_{m\to\fz}\lf[S(f_m-f_n)(x)\r]^p\,d\mu(x)\\
&\le \liminf_{m\to\fz}\int_X\lf[S(f_m-f_n)(x)\r]^p\,d\mu(x)=\liminf_{m\to\fz}\|f_m-f_n\|_{H^p_w(X)}^p.
\end{align*}
Letting $n\to\fz$, we find that $f\in H^p_w(X)$ and $\lim_{n\to\fz}\|f-f_n\|_{H^p_w(X)}=0$. Therefore,
$H^p_w(X)$ is complete. This finishes the proof of Proposition \ref{prop:bwave}.
\end{proof}

Applying Theorem \ref{thm:wLP} and Proposition \ref{prop:bwave}, we show the following wavelet characterizations of Hardy spaces.

\begin{theorem}\label{thm:H=wH}
Suppose $p\in(\om,1]$ and $\bz,\ \gz\in(\omega(1/p-1),\eta)$. As subspaces of $(\GOO{\bz,\gz})'$,
$H^p(X)=H^p_w(X)$ with equivalent (quasi-)norms.
\end{theorem}

\begin{proof}
Due to  \eqref{eq:less}, Theorems \ref{thm:gl*} and \ref{thm:at2}, we obtain $H^p_w(X)\subset H^p(X)$
and $\|\cdot\|_{H^p(X)}\ls \|\cdot\|_{H^p_w(X)}$.

It remains to show $H^p(X)\subset H_w^p(X)$. To this end, by Theorem \ref{thm:at2}, we conclude that
$L^2(X)\cap H^p(X)$ is dense in $H^p(X)$. Thus, for any $f\in H^p(X)$, there exist
$\{f_n\}_{n=1}^\fz\subset L^2(X)\cap H^p(X)$ such that $\lim_{n\to\fz}\|f-f_n\|_{H^p(X)}=0$.
Obviously, $\{f_n\}_{n=1}^\fz$ is a Cauchy sequence of $H^p(X)$.
Noticing that $\{f_n\}_{n=1}^\fz\subset L^2(X)$, we use
\eqref{eq:sim} and Theorem \ref{thm:gl*} to conclude that
$$
\|f_m-f_n\|_{H^p_w(X)}=\|S(f_m-f_n)\|_{L^p(X)}\sim \|\CG(f_m-f_n)\|_{L^p(X)}\sim\|f_m-f_n\|_{H^p(X)}\to 0
$$
as $m,\ n\to\fz$, so that $\{f_n\}_{n=1}^\fz$ is also a Cauchy sequence of $H^p_w(X)$. By Proposition
\ref{prop:bwave}, there exists $\wz{f}\in H^p_w(X)$ such that $f_n\to \wz f$ as $n\to\fz$ in $H^p_w(X)$, also in
$(\GOO{\bz,\gz})'$. Meanwhile, $f_n\to f$ as $n\to\fz$ in $H^p(X)$, also in $(\GOO{\bz,\gz})'$. Therefore,
$\wz{f}=f$ in $(\GOO{\bz,\gz})'$ and $f\in H^p_w(X)$. Moreover,
$$
\|f\|_{H^p_w(X)}^p\le\|f-f_n\|_{H^p_w(X)}^p+\|f_n\|_{H^p_w(X)}^p\sim \|f-f_n\|_{H^p_w(X)}^p+\|f_n\|_{H^p(X)}^p
\ls\|f\|_{H^p(X)}^p
$$
when $n$ is sufficiently large. Thus, we obtain $H^p(X)\subset H^p_w(X)$ and
$\|\cdot\|_{H^p_w(X)}\ls \|\cdot\|_{H^p(X)}$. This finishes the proof of Theorem \ref{thm:H=wH}.
\end{proof}


\section{Criteria of the boundedness of sublinear operators}\label{bd}

Let $p\in(\omega/(\omega+\eta),1]$.
By the argument used in Sections \ref{max} through \ref{wave}, we conclude that
the Hardy spaces $H^{+,p}(X)$, $H^p_\thz(X)$
with $\thz\in(0,\fz)$, $H^{*,p}(X)$, $H^{p,q}_\at(X)$, $H^{p,q}_\cw(X)$, $\OH^{p,q}_\at(X)$ with
$q\in(p,\fz]\cap[1,\fz]$ and $H^p_w(X)$ are essentially the same space in the sense of equivalent (quasi-)norms.
From now on, we simply use $H^p(X)$ to denote either one of them if there is no confusion.
In this section, we give criteria of the boundedness of sublinear operators on Hardy spaces via first
establishing finite atomic characterizations of $H^p(X)$.


\subsection{Finite atomic characterizations of Hardy spaces}\label{fin}

For any $p\in(\om,1]$ and $q\in(p,\fz]\cap[1,\fz]$, we say
$f\in H^{p,q}_\fin(X)$
if there exist $N\in\nn$, a sequence $\{a_j\}_{j=1}^N$ of $(p,q)$-atoms and $\{\lz_j\}_{j=1}^N\subset\cc$ such
that
$$
f=\sum_{j=1}^N\lz_ja_j.
$$
Also, define
$$
\|f\|_{H^{p,q}_\fin(X)}:=\inf\lf(\sum_{j=1}^N\lf|\lz_j\r|^p\r)^{\frac 1p},
$$
where the infimum is taken over all the decompositions of $f$ above. It is easy to see that
$H^{p,q}_\fin(X)$  is a  dense subset of $H^{p,q}_\at(X)$ and
$\|\cdot\|_{H^{p,q}_\at(X)}\le\|\cdot\|_{H^{p,q}_\fin(X)}$. Denote by the \emph{symbol $\UC(X)$} the space of
all uniformly continuous functions on $X$, that is, a function $f\in\UC(X)$
if and only if, for any fixed $\ez\in(0,\fz)$, there exists
$\sigma\in(0,\fz)$ such that $|f(x)-f(y)|<\ez$ whenever $d(x,y)<\sigma$.
The next theorem characterizes $H^{p,q}_\at(X)$ via $H^{p,q}_\fin(X)$.

\begin{theorem}\label{thm:fin=at}
Suppose $p\in(\om,1]$. Then  the following statements hold true:
\begin{enumerate}
\item if $q\in(p,\fz)\cap[1,\fz)$, then  $\|\cdot\|_{H^{p,q}_\fin(X)}$ and  $\|\cdot\|_{H^{p,q}_\at(X)}$
are equivalent  (quasi)-norms on $H^{p,q}_\fin(X)$;
\item   $\|\cdot\|_{H^{p,\fz}_\fin(X)}$ and  $\|\cdot\|_{H^{p,\fz}_\at(X)}$ are equivalent (quasi)-norms on
$H^{p,q}_\fin(X)\cap\UC(X)$;
\item $H^{p,\fz}_\fin(X)\cap\UC(X)$ is a dense subspace of $H^{p,\fz}_\at(X)$.
\end{enumerate}
\end{theorem}

\begin{proof}
First, we prove (i).
It suffices to show that $\|f\|_{H^{p,q}_\fin(X)}\ls\|f\|_{H^{p,q}_\at}$ for any $f\in H^{p,q}_\fin(X)$ with $q\in(p,\fz)\cap[1,\fz)$. We may as well assume that
$\|f\|_{H^{*,p}(X)}=1$. Let all the notation be as in the proof that $H^{*,p}(X)\subset H^{p,q}_\at(X)$ of
Theorem \ref{thm:atom}. Then
$$
f=\sum_{j\in\zz}\sum_{k\in I_j}\lz^j_k a^j_k=\sum_{j\in\zz}\sum_{k\in I_j} h^j_k=\sum_{j\in\zz}h_j
$$
both in $(\go{\bz,\gz})'$ and almost everywhere. Here and hereafter, for any $j\in\zz$ and $k\in I_j$, the
quantities $h_j$, $h^j_k$, $\lz^j_k$ and $a^j_k$ are as in \eqref{xxx} and \eqref{x4}.
Since  $f\in H^{p,q}_\fin(X)$, it follows that
there exist $x_1\in X$ and $R\in(0,\fz)$ such that $\supp f\subset B(x_1,R)$. We claim
that there exists a positive constant $\tilde c$ such that, for any $x\notin B(x_1,16A_0^4R)$,
\begin{equation}\label{eq:clfin}
f^\star(x)\le \tilde c[\mu(B(x_1,R))]^{-\frac 1p}.
\end{equation}
We admit \eqref{eq:clfin} temporarily and use it to prove (i) and (ii). Let $j'$ be the
maximal integer such that $2^j\le \tilde c[\mu(B(x_1,R))]^{-\frac 1p}$ and define
\begin{equation}\label{h-ell}
h:=\sum_{j\le j'}\sum_{k\in I_j}\lz^j_ka^j_k\quad\textup{and}\quad
\ell:=\sum_{j>j'}\sum_{k\in I_j}\lz^j_ka^j_k
\end{equation}
In what follows, for the sake of convenience, we elide the fact whether $I_j$ or not is finite and simply write
the summation $\sum_{k\in I_j}$ in \eqref{h-ell} as $\sum_{k=1}^\fz$.
If $j>j'$, then $\Omega^j=\{x\in X:\ f^{\star}(x)>2^j\}\subset B(x_1,16A_0^4R)$, which implies that
$\supp\ell\subset B(x_1,16A_0^4R)$ because $\supp a^j_k\subset \Omega^j$. From $f=h+\ell$, it then follows that
$\supp h\subset B(x_1,16A_0^4)$.
Noticing that
$$
\|h\|_{L^\fz(X)}\le\sum_{j\le j'}\lf\|h^j\r\|_{L^\fz(X)}\ls\sum_{j\le j'}2^j\sim[\mu(B(x_1,R))]^{-\frac 1p}
$$
and $\int_X h(x)\,d\mu(x)=0$, we conclude that $h$ is a harmlessly constant multiple of a $(p,\fz)$-atom.

Next we deal with $\ell$. For any $N:=(N_1,N_2)\in\nn^2$, define
$$\ell_N:=\sum_{j=j'+1}^{N_1}\sum_{k=1}^{N_2}\lz^j_ka^j_k=\sum_{j=j'+1}^{N_1}\sum_{k=1}^{N_2}h^j_k.$$ Then
$\ell_N$ is a finite linear combination of $(p,\fz)$-atoms and
$\sum_{j=j'+1}^{N_1}\sum_{k=1}^{N_2}|\lz^j_k|^p\ls 1$. Notice that
$\supp(\ell-\ell_N)\subset B(x_1,16A_0^4R)$ and $\int_X [\ell(x)-\ell_N(x)]\,d\mu(x)=0$. It suffices to show
that $\|\ell-\ell_N\|_{L^q(X)}\to 0$ can be sufficiently small when $N_1$ and $N_2$ are big enough.
Noticing that
$
\ell=\sum_{j=N_1+1}^\fz h^j+\sum_{j=j'+1}^{N_1}\sum_{k=1}^\fz h^j_k,
$
we have
$$
\|\ell-\ell_N\|_{L^q(X)}\le\lf\|\sum_{j=N_1+1}^\fz h^j\r\|_{L^q(X)}
+\sum_{j=j'+1}^{N_1}\lf\|\sum_{k=N_2+1}^\fz h^j_k\r\|_{L^q(X)}.
$$

For any $j\in\zz$ and $k\in\nn$, we recall that $\supp h^j_k\subset B^j_k\subset\Omega^j$ and $\|h^j\|_{L^\infty(X)}\ls 2^{j}$.
By $f=\sum_{j=-\fz}^\fz h^j$ and
$\supp(\sum_{j=N_1+1}^\fz h^j)\subset\Omega^{N_1}$, we conclude that, for any $z\in\Omega^{N_1}$,
$$
\lf|\sum_{j=N_1+1}^\fz h^j(z)\r|=\lf|f(z)-\sum_{j\le N_1}h^j(z)\r|\le |f(z)|+\sum_{j\le N_1}\lf|h^j(z)\r|
\ls |f(z)|+2^{N_1}.
$$
Notice that, by \cite[Proposition 3.9]{GLY08}, there exists a constant $\wz C>1$ such that
$f^\star\le \wz C\CM(f)$. With $f_1:=f\chi_{\{x\in X:\ |f(x)|>2^{N_1-1}/\wz C\}}$ and $f_2:=f-f_1$, we have
\begin{align*}
2^{N_1 q}\mu\lf(\Omega^{N_1}\r)&\le 2^{N_1 q}\mu\lf(\lf\{x\in X:\ \wz C\CM(f)(x)>2^{N_1}\r\}\r)\\
&\le 2^{N_1 q}\mu\lf(\lf\{x\in X:\ \wz C\CM(f_1)(x)>2^{N_1-1}\r\}\r)
\ls \|f_1\|_{L^q(X)}^q\to 0
\end{align*}
as $N_1\to\fz$, because $\CM$ is bounded from $L^q(X)$ to $L^{q,\fz}(X)$ and $f\in H_\fin^{p,q}(X)\subset
L^q(X)$. Therefore,
$$
\lf\|\sum_{j=N_1+1}^\fz h^j\r\|_{L^q(X)}^q\ls\int_{\Omega^{N_1}}\lf[|f(z)|^q+2^{N_1 q}\r]\,d\mu(z)
\ls\lf\|f\chi_{\Omega^{{N_1}}}\r\|_{L^q(X)}^q+2^{N_1 q}\mu\lf(\Omega^{N_1}\r)\to 0
$$
as $N_1\to\fz$.
Then, for any  $\ez\in(0,\fz)$, we choose $N_1\in\nn$ such that $\|\sum_{j=N_1+1}^\fz h^j\|_{L^q(X)}<\ez/2$.

If we fix $N_1\in\nn$ and $N_1\ge j>j'$, then the fact $\sum_{k=1}^\fz|h^j_k|\ls 2^j\chi_{\Omega^j}\in L^q(X)$
implies that
$$
\lim_{N_2\to 0}\lf\|\sum_{k=N_2+1}^\fz h^j_k\r\|_{L^q(X)}=0.
$$
So, we further choose $N_2\in\nn$ such that $\sum_{j=j'+1}^{N_1}\|\sum_{k=N_2+1}^\fz h^j_k\|_{L^q(X)}<\ez/2$.
In this way, we have $\|\ell-\ell_N\|_{L^q(X)}<\ez$ for large $N$. Then there exist a positive constant
$C_\flat$, independent of $N$ and $\ez$, and a $(p,q)$-atom $a_{(N)}$ such that $\ell-\ell_N=C_\flat\ez
a_{(N)}$. Therefore, we obtain $\|f\|_{H^{p,q}_\fin(X)}\ls 1\sim \|f\|_{H_\at^{p,q}(X)}$ and complete the proof
of (i)  under the assumption \eqref{eq:clfin}.

To obtain (ii), we only need to prove that $\|f\|_{H^{p,\infty}_\fin(X)}\ls\|f\|_{H^{p,\infty}_\at}$ whenever $f\in H^{p,\infty}_\fin(X)\cap \UC(X)$. We may also assume that
$\|f\|_{H^{*,p}(X)}=1$. Notice that $f\in L^\fz(X)$ and $\|f^\star\|_{L^\fz(X)}\ls \|\CM(f)\|_{L^\fz(X)}\le
c_0\|f\|_{L^\fz(X)}$, where $c_0$ is a positive constant independent of $f$. Let $j''>j'$ be the largest
integer such that $2^j\le c_0\|f\|_{L^\fz(X)}$.
We write $f=h+\ell$ with $h$ as in \eqref{h-ell} but now
$\ell=\sum_{j'<j\le j''}\sum_{k=1}^\fz h^j_k$.
As in the proof of (i), we know that  $h$ is a harmlessly positive constant multiple of some $(p,\fz)$-atom.

Now we consider $\ell$. Notice
that $f\in\UC(X)$. Then, for any  $\ez\in(0,\fz)$, there exists $\sigma\in(0,\fz)$ such that $|f(x)-f(y)|\le\ez$
whenever $d(x,y)\le\sigma$. Split $\ell=\ell^\sigma_1+\ell^\sigma_2$ with
$$\ell^\sigma_1:=\sum_{(j,k)\in G_1} h^j_k=\sum_{(j,k)\in G_1}\lz^j_k a^j_k\qquad \textup{and}\qquad
\ell^\sigma_2:=\sum_{(j,k)\in G_2} h^j_k,$$ where
$$
G_1:=\{(j,k):\ 12A_0^3r^j_k\ge\sigma,\ j'<j\le j''\}\quad\textup{and}\quad
G_2:=\{(j,k):\ 12A_0^3r^j_k<\sigma,\ j'<j\le j''\}.
$$
Notice that, for any $j'<j\le j''$,
$\Omega^j$ is bounded. Thus, by Proposition \ref{prop:ozdec}(vi), we find that $G_1$ is a finite set, which
further implies that $\ell^\sigma_1$ is a finite linear combination of $(p,\fz)$-atoms and
$$
\sum_{(j,k)\in G_1}\lf|\lz^j_k\r|^p\ls 1.
$$
To consider $\ell^\sigma_2$, it is obvious that $\supp\ell^\sigma_2\subset B(x_1,16A_0^4R)$ and
$\int_X\ell^\sigma_2(x)\,d\mu(x)=0$, so it remains to estimate
$\|\ell^\sigma_2\|_{L^\fz(X)}$. For any $(j,k)\in G_2$, applying the definition of $h^j_k$ in \eqref{xxx} implies that
$$
\lf|h^j_k\r|\le \lf|b^j_k\r|+\sum_{l\in I_{j+1}}\lf|b^{j+1}_l\phi^j_k\r|+ \sum_{l\in I_{j+1}}\lf|L^{j+1}_{k,l}\phi^{j+1}_l\r|.
$$
By the definition of $b^j_k$, we have
$\supp b^j_k\subset B(x^j_k,2A_0r^j_k)$. Moreover, for any $x\in B(x^j_k,2A_0r^j_k)$,
\begin{align}\label{eq-x7}
\lf|b^j_k(x)\r|&\le\lf|f(x)-\frac 1{\|\phi^j_k\|_{L^1(X)}}\int_{B(x^j_k,2A_0r^j_k)}f(\xi)\phi^j_k(\xi)\,d\mu(\xi)\r|\\
&\le \lf|f(x)-f\lf(x^j_k\r)\r|+\frac 1{\|\phi^j_k\|_{L^1(X)}}\int_{B(x^j_k,2A_0r^j_k)}\lf|f(\xi)-f\lf(x^j_k\r)\r|
\phi^j_k(\xi)\,d\mu(\xi)\ls\ez.\notag
\end{align}
If $b^{j+1}_l\phi^j_k\neq 0$, then
$B(x^j_k,2A_0r^j_k)\cap B(x^{j+1}_l,2A_0r^{j+1}_l)\neq\emptyset$, which further implies that $r^{j+1}_l\le 6A_0^2r^j_k$.
Thus, for any $x\in B(x^{j+1}_l,2A_0r^{j+1}_l)$, we have $d(x,x^{j+1}_l)<12A_0^3r^j_k$
and hence an argument similar to  the estimation of  \eqref{eq-x7} gives
$$
\lf|b^{j+1}_l(x)\r|=\lf|f(x)-\frac 1{\|\phi^{j+1}_l\|_{L^1(X)}}
\int_{B(x^{j+1}_l,2A_0r^{j+1}_l)}f(\xi)\phi^{j+1}_l(\xi)\,d\mu(\xi)\r|\phi^{j+1}_l(x)\ls\ez\phi^{j+1}_l(x),
$$
so that
$$
\sum_{l\in I_{j+1}}\lf|b^{j+1}_l(x)\phi^j_k(x)\r|\ls\ez\phi^j_k(x)\sum_{l\in I_{j+1}}\phi^{j+1}_l(x)\sim\ez\phi^j_k(x)\ls \ez.
$$
Using the definition of $L^{j+1}_{k,l}$ and arguing similarly as  \eqref{eq-x7}, we conclude that, for any
$x\in X$,
$$
\sum_{l\in I_{j+1}}\lf| L^{j+1}_{k,l}\phi^j_k(x)\r|\ls\ez,
$$
where $L^{j+1}_{k,l}$ is as in \eqref{x1}.
Summarizing all gives $\|h^j_k\|_{L^\fz(X)}\ls\ez$. Recalling that $\supp h^j_k\subset B^j_k$ and
$\sum_{k=1}^\fz\chi_{B^j_k}\le L_0$, we obtain
$\|\ell^\sigma_{2}\|_{L^\fz(X)}\ls\ez$. Therefore, there exist a positive constant $\wz C_\flat$, independent of
$\sigma$ and $\ez$, and a $(p,\fz)$-atom $a_{(\sigma)}$ such that $\ell^\sigma_{2}=\wz C_\flat\ez
a_{(\sigma)}$. This proves that $\|f\|_{H^{p,\fz}_\fin(X)}\ls 1$ and hence finishes the
proof (ii) under the assumption \eqref{eq:clfin}.

Now we  prove \eqref{eq:clfin}. Let $x\notin B(x_1,16A_0^4R)$. Suppose that $\vz\in\go{\bz,\gz}$ with
$\|\vz\|_{\CG(x,r,\bz,\gz)}\ls 1$ for some $r\in(0,\fz)$.
First we consider the case $r\ge 4A_0^2d(x,x_1)/3$.
For any $y\in B(x,d(x,x_1))$, we have
 $\|\vz\|_{\CG(y,r,\bz,\gz)}\ls 1$, which implies that $|\langle f,\vz\rangle|\ls f^*(y)$ and hence
\begin{equation}\label{eq:fvz}
|\langle f,\vz\rangle|\ls\lf\{\frac 1{\mu(B(x,d(x,x_1)))}\int_{B(x,d(x,x_1))}\lf[f^*(y)\r]^p
\,d\mu(y)\r\}^{\frac 1p}\ls[\mu(B(x_1,R))]^{-\frac 1p}.
\end{equation}
Next we consider the case $r<4A_0^2d(x_1,x)/3$. Choose a function $\xi$ satisfying
$\chi_{B(x_1,(2A_0)^{-4}d(x_1,x))}\le\xi\le\chi_{B(x_1,(2A_0)^{-3}d(x_1,x))}$ and
$\|\xi\|_{\dot{C}^\eta(X)}\ls [d(x_1,x)]^{-\eta}$. Since $\supp f\subset B(x_1,R)$, it follows that $f\xi=f$.
Let $\wz{\vz}:=\vz\xi$. For any $y\in B(x,d(x,x_1))$, assuming for the moment that
\begin{equation}\label{eq:clfin2}
\lf\|\wz{\vz}\r\|_{\CG(y,r,\bz,\gz)}\ls 1,
\end{equation}
 we obtain
$$
|\langle f,\vz\rangle|=\lf|\int_X f(z)\vz(z)\,d\mu(z)\r|=\lf|\int_X f(z)\xi(z)\vz(z)\,d\mu(z)\r|
=\lf|\langle f,\wz{\vz}\rangle\r|\ls f^*(y),
$$
which implies that \eqref{eq:fvz}  remains true in this case. Therefore, by the arbitrariness of $\vz$ and the fact that
$f^*\sim f^\star$, we obtain \eqref{eq:clfin}.

Now we fix $y\in B(x_1,d(x_1,x))$ and prove \eqref{eq:clfin2}. First we consider the size condition. Indeed, if $\wz{\vz}(z)\neq 0$, then
$d(z,x_1)<(2A_0)^{-3}d(x_1,x)$ and hence
$d(z,y)<(16A_0^2/7)d(x,z)$, which implies that
$$
\lf|\wz{\vz}(z)\r|\le|\vz(z)|\le\frac 1{V_r(x)+V(x,z)}\lf[\frac{r}{r+d(x,z)}\r]^\gz
\sim\frac 1{V_r(y)+V(y,z)}\lf[\frac{r}{r+d(y,z)}\r]^\gz.
$$
To consider the regularity condition of $\wz\vz$, we may assume that $d(z,z')\le(2A_0)^{-10}[r+d(y,z)]$
due to the size condition. For the case $d(z,x_1)>(2A_0)^{-1}d(x_1,x)$, we have $\wz{\vz}(z)=0$ and, by
$y\in B(x_1,d(x_1,x))$  and $r<4A_0^2d(x_1,x)/3$, we further obtain
\begin{align*}
d(z,z')\le(2A_0)^{-10}[r+d(y,z)]&\le (2A_0)^{-10}[r+A_0d(y,x_1)+A_0d(x_1,z)]\\
&\le(2A_0)^{-10}[4A_0^2d(x_1,x)+A_0d(x_1,z)]\le (2A_0)^{-2}d(x_1,z),
\end{align*}
which further implies that $d(z',x_1)\ge \frac 1{A_0}d(x_1,z)-d(z,z')\ge(2A_0)^{-2}d(x_1,x)$ and hence $\wz\vz(z')=0$.
So we only need to consider the case
$d(z,x_1)\le (2A_0)^{-1}d(x_1,x)$. Then we have  $(2A_0)^{-1} d(x_1, x)\le d(z,x)\le 2A_0 d(x_1, x)$ and
$$
d(y,z)\le A_0^2[d(y,x_1)+d(x_1,x)+d(x,z)]\le 2A_0^2d(x_1,x)+A_0^2d(x,z)\le(2A_0)^3d(x,z),
$$
which implies that
$d(z,z')\le(2A_0)^{-1}[r+d(x,z)]$ and $r+d(y,z)\ls \min\{ r+d(x,z), r+d(x, z'), d(x_1,x)\}$.
Therefore, by the regularity of $\vz$ and the definition of $\xi$, we conclude that
\begin{align*}
\lf|\wz{\vz}(z)-\wz{\vz}(z')\r|&\le\xi(z)|\vz(z)-\vz(z')|+|\vz(z')||\xi(z)-\xi(z')|\\
&\ls\lf[\frac{d(z,z')}{r+d(x,z)}\r]^\bz\frac 1{V_r(x)+V(x,z)}\lf[\frac r{r+d(x,z)}\r]^\gz\\
&\quad+\frac 1{V_r(x)+V(x,z')}\lf[\frac r{r+d(x,z')}\r]^\gz\lf[\frac{d(z,z')}{d(x_1,x)}\r]^\bz\\
&\ls\lf[\frac{d(z,z')}{r+d(y,z)}\r]^\bz\frac 1{V_r(y)+V(y,z)}\lf[\frac r{r+d(y,z)}\r]^\gz.
\end{align*}
This proves \eqref{eq:clfin2} and hence finishes the proofs of (i) and (ii).

Now we prove (iii). According to \cite[pp.\ 3347--3348]{HHL16} (see also \cite[Theorem 2.6]{HMY08}), there exists a sequence $\{S_k\}_{k\in\zz}$ of bounded operators on $L^2(X)$ with their kernels satisfying the following conditions:
\begin{enumerate}
\item $S_k(x,y)=0$ if $d(x,y)\ge C_\sharp\dz^k$ and, for any $x,\ y\in X$,
$$
|S_k(x,y)|\ls \frac 1{V_{\dz^k}(x)+V_{\dz^k}(y)},
$$
where $C_\sharp$ is a fixed positive constant greater than $1$;
\item for any $x,\ x',\ y\in X$ with $d(x,x')\le C_\sharp\dz^k$,
$$
|S_k(x,y)-S_k(x',y)|+|S_k(y,x)-S_k(y,x')|
\ls \lf[\frac{d(x,x')}{\dz^k}\r]^\thz\frac 1{V_{\dz^k}(x)+V_{\dz^k}(y)},
$$
where $\thz$ is as in \cite[Theorem 2.4]{HHL16};
\item for any $x\in X$,
$
\int_X S_k(x,y)\,d\mu(y)=1=\int_X S_k(y,x)\,d\mu(y).
$
\end{enumerate}
For any $g\in \bigcup_{p\in[1,\fz]} L^p(X)$ and $x\in X$, define
$$
S_kg(x):=\int_X S_k(x,y)g(y)\,d\mu(y).
$$
Then, for any $(p,\fz)$-atom $a$  supported on $B(z,r)$ with $z\in X$ and $r\in(0,\fz)$,
we observe that $S_ka$ satisfies the following properties:
\begin{enumerate}
\item[(a)] $\|S_ka\|_{L^\infty(X)}\ls\|a\|_{L^\infty(X)}$ and $\lim_{k\to\fz}\|S_ka-a\|_{L^2(X)}=0$;
\item[(b)] when $k$ is sufficiently large, $\supp S_k(a)\subset B(z,2A_0r)$;
\item[(c)] $\int_X S_ka(x)\,d\mu(x)=0$;
\item[(d)] $S_ka\in\UC(X)$.
\end{enumerate}
Consequently,  $S_ka$ is a harmlessly constant multiple of a $(p,\infty)$-atom and hence of a $(p,2)$-atom.
Thus, $\|S_ka-a\|_{H^{p,\infty}_\at(X)}\sim \|S_ka-a\|_{H^{p,2}_\at(X)}\to 0$ as $k\to\infty$.
For any $f\in H^{p,\fz}_\at(X)$, there exists a sequence  $\{f_n\}_{n\in\nn}\subset H^{p,\fz}_\fin(X)$ such
that $\lim_{n\to\fz}\|f_n-f\|_{H^{p,q}_\at(X)}=0$.
Then, for any $n\in\nn$, by the above (a) through (d), we find that $S_k(f_n)\in H^{p,\fz}_\fin(X)\cap \UC(X)$
and $\lim_{k\to\fz}\|S_kf_n-f_n\|_{H^{p,\infty}_\at(X)}=0$.
This proves that $\|S_k f_n-f\|_{H^{p,\infty}_\at(X)}\to 0$ as $n,\ k\to\infty$, which completes the proof of
(iii) and hence of Theorem \ref{thm:fin=at}.
\end{proof}


\subsection{Criteria of the boundedness of sublinear operators on Hardy spaces}\label{bd2}

In this section, applying the finite atomic characterizations of Hardy spaces,
we obtain two criteria on the boundedness of sublinear operators on Hardy spaces.

Recall that a complete vector space $\CB$ is called a \emph{quasi-Banach space} if its quasi-norm
$\|\cdot\|_\CB$ satisfies the following condition:
\begin{enumerate}
\item for any $f\in\CB$, $\|f\|_\CB=0$ if and only if $f$ is the zero element in $\CB$;
\item for any $\lz\in\cc$ and $f\in\CB$, $\|\lz f\|_\CB=|\lz|\|f\|_\CB$;
\item there exists $C\in[1,\fz)$ such that, for any $f,\ g\in\CB$,
$\|f+g\|_\CB\le C(\|f\|_\CB+\|g\|_\CB)$.
\end{enumerate}
Next we recall the definition of $r$-quasi-Banach spaces (see, for example, \cite{Ky14,ylk17,yz09,yz08,GLY08}).
\begin{definition}\label{dqba}
Suppose that $r\in(0,1]$ and $\CB_r$ is a quasi-Banach space with its quasi-norm $\|\cdot\|_{\CB_r}$. The
space $\CB_r$ is called an \emph{$r$-quasi-Banach space} if there exists $\kz\in[1,\fz)$ such that, for any $m\in\nn$ and
$\{f_j\}_{j=1}^m\subset\CB_r$,
$$
\lf\|\sum_{j=1}^m f_j\r\|_{\CB_r}^r\le\kappa \sum_{j=1}^m \|f_j\|_{\CB_r}^r.
$$
\end{definition}

Obviously, when $p\in(0,1]$, $L^p(X)$ and $H^{*,p}(X)$ are $p$-quasi-Banach-spaces. Let $\CY$ be a linear space
and $\CB_r$ is an $r$-quasi-Banach space with $r\in(0,1]$.
An operator $T:\ \CY\to\CB_r$ is said to be \emph{$\CB_r$-sublinear} if there exists a positive constant
$\kappa\in[1,\fz)$ such that
\begin{enumerate}
\item for any $f,\ g\in\CY$, $\|T(f)-T(g)\|_{\CB_r}\le\kappa\|T(f-g)\|_{\CB_r}$;
\item for any $m\in\nn$, $\{f_j\}_{j=1}^m\subset\CY$ and $\{\lz_j\}_{j=1}^m\subset\cc$,
$$
\lf\|T\lf(\sum_{j=1}^m\lz_jf_j\r)\r\|_{\CB_r}^r\le\kappa \sum_{j=1}^m|\lz_j|^r\|T(f_j)\|_{\CB_r}^r.
$$
\end{enumerate}
(see, for example, \cite[Definition 2.5]{Ky14}, \cite[Definition 1.6.7]{ylk17}, \cite[Remark 1.1(3)]{yz09},
\cite[Definition 1.6]{yz08} and \cite[Definition 5.8]{GLY08}).

The next theorem gives us a criteria for $\CB_r$-sublinear operators that can be extended to bounded
$\CB_r$-sublinear operators from Hardy spaces to $\CB_r$. It can be proved by following the proof of
\cite[Theorem 5.9]{GLY08} with slight modifications, the details being omitted.

\begin{theorem}\label{thm:sub}
Let $p\in(\om,1]$ and $r\in[p,1]$. Suppose that $\CB_r$ is an $r$-quasi-Banach space and either of the following
holds true:
\begin{enumerate}
\item $q\in(p,\fz)\cap[1,\fz)$ and $T:\ H^{p,q}_\fin(X)\to\CB_r$ is a $\CB_r$-sublinear operator with
$$
\sup\{\|T(a)\|_{\CB_r}:\ a \textit{ is any } (p,q)\textit{-atom}\}<\fz;
$$
\item $T:\ H^{p,\fz}_\fin(X)\cap\UC(X)\to\CB_r$ is a $\CB_r$-sublinear operator with
$$
\sup\{\|T(a)\|_{\CB_r}:\ a \textit{ is any } (p,\fz)\textit{-atom}\}<\fz.
$$
\end{enumerate}
Then $T$ can be uniquely extended to a bounded $\CB_r$-sublinear operator from $H_\at^{p,q}(X)$ to $\CB_r$.
\end{theorem}


\bigskip

\noindent Ziyi He, Dachun Yang (Corresponding author) and Wen Yuan

\medskip

\noindent Laboratory of Mathematics and Complex Systems (Ministry of Education of China),
School of Mathematical Sciences, Beijing Normal University, Beijing 100875, People's Republic of China

\smallskip

\noindent{\it E-mails:} \texttt{ziyihe@mail.bnu.edu.cn} (Z. He)

\noindent\phantom{{\it E-mails:} }\texttt{dcyang@bnu.edu.cn} (D. Yang)

\noindent\phantom{{\it E-mails:} }\texttt{wenyuan@bnu.edu.cn} (W. Yuan)

\bigskip

\noindent Yongsheng Han

\medskip

\noindent Department of Mathematics, Auburn University, Auburn, AL 36849-5310, USA

\smallskip

\noindent{\it E-mail:} \texttt{hanyong@auburn.edu}

\bigskip

\noindent Ji Li

\medskip

\noindent Department of Mathematics, Macquarie University, Sydney, NSW 2109, Australia

\smallskip

\noindent{\it E-mail:} \texttt{ji.li@mq.edu.au}

\bigskip

\noindent Liguang Liu

\medskip

\noindent Department of Mathematics, School of Information, Renmin University of China, Beijing 100872,
People's Republic of China

\smallskip

\noindent{\it E-mail:} \texttt{liuliguang@ruc.edu.cn}
\end{document}